\documentclass[a4paper,twoside]{article}
\usepackage{amssymb,amsfonts,amsmath,amsbsy,theorem,amscd}
\usepackage{dina42e}
\usepackage{epsfig,psfrag,graphicx,xcolor}
\usepackage{stmaryrd}
\usepackage{extarrows}
\usepackage{bookmark,hyperref}
\usepackage{csquotes}

\usepackage{color}

\definecolor{CMUrot}{RGB}{128,18,18}
\definecolor{Gold}{RGB}{238,180,34}
\allowdisplaybreaks[4]

\newcommand{\ol}[1]{\overline{#1}}

\numberwithin{equation}{section}

\newcommand{\R}{\ensuremath{\mathbb{R}}}

\newcommand{\Om}{\ensuremath{\Omega}}

\newcommand{\N}{\ensuremath{\mathbb{N}}}

\newcommand{\Z}{\ensuremath{\mathbb{Z}}}

\newcommand{\HS}{\ensuremath{\mathcal{H}}}
\newcommand{\SD}{\ensuremath{\mathcal{S}}}

\newcommand{\sd}{{\rm d}}
\newcommand{\supp}{\operatorname{supp}}
\newcommand{\eps}{\ensuremath{\varepsilon}}
\newcommand{\weight}[1]{\langle #1\rangle}

\newcommand{\Div}{\operatorname{div}}

\newcommand{\T}{\ensuremath{\mathbb{T}}}

\newcommand{\no}{\mathbf{n}}

\newcommand{\tn}[1]{\mathbf{#1}}
\def\nn{\mathbf{n}}
\let\e=\varepsilon
\newcommand{\ve}{\mathbf{v}}
\newcommand{\we}{\mathbf{w}}
\newcommand{\ue}{\mathbf{u}}

\newcommand{\btau}{{\boldsymbol{\tau}}}
\newtheorem{thm}{THEOREM}[section]
\newtheorem{cor}[thm]{Corollary}
\newtheorem{lem}[thm]{Lemma}
\newtheorem{defn}[thm]{Definition}
\newtheorem{theorem}[thm]{Theorem}
\newtheorem{prop}[thm]{Proposition}

\newtheorem{claim*}{Claim}

\theorembodyfont{\rmfamily}
\newtheorem{rem}[thm]{Remark}

\newenvironment{proof*}[1]{{\bf Proof
#1:}}{\hspace*{\fill}\rule{1.2ex}{1.2ex}\\ }
\newenvironment{proof}{{\bf
Proof:\,}}{\hspace*{\fill}\rule{1.2ex}{1.2ex}\\ }

\global\long\def\cte{\hat{c}_{\varepsilon}}

\newcommand{\p}{\partial}
\newcommand{\G}{\Gamma}
\def\({\left(}
\def\){\right)}
\renewcommand{\O}{\Omega}

\newcommand{\tc}{\hat{c}}

\newcommand{\tr}{\hat{r}}

\newcommand{\order}{N}
\newcommand{\zg}{\zeta\circ d_\Gamma}

\begin{document}
\begin{titlepage}
\title{Sharp Interface Limit for a   Navier-Stokes/Allen-Cahn System in the Case of a Vanishing Mobility}
\author{Helmut Abels, Mingwen Fei, and Maximilian Moser}
\end{titlepage}

\maketitle
\abstract{
	 We consider the sharp interface limit of a Navier-Stokes/Allen Cahn equation in a bounded smooth domain in two space dimensions, in the case of vanishing mobility $m_\eps=\sqrt{\eps}$, where the small parameter $\eps>0$ related to the thickness of the diffuse interface is sent to zero. For well-prepared initial data and sufficiently small times, we rigorously prove convergence to the classical two-phase Navier-Stokes system with surface tension. The idea of the proof is to use asymptotic expansions to construct an approximate solution and to estimate the difference of the exact and approximate solutions with a spectral estimate for the (at the approximate solution) linearized Allen-Cahn operator. In the calculations we use a fractional order ansatz and new ansatz terms in higher orders leading to a suitable $\eps$-scaled and coupled model problem. Moreover, we apply the novel idea of introducing $\eps$-dependent coordinates.}


       {\small\noindent
}

       
       \section{Introduction and Main Result}\label{sec:Introduction}

Two-phase flows of macroscopically immiscible fluids is an important research area with many applications. There are two important model categories: sharp interface models and diffuse interface models. For sharp interface models the interface separating the fluids is assumed to be a hypersurface. These models usually consist of an evolution law for the hypersurface, coupled to equations in the bulk domains and on the interface. Solutions often develop singularities in finite time, in particular when the interface changes its topology. In contrast, diffuse interface models use a typically smooth order parameter (e.g. the density or volume fraction of the two fluids) that distinguishes the bulk domains and inbetween typically has steep gradients in a small transition zone (also called diffuse interface), which is proportional to a small parameter, e.g.~$\eps>0$. In applications the diffuse interface can be interpreted as microscopically small mixing region of the fluids. Quantities defined on the hypersurface in sharp interface models typically have a diffuse analogue that is defined in the diffuse interface. An important example is the relation of surface tension and capillary stress tensor, see Anderson, McFadden, Wheeler \cite{DiffIntModels}. Diffuse interface models may be more suited to describe phenomena acting on length scales related to the interface thickness, e.g. interface thicking phenomena, complicated contact angle behaviour and topology changes, cf. \cite{DiffIntModels}. Moreover, topology changes typically are no problem from an analytical or numerical point of view in contrast to sharp interface models. However, both model types are usually derived from physical principles or observations and can be used to model the same situations in applications. This motivates to study the connection between diffuse and sharp interface models by sending the small parameter $\eps$ (related to the thickness of the diffuse interface) to zero. Such limits are known as \enquote{sharp interface limits}.

Let $T_0>0$, $\Om\subseteq\R^2$ be a bounded smooth domain and $\eps>0$ be small. For $\ve_\eps\colon\overline{\Om}\times[0,T_0]\rightarrow\R^2$, $p_\eps, c_\eps\colon\overline{\Om}\times [0,T_0]\rightarrow\R$ we consider the following Navier-Stokes/Allen-Cahn system for small $\eps>0$: 
\begin{alignat}{2}\label{eq:NSAC1}
	\partial_t \ve_\eps +\ve_\eps\cdot \nabla \ve_\eps-\Div(2\nu(c_\eps)D\ve_\eps)  +\nabla p_\eps & = -\eps \Div (\nabla c_\eps \otimes \nabla c_\eps)&\quad & \text{in}\ \Omega\times(0,T_0),\\\label{eq:NSAC2}
	\Div \ve_\eps& = 0&\quad & \text{in}\ \Omega\times(0,T_0),\\\label{eq:NSAC3}
	\partial_t c_\eps +\ve_\eps\cdot \nabla c_\eps & =m_\eps\left[\Delta c_\eps - \frac{1}{\eps^{2}} f'(c_\eps)\right] &\quad & \text{in}\ \Omega\times(0,T_0),\\
	\label{eq:NSAC4}
	(\ve_\eps,c_\eps)|_{\partial\Omega}&= (0,-1)&\quad & \text{on }\partial\Omega\times(0,T_0),\\
	\label{eq:NSAC5}
	(\ve_\eps,c_\eps) |_{t=0}& = (\ve_{0,\eps},c_{0,\eps})&\quad& \text{in }\Omega,
\end{alignat}
where $\ve_\eps, p_\eps$ have the interpretation of a mean fluid velocity and pressure, respectively, and $c_\eps$ has the role of an order parameter distinguishing two components of a fluid mixture. Moreover, $\nu\colon\R\rightarrow(0,\infty)$ is a smooth concentration-dependent viscosity,  $m_\eps:=\sqrt{\eps}$ is the mobility and $f\colon\R\rightarrow[0,\infty)$ is a suitable smooth double-well potential with wells of equal depth, e.g.~$f(c)=\frac{1}{8}(c^2-1)^2$, specified below. For simplicity in the following analysis we assume that $\nu'\colon \R\to \R$ is even. Furthermore, $D\ve_\eps=\frac 12(\nabla\ve_\eps+(\nabla\ve_\eps)^T)$ is the symmetrized gradient and the operators $\nabla$, $\Delta$ and $\Div$ are defined to act on spatial variables only. Finally, note that $\nabla c_\eps \otimes \nabla c_\eps$ is a contribution to the stress tensor that represents capillary stresses due to surface tension effects in the (typically small) mixing region. The above model was introduced by Liu and Shen in \cite{LiuShenModelH} for constant viscosity along with a Navier-Stokes/Cahn-Hilliard variant in order to describe two-phase incompressible Newtonian fluids with the diffuse interface approach. The model was later derived in a thermodynamically consistent way by Jiang, Li, Liu~\cite{TwoPhaseVariableDensityJiangEtAl} via an energetic variational approach including the case of different densities. Moreover, they showed global existence of weak solutions in 3D and global well-posedness and longtime behaviour of strong solutions in 2D. 

We are interested in the sharp interface limit $\eps\rightarrow 0$ for the above system \eqref{eq:NSAC1}-\eqref{eq:NSAC5}. For well-prepared initial data and small times, we will rigorously prove the convergence of \eqref{eq:NSAC1}-\eqref{eq:NSAC5} to the following classical two-phase Navier-Stokes equation with surface tension:
\begin{alignat}{2}
	\label{eq:Limit1}
	\p_t \ve_0^\pm+\ve_0^\pm\cdot\nabla\ve_0^\pm-\nu^\pm\Delta \ve_0^\pm  +\nabla p^\pm_0 &= 0 &\qquad &\text{in }\Omega^\pm_t, t\in [0,T_0],\\\label{eq:Limit2}
	\Div \ve_0^\pm &= 0 &\qquad &\text{in }\Omega^\pm_t, t\in [0,T_0],\\\label{eq:Limit3}
	\llbracket 2\nu^\pm D\ve_0^\pm -p_0^\pm \tn{I}\rrbracket\no_{\Gamma_t} &= -\sigma H_{\Gamma_t}\no_{\Gamma_t} && \text{on }\Gamma_t, t\in [0,T_0],\\ \label{eq:Limit4}
	\llbracket\ve_0^\pm \rrbracket &=0 && \text{on }\Gamma_t, t\in [0,T_0],\\
	\label{eq:Limit5}
	V_{\Gamma_t}  &=\no_{\Gamma_t}\cdot \ve_0^\pm && \text{on }\Gamma_t, t\in [0,T_0],\\
	\label{eq:Limit6}
	\ve_0^-|_{\partial\Omega}&= 0&&\text{on }\partial\Omega\times(0,T_0),\\
	\Gamma_0=\Gamma^0,\qquad \ve_0^\pm|_{t=0}&= \ve_{0,0}^\pm && \text{ in }\Omega^\pm_0,\label{eq:Limit7}
\end{alignat}
where $T_0>0$, $\nu^\pm:=\nu(\pm 1)$,
$\Omega$ is the disjoint union of $\Omega^+_t, \Omega^-_t$ and $\Gamma_t$ for every $t\in[0,T_0]$, $\Omega^\pm_t$ are smooth domains, $\Gamma_t=\partial\Omega^+_t\subseteq \Omega$ and $\no_{\Gamma_t}$ is the interior normal of $\Gamma_t$ with respect to  $\Omega^+_t$. The jump $\llbracket u\rrbracket(.,t)$ in $x\in\Gamma_t$ of a quantity $u$ defined on $\Omega^+_t\cup\Omega^-_t$ is defined as
\begin{equation*}
	\llbracket u\rrbracket(x,t)= \lim_{r\to 0+} \left[u(x+r\no_{\Gamma_t}(x),t)- u(x-r\no_{\Gamma_t}(x),t)\right].
\end{equation*}
Moreover, $H_{\Gamma_t}$ is the (mean) curvature and $V_{\Gamma_t}$ is the normal velocity of $\Gamma_t$ with respect to $\no_{\Gamma_t}$. Furthermore, $(\Gamma^0,\ve_{0,0}^\pm)$ are suitable initial data. For the following we denote
\begin{equation*}
  \Omega^\pm = \bigcup_{t\in [0,T_0]}\Omega^\pm_t\times \{t\},\quad  \Gamma = \bigcup_{t\in [0,T_0]}\Gamma_t\times \{t\}.
\end{equation*}
The surface tension constant $\sigma$ is determined by $\sigma= \int_{\R}\theta_0'(\rho)^2\sd \rho$, where $\theta_0$ is the well-known optimal profile, i.e., the unique solution of
\begin{align}\label{eq:OptProfile}
	-\theta_0''+f'(\theta_0)=0\quad\text{ in }\R, \qquad
	\lim_{\rho\to\pm \infty}\theta_0(\rho)=\pm 1, \qquad \theta_0(0)=0.
\end{align}
As in Abels and Liu~\cite{StokesAllenCahn} we assume for the double well potential $f\colon \R\to \R$ that it is smooth and satisfies the assumptions
\begin{equation*}
  f'(\pm 1)=0, \quad f''(\pm 1)>0,\quad f(s)=f(-s)>0 \quad \text{for all }s\in (-1,1).
\end{equation*}
Then there is a unique solution $\theta_0 \colon\mathbb{R}\to \R$ of \eqref{eq:OptProfile},
which is monotone.
Moreover, for every $m\in\N_0$, there is some $C_m>0$ such that
\begin{equation*}\label{yuning:decayopti}
  |\p_\rho^m(\theta_0(\rho)\mp 1)|\leq C_m e^{-\alpha|\rho|}\quad\text{for all } \rho\in\R  \text{ with }\rho \gtrless 0,
\end{equation*}
where $\alpha=\min(\sqrt{f''(-1)}, \sqrt{f''(1)})$.  Since $f$ is assumed to be even, $\theta_0$ is odd and $\theta'_0$ is even.

Strong solutions for the problem \eqref{eq:Limit1}-\eqref{eq:Limit7} have been studied extensively in the literature starting with the results by Denisova and Solonnikov~\cite{DenisovaTwoPhase}. For further references we refer to Köhne, Prüss, and Wilke \cite{KoehnePruessWilkeTwoPhase} and the monograph by Pr\"uss and Simonett~\cite{PruessSimonettMovingInterfaces}, where in particular local well-posedness in an $L^p$-setting is shown. Existence of a notion of weak solutions, called varifold-solutions, globally in time was shown in \cite{GeneralTwoPhaseFlow}. Weak-strong uniqueness for these kind of solutions was shown by Hensel and Fischer~\cite{FischerHensel_NavStk}. 

Let us now comment on the choice of vanishing mobility $m_\eps=\sqrt{\eps}\rightarrow 0$ in \eqref{eq:NSAC4}. In \cite{AbelsConvectiveAC} a non-convergence result was shown for a convective Allen-Cahn equation for a mobility $m_\eps=m_0\eps^\alpha$, where $m_0>0$ is a constant and $\alpha>2$, and formal asymptotic calculations were carried out for the case $\alpha=0,1$. Hence for constant mobility $m_\eps=m_0$ the formal limit is a transport equation coupled to mean curvature flow, whereas for the case $m_\eps=m_0\eps$ the formal limit is a pure transport equation, cf.~\eqref{eq:Limit6} above. It is possible to adapt the formal calculations to the case of mobilities $m_\eps=m_0\eps^\alpha$ for all exponents $\alpha\in[0,1]$ (with the same limit system for $\alpha\in(0,1]$), the expansions just become more tedious and lengthy due to the fractional order ansatz. In Abels, Fei \cite{AbelsFei} the case of $m_\eps=1$, $\alpha=0$ was studied and rigorous convergence to a two-phase Navier-Stokes system coupled with mean curvature flow was shown as expected from the formal asymptotic expansions as long as a smooth solution of the limit system exists. However, for this limit system there is no conservation of mass and hence it could be considered physically less relevant compared to the classical two-phase Navier-Stokes system with surface tension, where one has pure transport of the interface. This clearly motivates the study of the case of vanishing mobility $m_\eps$ for $\eps\rightarrow 0$. To the best of our knowledge there is no rigorous convergence result in the case of vanishing mobility in the literature. The choice of $m_\eps=\sqrt{\eps}$ and $\alpha=\frac12$ for our result is motivated as follows: for the arguments in \cite{AbelsFei}, the exponent $\alpha=\frac12$ is critical in a heuristic sense by calculating the orders for $\alpha=1$ and by assuming some linear depencence on $\alpha$. The cases $\alpha\in(0,\frac12]$ should work formally with the strategy in \cite{AbelsFei}, but we decided to simply choose $\alpha=\frac12$, in particular in order to have simpler asymptotic expansions with just $\sqrt{\eps}$-spacing in the sums. We note that in a joint-work with Fischer the first and third author show convergence for more general scalings of $m_\eps>0$ using the relative entropy method. In this work the convergence is obtained in weaker norms (and assuming same viscosities for simplicity), but it also holds for three space dimensions, see \cite{AbelsFischerMoser}.

Our strategy to prove the sharp interface limit is via rigorous asymptotic expansions. The method goes back to de Mottoni and Schatzman \cite{DeMottoniSchatzman} who first applied it to prove the rigorous sharp interface limit for the Allen-Cahn equation. The strategy works as follows: it is assumed that there exists a smooth solution to the limit sharp interface problem locally in time (usually this is no restriction). Then in the first step, one rigorously constructs an approximate solution to the diffuse interface system via rigorous asymptotic expansions based on the evolving hypersurface that is part of the solution to the limit problem. In the second step, one estimates the difference between the exact and approximate solution with the aid of a spectral estimate for a linear operator depending on the diffuse interface equation and the approximate solution. Comparison principles are not needed for the method and one even obtains the typical profile of solutions across the diffuse interface. The strategy was applied to many other sharp interface limits as well, see Moser \cite{MoserAsy} for a list of results. Let us just mention the famous result by Alikakos, Bates, Chen \cite{AlikakosLimitCH} for the Cahn-Hilliard equation, Abels, Liu \cite{StokesAllenCahn} for a Stokes/Allen-Cahn system, Abels, Marquardt \cite{AbelsMarquardt1,AbelsMarquardt2} for a Stokes/Cahn-Hilliard system, and the recent result Abels, Fei \cite{AbelsFei} for the Navier-Stokes/Allen Cahn system with constant mobility. 

In general, rigorous results for sharp interface limits can be grouped into results concerning strong solutions for the limit system, in particular before singularities appear, and global time results using some weak notion for the sharp interface system. As described above, our result relies on the existence of a smooth solution for the limit system and assumes sufficiently small times. Another important strategy for sharp interface limits using strong solutions is the relative entropy method, see Fischer, Laux, Simon \cite{FischerLauxSimon_AC_MCF} where the convergence of the Allen-Cahn-equation to mean curvature flow is considered and Hensel, Liu \cite{HenselLiu}, where the Navier-Stokes/Allen-Cahn system with constant mobility (but equal viscosities) is considered, cf.~\cite{MoserAsy} for more references concerning the relative entropy method. Weak notions used for global time results for the Allen-Cahn equation are viscosity solutions (\cite{EvansSonerSouganidis,KatsoulakisKR,BarlesDaLio,BarlesSouganidis}), varifold solutions (\cite{Ilmanen,MizunoTonegawa,Kagaya}), BV-solutions (\cite{LauxSimon_BV}; conditional result) and a solution concept inbetween \cite{HenselLaux}. In \cite{AbelsFei} there are more references for results on Navier-Stokes/Cahn-Hilliard-type models. 

The following theorem is our main result about convergence of \eqref{eq:NSAC1}-\eqref{eq:NSAC5} to \eqref{eq:Limit1}-\eqref{eq:Limit7}:
\begin{thm}\label{thm:main}
Let $T_0>0$, $m_\eps=\sqrt{\eps}$ for all $\eps>0$, and $(\ve_0^\pm,\Gamma)$ be a smooth solution of the two-phase Navier-Stokes system with surface tension \eqref{eq:Limit1}-\eqref{eq:Limit7} on $[0,T_0]$ with $c_{0,\eps}(x,t)\in [-1,1]$ for all $(x,t)\in \overline{\Omega}\times [0,T_0]$, $\eps\in (0,1]$. Let $N\in\N$, $N\geq 3$. Then there exist $c_A=c_A(N,\eps),\ve_A=\ve_A(N,\eps)\in H^1(0,T_0;L^2(\Om))\cap L^2(0,T_0;H^2(\Om))$ for $\eps\in(0,1]$, uniformly bounded in these spaces and $\|c_A\|_\infty\leq 1+c$ with $c>0$ independent of $\eps\in(0,1]$, such that the following holds: 

Let $(\ve_\eps,c_\eps)$ be strong solutions of \eqref{eq:NSAC1}-\eqref{eq:NSAC5} with initial values  $\ve_{0,\eps}$, $c_{0,\eps}$ such that
\begin{equation}\label{initial assumption}
	\|c_{0,\eps}-c_A|_{t=0}\|_{L^2(\Omega)}+ \varepsilon^2\|\nabla(c_{0,\eps}-c_A|_{t=0})\|_{L^2(\Omega)}+ \|\ve_{0,\eps}-\ve_A|_{t=0}\|_{L^2(\Omega)}\leq C\eps^{\order+\frac12}
\end{equation}
for all $\eps\in (0,1]$ and some $C>0$.
Then there are some $\eps_0 \in (0,1]$, $R>0$ and $T_1\in(0,T_0]$ small such that for all $\eps \in (0,\eps_0]$ and some $C_R>0$ it holds
\begin{align}
		\|c_\eps-c_A\|_{L^\infty(0,T_1;L^2(\Omega))}+\eps^{\frac14}\|\nabla (c_\eps -c_A)\|_{L^2((\Omega\times (0,T_1))\setminus\Gamma(\delta))}  &\leq R\eps^{\order+\frac12},\label{eq_convC1}\\
		\eps^{\frac14} \|\nabla_{\btau_\eps}(c_\eps -c_A)\|_{L^2((\Omega\times(0,T_1))\cap \Gamma(2\delta))}+ \eps \|\nabla (c_\eps -c_A)\|_{L^2((\Omega\times(0,T_1))\cap \Gamma(2\delta))} &\leq R\eps^{\order+\frac12},\label{eq_convC2}\\
		\eps^2\|\nabla(c_\eps -c_A)\|_{L^\infty(0,T_1;L^2(\Omega))}+\eps^{2+\frac14}\|\nabla^2(c_\eps -c_A)\|_{L^2(\Omega\times(0,T_1))} &\leq R\eps^{\order+\frac12},\label{eq_convC3}\\
		\|\ve_\eps -\ve_A\|_{H^{\frac12}(0,T_1;L^2(\Om))}+\|\ve_\eps -\ve_A\|_{L^\infty(0,T_1;L^2(\Om))\cap L^2(0,T_1;H^1(\Om))} &\leq C_R\eps^{N+\frac14},\label{eq:convVelocityb}
\end{align}
where $\Gamma(\tilde{\delta})$ are standard tubular neighbourhoods for $\tilde{\delta}\in[0,3\delta]$, $\delta>0$ small and $\nabla_{\btau_\eps}$ is a suitable (approximate) tangential gradient, see Section \ref{subsec:Coordinates}. Moreover, let 
$d_\Gamma$ be the signed distance to $\Gamma$. Then
\begin{alignat}{2}
	c_A&=\zeta(d_\Gamma)\theta_0(\rho_\eps) \pm\chi_{\Omega^\pm} (1-\zeta(d_\Gamma)) + O(\eps^{\frac32})&\quad&\text{ in }L^\infty((0,T_0)\times\Om),\\
	\ve_A&=\ve_0^+(x,t)\eta(\rho_\eps)+\ve_0^-(x,t)(1-\eta(\rho_\eps)) + O(\sqrt{\eps}) &\quad&\text{ in }L^\infty(0,T_0;L^p(\Om)),
\end{alignat}
where $\eps\rho_\eps=d_\Gamma+O(\sqrt{\eps})$ in $\Gamma(3\delta)$, $p\in[1,\infty)$ is arbitrary, $\zeta\colon \R\to [0,1]$ is smooth such that $\supp \zeta \subseteq [-2\delta,2\delta]$ and $\zeta\equiv 1$ on $[-\delta,\delta]$, and $\eta\colon \mathbb{R}\rightarrow[0,1]$ is smooth such that $\eta=0$ in $(-\infty,-1]$, $\eta=1$ in $[1,\infty)$, $\eta-\frac{1}{2}$ is odd and $\eta'\geq0$ in $\mathbb{R}$. In particular, 
\[
\lim_{\eps\to 0} c_A=\pm 1\quad\text{ uniformly on compact subsets of }\quad \Omega^\pm.
\]
\end{thm}

\begin{rem}
Note that for strong solutions of \eqref{eq:NSAC1}-\eqref{eq:NSAC5} we have the energy inequality 
\begin{equation}\label{eq:EnergyEstim}
	\sup_{t\in [0,T_0]}\int_\Omega \tfrac12{|\ve_\eps(t)|^2}+\tfrac{\eps}2|\nabla c_\eps(t)|^2 + \tfrac1\eps f(c_\eps(t))\, dx+ \int_0^{T_0}\int_\Omega |D\ve_\eps|^2+ \tfrac1\eps|\mu_\eps|^2 \, dx \, dt \leq E_{0,\eps},
\end{equation}
where $\mu_\eps = -\eps \Delta c_\eps + \tfrac1{\eps} f'(c_\eps)$ and
\begin{equation*}
	E_{0,\eps} := \int_\Omega \tfrac12{|\ve_{0,\eps}|^2} \,dx
	+\int_\Omega (\tfrac{\eps}2|\nabla c_{0,\eps}|^2 + \tfrac1\eps f(c_{0,\eps}))\, dx.
\end{equation*}
Therefore the left-hand side of \eqref{eq:EnergyEstim} is uniformly bounded in $\eps \in (0,1)$ if $\sup_{\eps\in (0,1)} E_{0,\eps}<\infty$. Using a Taylor expansion for $f$ and the form of $c_A$, $\ve_A$ in Section \ref{sec:ApproxSolutions} below, one can show this bound under the assumption \eqref{initial assumption}. 
\end{rem}

Let us comment on the novelty of our contribution. We use a similar strategy as in Abels, Fei \cite{AbelsFei}. Compared to \cite{AbelsFei}, we consider the case of vanishing mobility $m_\eps=\sqrt{\eps}$ in \eqref{eq:NSAC4}, leading to the classical two-phase Navier-Stokes system with surface tension \eqref{eq:Limit1}-\eqref{eq:Limit7} in the sharp interface limit instead of the coupling with mean curvature flow in \eqref{eq:Limit6} obtained in \cite{AbelsFei}. Some remarks on the choice of the scaling for the mobility were included before. Note that our choice turns out to be critical for the arguments we use, and therefore we need to take time small in our result compared to \cite{AbelsFei}. Moreover, we need fractional order expansions with $\sqrt{\eps}$-spacing in the terms, cf.~Section \ref{sec:MatchedAsymptotics} below. Additionally, note that in \cite{AbelsFei} a new type of ansatz in higher orders was introduced based on a linearization idea that simplified the previous works \cite{StokesAllenCahn,AbelsMarquardt1,AbelsMarquardt2}. However, a direct modification with uncoupled equations for the higher order ansatz terms as in \cite{AbelsFei} does not lead to suitable estimates and hence is not enough to close the argument in our case. Therefore we modify this type of ansatz and obtain as model problem a coupled system (and another uncoupled problem in higher order) with suitable scaling in $\eps$, see Section \ref{subsec:ParabolicEq} and Section \ref{sec:ApproxSolutions} below. Moreover, we even have a term at order $O(\sqrt{\eps})$ in the expansion of the distance function which leads to problems when applying spectral estimates within standard tubular neighbourhood coordinates. Therefore we use the novel idea of working with $\eps$-dependent coordinates, in particular as framework for the spectral estimates, cf.~Section \ref{subsec:Coordinates} and Section \ref{subsec:spectral_estimate} below.  

Finally, let us summarize the structure of the paper. Section \ref{sec:Prelim} contains the required preliminaries, i.e., $\eps$-dependent coordinates, estimates of remainder terms, the (coupled and uncoupled) model problems with scalings in $\eps$ as well as spectral estimates based on the $\eps$-scaled coordinates. The asymptotic expansion is done in Section \ref{sec:MatchedAsymptotics}, where the novelty lies in the expansion in integer powers of $\sqrt\eps$ instead of integer powers of $\eps$.  The sophisticated higher order ansatz terms and remainder estimates are the content of Section \ref{sec:ApproxSolutions}. Finally, the main result is proven in Section \ref{sec:main_proof}, where a major part is the control of the error in the velocities in Section \ref{subsec:LeadingErrorVelocity}.

\section{Preliminaries}\label{sec:Prelim}

Throughout the manuscript $\N$ denotes the set of natural numbers (without $0$) and $\N_0=\N\cup\{0\}$. Let $\Omega\subseteq\R^N$ be open, $m\in\N_0$, $p\in[1,\infty]$ and $X$ be a Banach space. Then we denote with $L^p(\Omega;X)$ and $W^m_p(\Omega;X)$ the standard Lebesgue and Sobolev spaces. In the case $X=\R$ we write $L^p(\Omega)$ and $W^m_p(\Omega)$, respectively. Moreover, if $\Omega$ has finite measure, we define for $1\leq q\leq \infty$ and $k\in\N_0$
\begin{align*}
  L^q_{(0)}(\Omega) &:= \left\{f\in L^q(\Omega): \int_\Omega f(x)\, dx=0\right\},
\quad    W^k_{q,(0)}(\Omega) := W^k_q(\Omega)\cap L^q_{(0)}(\Omega). 
\end{align*}

\subsection{Coordinates}\label{subsec:Coordinates}
Let $\Omega\subseteq\R^2$ be a domain, $T_0>0$ and $\Gamma=\bigcup_{t\in[0,T_0]}\Gamma_t\times\{t\}$ be a smooth evolving compact closed curve contained in $\Omega$. Then $\Omega$ is divided into two disjoint connected components $\Omega^\pm_t$ such that $\partial\Omega^+_t=\Gamma_t$ for all $t\in[0,T_0]$. We parametrize $\Gamma_t$ for every $t\in [0,T_0]$ over the torus $\T^1=\R/2\pi\Z$ with an $X_0\colon \T^1\times[0,T_0]\rightarrow\Gamma$ such that $\partial_sX_0(s,t)\neq 0$ for all $s\in\T^1$, $t\in[0,T_0]$. Moreover, we denote the corresponding tubular neighbourhood coordinates with $(d_0,S_0)\colon\overline{\Gamma(3\delta)}\rightarrow[-3\delta,3\delta]\times\T^1$, where $\Gamma(\tilde{\delta}):=\{d_0\in(-\tilde{\delta},\tilde{\delta})\}$ is a relatively open neighbourhood of $\Gamma$ in $\Omega\times[0,T_0]$ for $\tilde{\delta}\in(0,3\delta]$ and $\delta>0$ is small such that $\overline{\Gamma(3\delta)}\subseteq\Omega\times [0,T_0]$. Here $d_0(\cdot ,t)$ is the signed distance function to $\Gamma_t$ for every $t\in[0,T_0]$ and $d_0(x,t)\gtrless 0$ in $\Omega^\pm_t$. We set $\Gamma_t(\tilde{\delta}):=\{x \in \Omega: (x,t)\in \Gamma(\tilde{\delta})\}$ for $\tilde{\delta}\in(0,3\delta]$ and we also write $d_\Gamma:=d_0$. Moreover, we denote with
 \begin{equation*}
\btau(s,t):=\frac{\partial_{s} X_0(s,t)}{|\partial_s X_0(s,t)|}\quad \text{and}\quad \no(s,t):=
\begin{pmatrix}
  0 & -1\\ 1 & 0
\end{pmatrix}
\btau(s,t), \quad\text{ for }(s,t)\in\T^1\times[0,T_0]
\end{equation*}
the unit tangent and normal vectors of $\Gamma_t$ at $X_0(s,t)$, where $X_0$ is chosen such that $\no(s,t)$ is the interior normal with respect to $\Omega^+(t)$ for all $t\in[0,T_0]$. Moreover, we define
\begin{equation*}
\no_{\Gamma_t}(x):= \no (s,t)\quad \text{ for all }~ x=X_0(s,t)\in \Gamma_t,
\end{equation*}
and let $V_{\Gamma_t}$ and $H_{\Gamma_t}$ be the normal velocity and (mean) curvature of $\Gamma_t$ with respect to $\no_{\Gamma_t}$ for $t\in[0,T_0]$. We denote
\begin{equation*}
   V(s,t):=V_{\Gamma_t}|_{X_0(s,t)},\quad H(s,t):= H_{\Gamma_t}|_{X_0(s,t)}\quad \text{for all }(s,t)\in\T^1\times[0,T_0].
 \end{equation*}
 Here and in the following $u|_{X_0}\colon \T^1\times [0,T_0]\to\R$ is defined by
 \begin{equation*}
   u|_{X_0}(s,t):= u|_{X_0(s,t)} := u(X_0(s,t))\quad \text{for all }(s,t)\in\T^1\times[0,T_0]
 \end{equation*}
 for a function $u$ defined on a set containing $\Gamma$.
It is well known that
  \begin{equation*}
    \nabla d_{\G}|_{X_0}=\no,~\quad\partial_t d_{\G}|_{X_0}=-V,\quad \Delta d_{\G}|_{X_0}=-H\quad\text{ on }\T^1\times[0,T_0].
\end{equation*}

Later we will need a suitable $\eps$-perturbation of the standard tubular neighbourhood coordinate system. Therefore we consider for $\eta>0$ and $\eps\in(0,\eps_0]$
\begin{align}\label{eq_d_eps}
	d_\eps(x,t)&:=d_0(x,t)+\eps^\eta\tilde{d}_\eps(x,t),\\
	S_\eps(x,t)&:=S_0(x,t)+\eps^\eta\tilde{S}_\eps(x,t)/2\pi\quad\text{ for }(x,t)\in\overline{\Gamma(3\delta)},\label{eq_S_eps}
\end{align}
where $(\tilde{d}_\eps,\tilde{S}_\eps)\colon\overline{\Gamma(3\delta)}\rightarrow\R^2$ are smooth with $C^k$-norm uniformly bounded with respect to $\eps\in(0,\eps_0]$ for every $k\in\N$, and we assume that
\begin{align}\label{eq_tilde_deps_0}
\tilde{d}_\eps=0\quad\text{ on }\quad \Gamma(3\delta)\backslash\Gamma(\delta')
\end{align}
for some $\delta'\in (0,3\delta)$.
For small $\eps$ these coordinates also have suitable properties similar to a tubular neighbourhood system because of the following theorem.
\begin{theorem}[$\eps$-Coordinates]\label{th_coord}
	For $\eps_1>0$ sufficiently small and every $\eps\in(0,\eps_1]$ the $\eps$-coordinates $(d_\eps,S_\eps,\textup{id}_t)\colon \overline{\Gamma(3\delta)}\rightarrow[-3\delta,3\delta]\times\T^1\times[0,T_0]$ are well-defined and yield a smooth diffeomorphism with inverse $X_\eps$. Moreover, for $\eps_1$ small 
	\begin{align}\label{eq_Gamma_eps}
		\Gamma(\delta)\subseteq \Gamma^\eps(\tfrac{3\delta}{2}) \subseteq \Gamma(2\delta)\subseteq \Gamma^\eps(\tfrac{9\delta}{4}) \subseteq \Gamma(\tfrac{5\delta}2)\subseteq \Gamma^\eps(\tfrac{11\delta}{4}) \subseteq \Gamma(3\delta)
	\end{align}
for all $\eps\in (0,\eps_1]$, where for $\delta'>0$
        \begin{equation*}
          \Gamma^\eps(\delta'):=\big\{(x,t)\in \Gamma(3\delta): d_\eps(x,t)\in(-\delta',\delta')\big\}.
        \end{equation*}
Moreover, for every $k\in\N$ the $C^k$-norms of $d_\eps, S_\eps$ are uniformly bounded.
\end{theorem}
\begin{proof}
	Because of \eqref{eq_tilde_deps_0} we obtain that $(d_\eps,S_\eps,\textup{id}_t):\overline{\Gamma(3\delta)}\rightarrow[-3\delta,3\delta]\times\T^1\times[0,T_0]$ is well-defined and smooth for $\eps>0$ small. Moreover, because of compactness and the definitions it holds that $D(d_\eps,S_\eps,\textup{id}_t)$ is invertible pointwise in $\overline{\Gamma(3\delta)}$ for $\eps>0$ small. Hence $(d_\eps,S_\eps,\textup{id}_t)$ is a local diffeomorphism. Furthermore, a compactness and extension argument shows that $(d_0,S_0,\textup{id}_t)$ is globally bi-Lipschitz. This extends to $(d_\eps,S_\eps,\textup{id}_t)$ with uniform constants for $\eps>0$ small. Injectivity of $(d_\eps,S_\eps,\textup{id}_t)$ directly follows and surjectivity can now be proven by showing that the image is open and closed in the connected space $[-3\delta,3\delta]\times\T^1$. The additional statement is clear from the definitions for $\eps>0$ small.
      \end{proof}

      As before we define     
      $$
      \Gamma^\eps_t(\tilde{\delta}):=\{x \in \Omega: (x,t)\in \Gamma^\eps(\tilde{\delta})\}\quad \text{for }\tilde{\delta}\in(0,3\delta].
      $$
      \begin{rem}
	In order to transform integrals with $X_\eps$ later, we define $J_\eps\colon [-3\delta,3\delta]\times \T^1\times [0,T_0]\to (0,\infty)$ by
	\begin{align}\label{J_eps}
		J_\eps:=|\det DX_\eps| = \frac{1}{\sqrt{|\nabla d_\eps|^2|\nabla S_\eps|^2-(\nabla d_\eps\cdot\nabla S_\eps)^2}} \circ X_\eps.
	\end{align}
\end{rem}

Furthermore, we denote
\begin{equation*}
  \no_\eps (x,t):= \nabla d_\eps(x,t)\qquad \text{for all }(x,t)\in \ol{\Gamma(3\delta)}.
\end{equation*}
For the following we assume that
\begin{equation}\label{eq:condSepsneps}
  |\no_\eps|^2=|\nabla d_\eps|^2=1+O(\eps^2),\quad \no_\eps\cdot \nabla S_\eps= O(\eps^2)  
\end{equation}
in $C^k_b(\overline{\Gamma(3\delta)})$ for all $k\in\N$, which we will assure in the following constuction.
The following identity will be useful in relation with divergence free functions:
      \begin{equation}\label{eq:XepsId}
  \partial_r X_\eps (r,s,t) \otimes \no_\eps(X_\eps(r,s,t),t)+ \partial_s X_\eps (r,s,t)\otimes \nabla S_\eps (X_\eps(r,s,t),t) = I
\end{equation}
for all $r\in (-3\delta,3\delta)$, $s\in \T^1, t\in [0,T_0]$. It is a consequence of differentiating
\begin{equation*}
X_\eps (d_\eps(x,t), S_\eps(x,t))= x\qquad \text{for all }(x,t)\in \Gamma^\eps(3\delta).
\end{equation*}
This motivates to define for suitable $\psi$
\begin{equation}
  \label{eq:defnTauEps}
  \nabla_{\btau_\eps} \psi:= \nabla S_\eps [\partial_s (\psi\circ X_\eps)]{\circ X_\eps^{-1}}.
\end{equation}
Then
\begin{equation*}
  \nabla \psi = \nabla_{\btau_\eps} \psi+ \no_\eps(\partial_r X_\eps \circ X_\eps^{-1}\cdot \nabla) \psi. 
\end{equation*}
Moreover, \eqref{eq:XepsId} implies
\begin{align}\label{eq:XepsNoEps}
  \partial_r X_\eps(r,s,t)&= \no_\eps (X_\eps(r,s,t),t)+O(\eps^2)= \no_\eps (X_\eps(0,s,t),t)+O(\eps^2)
\end{align}
in $C^k_b(\overline{\Gamma(3\delta)})$ for all $k\in\N$ due to \eqref{eq:condSepsneps}
and
\begin{equation*}
  \partial_r \no_\eps (X_\eps (r,s,t))= \nabla \no_\eps (X_\eps(r,s,t))\cdot \partial_r X_\eps(r,s,t)= O(\eps^2)
\end{equation*}
in $C^k_b(\overline{\Gamma(3\delta)})$ for all $k\in\N$. In particular this shows
\begin{equation*}
  (\partial_r X_\eps)\circ X_\eps^{-1}\cdot \nabla u = \partial_{\no_\eps} u+ O(\eps^2)
\end{equation*}
for every sufficiently smooth $u \colon \Gamma(3\delta)\to \R$, where $\partial_{\no_\eps} u:= \no_\eps \cdot \nabla u$. Similarly, by multiplying \eqref{eq:XepsId} with $\nabla S_\eps (X_\eps(r,s,t),t)$ one obtains
\begin{equation*}
  \partial_s X_\eps (r,s,t)= |\nabla S(X_\eps(r,s,t))|^{-2} \nabla S(X_\eps(r,s,t)+ O(\eps^2)
\end{equation*}
in $C^k_b(\overline{\Gamma(3\delta)})$ for all $k\in\N$.

  \begin{rem}
    Let $a\colon\Gamma(3\delta)\to \R$ be smooth in normal direction and assume $a=0$ on $\Gamma$, then $\tilde{a}:\Gamma(3\delta)\to \R$ with
	\begin{equation*}
		\tilde{a}(x,t):=
		\begin{cases}
			\frac{a(x,t)}{d_\Gamma(x,t)} &\text{ for all }(x,t)\in \Gamma(3\delta)\setminus \Gamma,\\
			\partial_{\no} a(x,t) &\text{ for all }(x,t)\in \Gamma
		\end{cases}
	\end{equation*}
	is well-defined, smooth in normal direction and tangential regularity is conserved. In particular $\tilde{a}$ is smooth provided that $a$ is smooth. This can be shown with a Taylor expansion in $d_\Gamma$.
\end{rem}
        A similar statement, based on a Taylor expansion in normal direction for Sobolev functions, is given by the following lemma and will be useful to estimate remainder terms.

\begin{lem}\label{lem:Taylor}
  Let $t\in [0,T_0]$, $\eps\in (0,\eps_1)$ with $\eps_1>0$ as in Theorem~\ref{th_coord}, $\delta'\in [2\delta,3\delta]$ and $a\in W^k_p(\Gamma_t(\delta'))$ for some $k\in\N$, $1< p< \infty$. Then there are $r_{k,\eps,t} \in L^p(\Gamma_t(\delta'))$ such that
  \begin{equation*}
    a(x)= \sum_{j=0}^{k-1} (\partial_{\no_\eps}^j a)(P_\eps(x,t))\frac{d_\eps(x,t)^j}{j!} + d_\eps(x,t)^k r_{k,\eps,t}(x)\quad \text{for all }x\in \Gamma_t(\delta'),
  \end{equation*}
 where $P_\eps(x,t):= X_\eps(0,S_\eps(x,t),t)$, $\no_\eps(x,t)= \nabla d_\eps(x,t)$, and
 \begin{equation}\label{eq:rkepsEstim}
  \|r_{k,\eps,t}\|_{L^p(\Gamma_t(\delta'))}\leq C_k \|a\|_{W^k_p(\Gamma_t(\delta'))}.
 \end{equation}
  for some $C_k$ independent of $\eps,t$, and $a$.
\end{lem}
\begin{proof}
  Since smooth functions are dense in $W^k_p(\Gamma_t(\delta'))$, we can assume that $a$ is smooth. We define the auxiliary function $\Phi_{x,t}(r):= X_\eps(r,S_\eps(x,t),t)$ for all $r\in [-\delta',\delta']$. Then by a one-dimensional Taylor expansion
  \begin{align*}
    a(x)&= a(\Phi_{x,t}(d_\eps(x,t))\\
        &= \sum_{j=0}^{k-1} \frac{d^j}{dr^j} (a(\Phi_{x,t}(r)))|_{r=0} \frac{d_\eps(x,t)^j}{j!} + \int_0^{d_\eps(x,t)}\frac{d^k(a(\Phi_{x,t}(r)))}{dr^k} \frac{(d_\eps(x,t)-r)^{k-1}}{(k-1)!} \,dr\\
          &= \sum_{j=0}^{k-1} (\partial_{\no_\eps}^j a)(P_\eps(x,t))\frac{d_\eps(x,t)^j}{j!} + d_\eps(x,t)^kr_{k,\eps,t}(x)
  \end{align*}
  where
  \begin{equation*}
    r_{k,\eps,t}(x)= \frac1{d_\eps(x,t)}\int_0^{d_\eps(x,t)}\frac{d^k(a(\Phi_x(r)))}{dr^k} \frac{(1-r/d_\eps(x,t))^{k-1}}{(k-1)!} \,dr.
  \end{equation*}
  Now using that by Hardy's inequality
  \begin{equation*}
    \left(\int_{-\delta'}^{\delta'}\left| \frac1{\tilde{r}} \int_0^{\tilde{r}}\left|\frac{d^k(a(\Phi_{x,t}(r)))}{dr^k}\right|\,dr \right|^p\, d\tilde{r}  \right)^{\frac1p}\leq
    C_p \left(\int_{-\delta'}^{\delta'}\left|\frac{d^k(a(\Phi_{x,t}(r)))}{dr^k}\right|^p\,dr  \right)^{\frac1p}
  \end{equation*}
  one easily shows \eqref{eq:rkepsEstim}.
\end{proof}

For the following let $h\colon \T^1\times [0,T_0]\to \R$ be sufficiently smooth.
 Then we have
 \begin{equation}\label{Prelim:1.13}
  \begin{split}
    \nabla^{\Gamma_t^\eps} h(r,s,t) &:= \nabla_x (h(S_\eps(x,t),t) =  \nabla  S_\eps(x,t) \p_{s} h(s,t),\\
  \Delta^{\Gamma_t^\eps} h(r,s,t)&:= \Delta_x (h(S_\eps(x,t),t) =  (\Delta S_\eps)(x,t)\cdot\partial_s h(s,t)+|\nabla S_\eps(x,t)|^2 \p_{s}^2 h(s,t)
  \end{split}
\end{equation}
for all $(x,t)\in\Gamma(3\delta)$, where $r\in (-3\delta,3\delta)$ and $s\in\T^1$ are determined by $x=X_\eps(r,s,t)$. Therefore  we define for every sufficiently smoth $h\colon \T^1\times [0,T_0]\to \R$
 \begin{equation*}
 \begin{split}
  (\nabla_{\Gamma_t^\eps} h)(s,t)&:=(\nabla^{\Gamma_t^\eps} h)(0,s,t),\quad (\Delta_{\Gamma_t^\eps} h)(s,t):=(\Delta^{\Gamma_t^\eps} h)(0,s,t)\quad \text{for all }s\in\T^1,t\in [0,T_0].
 \end{split}
\end{equation*}
We note that coefficients of the differences 
\begin{equation*}
  (\nabla^{\Gamma_t^\eps} h)(r,s,t)-(\nabla_{\Gamma_t^\eps} h)(s,t),\quad (\Delta^{\Gamma_t^\eps} h)(r,s,t)-(\Delta_{\Gamma_t^\eps} h)(s,t)
\end{equation*}
vanish for $r=0$, which corresponds to $x\in\Gamma_t^\eps$.

Finally, let $U_t\subseteq \R^2$, $t\in [0,T]$, be open sets and  $\mathcal{U}:=\bigcup_{t\in [0,T]} U_t$. Then we define for $s\geq 0$ 
\begin{align*}
  L^2(0,T;H^s(U_t)) &:= \left\{g\in L^2(\mathcal{U}): g(\cdot,t)\in H^s(U_t)\text{ for a.e.\ }t\in (0,T),  \|g(\cdot,t)\|_{H^s(U_t)}\in L^2(0,T)\right\},\\
  \|g\|_{L^2(0,T;H^s(U_t))} &:= \left(\int_0^T \|g(\cdot,t)\|_{H^s(U_t)}^2\, dt\right)^{\frac12}, \\
  L^2(0,T;H^s(\Gamma_t)) &:= \left\{g\in L^2(\Gamma): g\circ X_0 \in L^2(0,T;H^s(\T^1)) \right\},\\
  \|g\|_{L^2(0,T;H^s(\Gamma_t))} &:= \|g\circ X_0\|_{L^2(0,T;H^s(\T^1))}.
\end{align*}

\subsection{The Stretched Variable and Remainder Terms}

In the following we will use a ``stretched variable'', which is defined by
\begin{equation}\label{eq:StretchedVariable}
  \rho=\rho_\eps:= \frac{d_\eps(x,t)}\eps \qquad \text{for } (x,t)\in \Gamma(3\delta), \eps \in (0,\eps_1],
\end{equation}
where $d_\eps\colon \Gamma(3\delta)\to \R$ is as in the previous subsection and $\eps_1>0$ is as in Theorem~\ref{th_coord}. In particular, it  satisfies
 \begin{equation*}
   |\nabla d_\eps (x,t)| = 1+ \eps^2b_\eps(x,t) \qquad \text{for all }(x,t)\in \Gamma(3\delta),
 \end{equation*}
 where $b_\eps$ and all its derivatives are uniformly bounded in $\eps \in (0,\eps_1]$ for some $\eps_1>0$ sufficiently small.

For a systematic treatment of the remainder terms, we introduce:
\begin{defn}\label{eq:1.15}
  For any $k\in \R$ and $\alpha>0$,  $\mathcal{R}_{k,\alpha}$ denotes the vector space of family of continuous functions $\tr_\eps\colon \R\times \Gamma(3\delta) \to \R$, indexed by $\eps\in (0,1)$, which are continuously differentiable with respect to $\no_{\Gamma_t}$ for all $t\in [0,T_0]$ such that
  \begin{equation}\label{eq:EstimRkalpha}
    |\partial_{\no_{\Gamma_t}}^j \tr_\eps(\rho,x,t)|\leq Ce^{-\alpha |\rho|}\eps^k\qquad \text{for all }\rho\in \R,(x,t)\in\Gamma(3\delta), j=0,1, \eps \in (0,1)
  \end{equation}
for some $C>0$ independent of $\rho\in \R,(x,t)\in\Gamma(3\delta)$, $\eps\in (0,1)$.  $\mathcal{R}_{k,\alpha}^0$ is the subclass of all $(\tr_\eps)_{\eps\in (0,1)}\in \mathcal{R}_{k,\alpha}$ such that
$\tr_\eps(\rho,x,t)= 0$ for all  $\rho \in\R, x\in\Gamma_t, t\in [0,T_0]$.
\end{defn}

  We remark that $\mathcal{R}_{k,\alpha}$ and $\mathcal{R}^0_{k,\alpha}$ are closed under multiplication and $\mathcal{R}_{k,\alpha}\subset \mathcal{R}_{k-1,\alpha}$.

\begin{lem}\label{lem:rescale2}
  Let $0<\eps\leq \eps_1$, $d_\eps$ be defined as before, $\delta'\in [\tfrac{\delta}2, 3\delta]$ and   $(\tr_\eps)_{0<\eps<1}\in L^2(\T^1;L^1(\R)\cap L^2(\R))$, $(a_\eps)_{\eps\in (0,1)} \subseteq C(\overline{\Gamma(3\delta)})$ such that
  \begin{equation*}
    |a_\eps(x,t)|\leq C|d_\eps (x,t)|^j\qquad \text{for all }(x,t)\in\Gamma(3\delta), \eps\in (0,\eps_1]
  \end{equation*}
  for some $C>0$ independent of $\eps,x,t$ and $j\in\N_0$.
  Then there is some $C>0$, independent of $0<\eps\leq \eps_1$, $\eps_1\in (0,1)$  such that for each $t\in[0,T]$
    \begin{equation*}
      r_\eps (x):= a_\eps(x,t) \tr_\eps\left(\frac{d_\eps(x,t)}\eps, S_\eps(x,t)\right)\qquad \text{for }x\in\Gamma_t(3\delta)
    \end{equation*}
  satisfies
  \begin{align}\label{eq:RemEstim1-2}
    \left\|r_\eps \varphi \right\|_{L^1(\G_t^\eps(\delta'))} &\leq C\eps^{1+j}\|\rho^j\hat{r}_\eps\|_{L^2(\T^1;L^1(\R))}\|\varphi\|_{H^1(\G_t^\eps(\delta'))},\\
  \label{eq:RemEstim2-2}
    \left\| r_\eps \right\|_{L^2(\G_t^\eps(\delta'))} &\leq C\eps^{\frac 12+j} \|\rho^j\hat{r}_\eps\|_{L^2(\T^1;L^2(\R))}
  \end{align}
  uniformly for all $\varphi\in H^1(\Gamma_t^\eps(\delta'))$, $t\in [0,T_0]$, and $\eps\in (0,\eps_1]$.
\end{lem}
\begin{proof}
  With the aid of the change of variables $x= X_\eps(r,s)$, where $r=d_\eps(x,t)$, $s=S_\eps(x,t)$, we obtain
  \begin{align*}
      \left\|r_\eps\varphi \right\|_{L^1(\G_t^\eps(\delta'))}     &=\int_{-\delta'}^{\delta'}\int_{\T^1} |a_\eps(X_\eps(r,s,t),t)| \left|\tr_\eps \left(\tfrac{r}\eps,s,t \right)\right||\varphi(X_\eps(r,s,t))|J_\eps(r,s,t)\,ds\,dr \\
    &\leq C\int_{\T^1}\int_\R |r^j|| \tr_\eps (\tfrac{r}\eps, s,t)|\,dr \sup_{r\in (-\delta',\delta')}|\varphi(X_\eps(r,s,t))|  \,ds\\
    &\leq C \eps^{1+j}\left(\int_{\T^1}\left|\int_{\R} |\rho^j \tr_\eps (\rho,s,t)|\,d\rho\right|^2 ds\right)^{\frac12} \left(\int_{\T^1}\sup_{|r|\leq \delta' }|\varphi(X_\eps(r,s,t))|^2\,ds\right)^{\frac12}\\
    &\leq C\eps^{1+j}\|\rho^j\tr_\eps\|_{L^2(\T^1;L^1(\R))} \|\varphi\|_{H^1(\Gamma^\eps_t(\delta'))}
  \end{align*}
for all  $\eps\in (0,\eps_1]$, $t\in [0,T_0]$, and $\varphi\in H^1(\Omega)$. This proves the first estimate.

In the same way we estimate
  \begin{align*}
    &  \left\|r_\eps\right\|_{L^2(\G_t^\eps(\frac{3\delta}2))}^2 =\int_{-\delta'}^{\delta'}\int_{\T^1} |a_\eps(X_\eps(r,s,t),t)|^2 \left|\tr_\eps \left(\tfrac{r}\eps,s,t \right)\right|^2J_\eps(r,s,t)\,ds\,dr \\
    &\leq C\int_{\T^1}\int_\R |r^j|^2| \tr_\eps (\tfrac{r}\eps, s,t)|^2\,dr\,ds= C\eps^{1+2j}\|\rho^j\tr_\eps\|_{L^2(\T^1\times \R)}^2
  \end{align*}
 for all  $\eps\in (0,\eps_1]$ and $t\in [0,T_0]$, which shows the second estimate.
\end{proof}
\begin{lem}\label{lem:DivergenceFreeRemainder}
  Let $g\in \SD(\R)$ and $\zeta\in C^\infty_0(\R)$ with $\supp \zeta \subseteq [-\tfrac{5\delta}2,\tfrac{5\delta}2]$. Then there is constant $C>0$ such that for all $t\in [0,T_0]$, $a\in H^1(\T^1)$ and $\boldsymbol{\varphi}\in H^1(\Omega)^d\cap L^2_\sigma (\Omega)$ we have
  \begin{equation*}
    \left|\int_{\Gamma_t(3\delta)}\zg g'(\rho_\eps(x,t))a(S_\eps(x,t))\no_{\eps}\otimes \no_{\eps} : \nabla \boldsymbol{\varphi}(x) \, dx\right|\leq C \eps^{\frac32} \|a\|_{H^1(\T^1)}\|\boldsymbol{\varphi}\|_{H^1(3\delta)},
  \end{equation*}
  where $\no_\eps = \nabla d_\eps$.
\end{lem}
\begin{proof}
  We use that
  \begin{equation}
    \label{eq:divtaueps}
  -\frac{\no_{\eps}}{|\no_\eps|}\otimes \frac{\no_{\eps}}{|\no_\eps|} : \nabla \boldsymbol{\varphi}= \underbrace{\left(\mathbf{I}-\frac{\no_{\eps}}{|\no_\eps|}\otimes \frac{\no_{\eps}}{|\no_\eps|}\right)}_{\mathbf{P}_\eps:=} : \nabla \boldsymbol{\varphi}
\end{equation}
since $\Div \boldsymbol{\varphi}=0$.
  To treat the remaining integral, we use integration by parts twice
  to get for all $\eps\in (0,\eps_1)$, $t\in [0,T_0]$
\begin{align*}
  I&:=\left|\int_{\Gamma_t(3\delta)} \zg g'(\rho_\eps(x,t))a(S_\eps(x,t))\no_{\eps}\otimes \no_{\eps} : \nabla \boldsymbol{\varphi} \, dx\right|\\
   &=\left|\int_{\Gamma_t(\frac{5\delta}2)}\zg \eps (\no_\eps\cdot \nabla g(\rho_\eps(x,t)))a(S_\eps(x,t))\frac{\no_{\eps}}{|\no_\eps|}\otimes \frac{\no_{\eps}}{|\no_\eps|} : \nabla \boldsymbol{\varphi} \, dx\right|\\
   &= \left|\int_{\Gamma_t(\frac{5\delta}2)} \eps  g(\rho_\eps(x,t))\no_\eps\cdot \nabla\left(\zg a(S_\eps(x,t))\mathbf{P}_\eps : \nabla \boldsymbol{\varphi}\right) \, dx\right|\\
   &\leq \left|\int_{\Gamma_t(\frac{5\delta}2)} \eps  g(\rho_\eps(x,t))\mathbf{P}_\eps : \nabla\left[\zg a(S_\eps(x,t)) (\no_\eps\cdot \nabla) \boldsymbol{\varphi}\right] \, dx\right|\\
  & \quad +\left|\int_{\Gamma_t(\frac{5\delta}2)} \eps g(\rho_\eps(x,t))(a(S_\eps(x,t))\mathbf{Q}_\eps+(\partial_s a)(S_\eps(x,t))\mathbf{R}_\eps) :\nabla  \boldsymbol{\varphi} \, dx\right|\\
   &\leq \left|\int_{\Gamma_t(\frac{5\delta}2)} \zg \eps \mathbf{P}_\eps(\nabla g(\rho_\eps(x,t)))a(S_\eps(x,t))\cdot (\no_\eps\cdot \nabla)  \boldsymbol{\varphi} \, dx\right|
  \\
 &\quad +\eps C\int_{\Gamma_t^\eps(\tfrac{11\delta}4)} |g(\rho_\eps(x,t))||a(S_\eps(x,t))|+|(\partial_s a)(S_\eps(x,t))||\nabla  \boldsymbol{\varphi}| \, dx
\end{align*}
for some uniformly bounded $\mathbf{Q}_\eps, \mathbf{R}_\eps$, where
\begin{equation*}
  \mathbf{P}_\eps(\nabla g(\rho_\eps(x,t)))= \mathbf{P}_\eps \left(\no_\eps \frac1\eps g'(\rho_\eps(x,t)\right) =0.
\end{equation*}
Now using $g\in \SD(\R)$ and \eqref{eq:RemEstim2-2} 
we obtain
\begin{align*}
I & \leq C\eps^{\frac{3}{2}}\left\Vert a\right\Vert _{H^{1}(\mathbb{T}^{1})}\left\Vert \boldsymbol{\varphi}\right\Vert _{H^{1}\left(\Gamma_{t}(2\delta)\right)}.
\end{align*}
This finishes the proof.
\end{proof}

\subsection{Parabolic Equations on Evolving Hypersurfaces}\label{subsec:ParabolicEq}

For $T\in (0,\infty)$ and $r\in[0,1]$ we shall denote the function spaces	
\begin{align*}
  X_{T,r}:=L^2(0,T;H^{2+r}(\T^1))\cap H^1(0,T;H^{r}(\T^1)),\quad X_T:=X_{T,\frac{1}{2}}.
\end{align*}
We equip $X_T$ with the norm
\begin{equation*}
  \|u\|_{X_T} =\|u\|_{L^2(0,T;H^{5/2}(\T^1))}+\|u\|_{H^1(0,T;H^{1/2}(\T^1))}+ \|u|_{t=0}\|_{H^{3/2}(\T^1)}.
\end{equation*}
Then it holds
\begin{equation}\label{eq:EmbeddingE1}
X_T\hookrightarrow C([0,T]; H^{3/2}(\T^1))\cap L^4(0,T; H^2(\T^1))
\end{equation}
and the operator norm of the embedding is uniformly bounded in $T$.

In the following theorem we derive uniform estimates for a class of degenerate parabolic partial differential equations. 

\begin{thm}\label{thm:ParabolicEqOnSurface}
  Let $0<T\leq T_1 <\infty$ and $\kappa\in (0,1]$, $r\in [0,1]$ and $a_\kappa,b_\kappa,c_\kappa\colon \T^1\times [0,T]\to \R$ be twice continuously differentiable with uniformly bounded $C^2$-norms with respect to $\kappa\in (0,1]$. Moreover, let there be some $c_0>0$, independent of $\kappa$, such that $c_\kappa(s,t) \geq c_0$ for all $(s,t)\in \T^1\times [0,T]$.
  For every $g\in L^2(0,T;H^r(\T^1))$ and $h_0\in H^{1+r}(\T^1)$ there is a unique solution $h\in X_{T,r}$ of
  \begin{alignat}{2}\label{eq:h1}
    \partial_t h+ a_\kappa\partial_s h + b_\kappa h -\kappa c_\kappa\partial_s^2 h  &= g&\qquad& \text{on }\T^1\times [0,T],\\\label{eq:h2}
    h|_{t=0} &=h_0 && \text{on } \T^1.
  \end{alignat}
  Moreover, there is some $C=C(T_1)>0$ independent of $\kappa \in (0,1]$, $T\in (0,T_1]$, $h,  g_\kappa, h_0$ such that
  \begin{align}\label{eq:hEstim1}
    \|h\|_{C([0,T]; H^{r}(\T^1))} + \sqrt{\kappa}\|h\|_{L^2(0,T; H^{1+r}(\T^1))}
    &\leq C\left(\|g\|_{L^1(0,T; H^{r}(\T^1))}+\|h_0\|_{H^{r}(\T^1)}\right)\\
     \label{eq:hEstim2}
    \sqrt{\kappa} \|h\|_{C([0,T]; H^{1+r}(\T^1))} + \kappa \|h\|_{L^2(0,T; H^{2+r}(\T^1))}
    &\leq C\left(\|g\|_{L^2(0,T; H^r(\T^1))}+\|h_0\|_{H^{1+r}(\T^1)}\right).
  \end{align}
\end{thm}
\begin{rem}
	Note that Theorem \ref{thm:ParabolicEqOnSurface} can be applied for right hand sides $g_\kappa$ depending on $\kappa$.
\end{rem}
\begin{proof*}{of Theorem~\ref{thm:ParabolicEqOnSurface}}
  Existence of a unique solution follows by standard results on linear parabolic equations. Therefore we only need to prove the uniform estimates.

  First we consider the case $r=0$. Then testing \eqref{eq:h1} with $h$ and integrating with respect to $t$ we obtain
  \begin{align*}
    &\sup_{0\leq t\leq T} \|h(t)\|_{L^2(\T^1)}^2 + \frac{\kappa c_0}{2} \int_0^T \int_{\T^1}|\partial_s h|^2\, ds\, dt \\
    &\leq C\sup_{s\in \T^1, t\in [0,T]} \left(|\partial_s a_\kappa(s,t)|+|b_\kappa(s,t)|+ |\partial_sc_\kappa (s,t)|^2 \right)\int_0^T \int_{\T^1}|h|^2\,ds \, dt\\
    &\qquad + \int_0^T \|g(t)\|_{L^2(\T^1)}\,dt \sup_{0\leq t\leq T} \|h(t)\|_{L^2(\T^1)}.
  \end{align*}
   Hence Young's and Gronwall's inequality imply \eqref{eq:hEstim1}.

  Next let $r=1$. Then differentiating \eqref{eq:h1} with respect to $s$ yields for $\tilde{h}=\partial_s h$
  \begin{equation}\label{eq:tildeh}
      \partial_t \tilde{h}+ \tilde{a}_\kappa\partial_s \tilde{h} + \tilde{b}_\kappa \tilde{h} -\kappa c_\kappa\partial_s^2 \tilde{h}  = \partial_s g\qquad \text{on }\T^1\times [0,T],
    \end{equation}
    for some $\tilde{a}_\kappa, \tilde{b}_\kappa\colon \T^1\times [0,T]\to \R$, which are smooth and have $C^1$-norms uniformly bounded in $\kappa\in (0,1]$. Hence the same estimate as before yields
    \begin{equation*}
         \|\partial_sh\|_{C([0,T]; L^2(\T^1))} + \sqrt{\kappa}\|\partial_s h\|_{L^2(0,T; H^1(\T^1))}
    \leq C\left(\|g\|_{L^1(0,T; H^{1}(\T^1))}+\|h_0\|_{H^{1}(\T^1)}\right)
  \end{equation*}
  for some $C>0$ independent of $\kappa\in (0,1]$, $T\in (0,T_1]$ and $g,h_0$. This implies \eqref{eq:hEstim1} in the case $r=1$. Finally, \eqref{eq:hEstim1} for the case $r\in [0,1]$ follows by interpolation.

  In order to prove \eqref{eq:hEstim2} in the case $r=0$ we test \eqref{eq:h1} with $-\kappa \partial_s^2 h$ and obtain
  \begin{align*}
    &\kappa\frac{d}{dt} \int_{\T^1} \frac{|\partial_s h|^2}2\, ds  + \kappa^2 \int_{\T^1} c_\kappa |\partial_s^2h|^2\,d s  = \kappa\int_{\T^1} a_\kappa \partial_s h \partial_s^2 h \, ds -\int_{\T^1} (g+b_\kappa h) \kappa \partial_s^2 h\,ds,
  \end{align*}
  where
  \begin{align*}
    \kappa\int_{\T^1} a_\kappa\partial_s h \partial_s^2 h \, ds &= -\kappa \int_{\T^1} \partial_s a_\kappa\frac{|\partial_s h|^2}2\,ds
                                                                   \leq C\kappa\|h(t)\|_{H^1(\T^1)}^2,\\
    -\int_{\T^1} (g+b_\kappa h) \kappa \partial_s^2 h\,ds&\leq C(\|g(t)\|_{L^2(\T^1)}+\|h(t)\|_{L^2(\T^1)})\kappa \|\partial_s^2 h(t)\|_{L^2(\T^1)}.
  \end{align*}
  Hence integration in time, \eqref{eq:hEstim1} with $r=0$ and Young's inequality finally yield \eqref{eq:hEstim2}.

  In the case $r=1$ we use again that $\tilde{h}= \partial_s h$ solves \eqref{eq:tildeh}. Testing this equation with $-\kappa \partial_s^2\tilde{h}$ yields in the same way as before
  \begin{align*}
    &\kappa\frac{d}{dt} \int_{\T^1} \frac{|\partial_s \tilde{h}|^2}2\, ds  + \kappa^2 \int_{\T^1} c_\kappa |\partial_s^2\tilde{h}|^2\,d s  = \kappa\int_{\T^1} a_\kappa \partial_s \tilde{h} \partial_s^2 \tilde{h} \, ds -\int_{\T^1} (\partial_sg+\tilde{b}_\kappa \tilde{h}) \kappa \partial_s^2 h\,ds
  \end{align*}
  where
  \begin{align*}
    \kappa\int_{\T^1} a_\kappa\partial_s \tilde{h} \partial_s^2 \tilde{h} \, ds &= -\kappa \int_{\T^1} \partial_s a_\kappa\frac{|\partial_s \tilde{h}|^2}2\,ds
                                                                   \leq C\kappa\|\tilde{h}(t)\|_{H^1(\T^1)}^2,\\
    -\int_{\T^1} (\partial_s g+\tilde{b}_\kappa \tilde{h}) \kappa \partial_s^2 \tilde{h}\,ds&\leq C(\|\partial_s g(t)\|_{L^2(\T^1)}+\|\partial_s h(t)\|_{L^2(\T^1)})\kappa \|\partial_s^2 \tilde{h}(t)\|_{L^2(\T^1)}.
  \end{align*}
  Therefore integration in time, \eqref{eq:hEstim1} with $r=1$ and Young's inequality yield \eqref{eq:hEstim2} in the case $r=1$. Finally, the case $r\in (0,1)$ follows again by interpolation.
\end{proof*}

For the construction of the approximate solutions we will essentially use solution to the following linearized system:
    \begin{alignat}{2}\label{eq:linTwoPhase1}
      \partial_t \we_\eps^\pm +\we_\eps^\pm \cdot \nabla \ve_0^\pm+ \ve_0^\pm \cdot \nabla \we_\eps^\pm -\nu^\pm \Delta \we_\eps^\pm + \nabla q_\eps^\pm &= \mathbf{f}^\pm &\quad& \text{in }\Omega_t^{\eps,\pm}, t\in (0,T),\\\label{eq:linTwoPhase2}
      \Div \we_\eps^\pm &= 0 && \text{in }\Omega_t^{\eps,\pm}, t\in (0,T),\\\label{eq:linTwoPhase3}
      \llbracket\we_\eps\rrbracket =0,\quad   \llbracket\nu D\we_\eps-q_\eps\tn{I}\rrbracket\cdot \no_\eps - \sigma (\Delta_{\Gamma_\eps}h_\eps)\circ S_\eps \no_\eps &= 0 & &\text{on }\Gamma_t^\eps, t\in (0,T),\\
      \we_\eps^-|_{\partial\Omega} &= 0 &&\text{on }\partial\Omega\times (0,T),\\\label{eq:linTwoPhase4}
      \we_\eps^\pm|_{t=0} &=\we_0&& \text{in } \Omega^{\eps,\pm}_0,
    \end{alignat}
    together with
      \begin{alignat}{2}\nonumber
        \partial_t h_\eps+(\no_\eps \cdot \we_\eps)\circ X_\eps|_{r=0}+a_\eps \partial_sh_\eps + b_\eps h_\eps  &-\kappa_\eps \Delta_{\Gamma_\eps}h_\eps \\\label{eq:linTwoPhase5}
        &= (\no_\eps\cdot \ue)\circ X_\eps|_{r=0}&\quad &\text{in }\T^1\times (0,T),\\
        h_\eps|_{t=0}& = h_0&&\text{in } \T^1,
  \end{alignat}
  where $\we_\eps^\pm= \we_\eps|_{\Omega^{\eps,\pm}}$, $q_\eps^\pm= q|_{\Omega^{\eps,\pm}}$, $\no_\eps = \nabla d_\eps$ and $\ue\colon \Omega\times (0,T)\to \R^2$ is given. Moreover, $T\in (0,T_0]$
  \begin{alignat*}{2}
    \Omega^{\eps,\pm}_t &:= \big(\Omega_t^\pm\setminus \Gamma_t(3\delta)\big)\cup\{x\in \Gamma_t(3\delta): d_\eps (x,t)\gtrless 0\},\quad \Omega^{\eps,\pm}:= \bigcup_{t\in [0,T]}\Omega^{\eps,\pm}_t\times \{t\},\\
    \Gamma^\eps_t &:= \{x\in \Gamma_t (3\delta): d_\eps(x,t)=0\},\quad \Gamma^\eps:= \bigcup_{t\in [0,T]}\Gamma^\eps_t\times \{t\}
  \end{alignat*}
  and
  \begin{equation*}
    \Delta_{\Gamma_\eps} h_\eps(s,t) :=   (\Delta S_\eps)(x,t)|_{x=X_\eps(0,s,t)}\cdot\partial_s h_\eps(s,t)+|\nabla S_\eps(x,t)|^2|_{x=X_\eps(0,s,t)} \p_{s}^2  h_\eps(s,t)
  \end{equation*}
  for all $s\in \T^1$, $t\in (0,T)$. We note that by chain rule
  \begin{equation*}
    \Delta (h_\eps (S_\eps(x,t),t))= \Delta_{\Gamma_\eps} h_\eps(S_\eps(x,t),t)  \qquad \text{for all }x\in \Gamma(3\delta), t\in (0,T)
  \end{equation*}
  for every sufficiently smooth $h_\eps \colon \T^1\times (0,T)\to \R$.

  More precisely, \eqref{eq:linTwoPhase1}-\eqref{eq:linTwoPhase3} are understood in the following weak sense:
  \begin{align}\nonumber
    &\weight{\partial_t \we_\eps(t),\boldsymbol{\varphi}}_{V(\Omega)',V(\Omega)} + \int_{\Omega} (\we_\eps\cdot \nabla \ve_0+ \ve_0\cdot \nabla \we_\eps) \cdot \boldsymbol{\varphi}\, dx + \sum_{\pm} \int_{\Omega^{\eps,\pm}_t(t)} \nu^\pm D\we_\eps : D\boldsymbol{\varphi}\, dx\\\label{eq:WeakTwoPhaseStokes}
    &\quad = \weight{\mathbf{f}(t),\boldsymbol{\varphi}}_{V(\Omega)',V(\Omega)} + \sigma \int_{\Gamma^\eps_t} (\Delta_{\Gamma_\eps}h_\eps)\circ S_\eps \no_\eps\cdot \boldsymbol{\varphi}\,d\sigma
  \end{align}
  for all $\boldsymbol{\varphi}\in V(\Omega):= H^1_0(\Omega)^2\cap L^2_\sigma(\Omega)$ and almost every $t\in (0,T)$, where as usual
  $$
  L^2_\sigma(\Omega)= \overline{\{\boldsymbol{\varphi}\in C_0^\infty(\Omega)^d: \Div \boldsymbol\varphi=0\}}^{L^2(\Omega)^d}.
  $$

 \begin{thm}\label{thm:LinStokesMCF}
  For $\eps\in (0,1]$ let $\kappa_\eps\in (0,1]$, $a_\eps,b_\eps\colon \T^1\times [0,T]\to \R$ be continuously differentiable with uniformly bounded $C^1$-norms with respect to $\eps\in (0,1]$.
  Then for every $\mathbf{f}\in L^2(0,T;V(\Omega)')^d$, $\ue\in H^1(0,T;V(\Omega)')\cap L^2(0,T;V(\Omega))$, $\we_0\in L^2_\sigma(\Omega)$, and $h_0\in H^1(\T^1)$
  there is a unique solution $\we_\eps\in H^1(0,T;V(\Omega)')\cap L^2(0,T;V(\Omega))$, $h_\eps\in H^1(0,T;L^2(\T^1))\cap L^2(0,T;H^2(\T^1))$ of  \eqref{eq:linTwoPhase4}-\eqref{eq:WeakTwoPhaseStokes}.
  Moreover, there is some $C>0$ independent of $\eps \in (0,1]$, $h$,  $g$, $h_0$, and $T\in (0,T_0]$ such that
  \begin{align}\nonumber
    &\|h_\eps\|_{L^\infty(0,T;H^1)} +\sqrt{\kappa_\eps} \|\partial_s^2 h_\eps\|_{L^2((0,T)\times \T^1)}+\|\we_\eps\|_{H^1(0,T;V(\Omega)')}+ \|\we_\eps\|_{L^2(0,T;H^1)}\\\label{eq:EstimLinStokesMCF1}
    &\ \leq
  C\left(\|\mathbf f\|_{L^2(0,T; V(\Omega)')}+\|\ue\|_{L^2(0,T;H^1)}+\|\partial_t \ue\|_{L^2(0,T;V(\Omega)')}+\|\we_0\|_{L^2}+\|h_0\|_{H^1} \right).
  \end{align}
  Finally, if additionally $\mathbf{f}\in L^2(0,T;L^2(\Omega)^2)$, $\we_0\in V(\Omega)$, and $h_0\in H^{\frac32}(\T^1)$, then
  \begin{align}\nonumber
    &\|h_\eps\|_{H^1(0,T;H^{\frac12})\cap L^2(0,T;H^{\frac52})} +\|\we_\eps\|_{H^1(0,T;L^2)}+ \|\we_\eps\|_{L^2(0,T;H^2(\Omega^{\eps,\pm}_t))}\\\label{eq:EstimLinStokesMCF2}
    &\ \leq
      \frac{C}{\kappa_\eps}\left(\|\mathbf f\|_{L^2((0,T)\times \Omega)}+\|\ue\|_{L^2(0,T;H^1)}+\|\partial_t \ue\|_{L^2(0,T;V(\Omega)')}+\|\we_0\|_{H^1}+\|h_0\|_{H^1} \right)
  \end{align}
\end{thm}
\begin{proof}
  First of all existence of a unique solution $\we_\eps \in H^1(0,T;V(\Omega)')\cap L^2(0,T;V(\Omega))$, $h_\eps \in H^1(0,T; H^{-1}(\T^1))\cap L^2(0,T; H^1(\T^1))$ follows from the standard theory of abstract parabolic evolution equation for the Gelfand triple
  \begin{equation*}
    V= V(\Omega)\times H^1(\T^1),\quad H= L^2_\sigma(\Omega)\times L^2(\T^1),\quad V'= V(\Omega)'\times H^{-1}(\T^1).
  \end{equation*}
  Moreover, $h_\eps \in H^1(0,T;L^2(\T^1))\cap L^2(0,T;H^2(\T^1))$ follows from standard regularity theory since $\no_\eps\cdot \we_\eps \circ X_\eps|_{r=0}\in L^2(0,T; L^2(\T^1))$. Hence it only remains to show the uniform estimates.

  \medskip

  \noindent
  \emph{Proof of \eqref{eq:EstimLinStokesMCF1}:} First of all we can reduce to the case $\ue=0$ simply by replacing $\we$ by $\we-\ue$ in the equations, where $\we_0$ is replaced by $\we_0-\ue|_{t=0}$ and $\mathbf{f}$ has to be replaced by $\mathbf{\tilde{f}}$ defined by
  \begin{align*}
    &\weight{\mathbf{\tilde{f}}(t),\boldsymbol\varphi}:=     \weight{\mathbf{f}(t),\boldsymbol\varphi}\\
    &\quad -
   \weight{\partial_t \ue(t),\boldsymbol{\varphi}}_{V(\Omega)',V(\Omega)} + \int_{\Omega} (\ue\cdot \nabla \ve_0+ \ve_0\cdot \nabla \ue) \cdot \boldsymbol{\varphi}\, dx + \sum_{\pm} \int_{\Omega^{\eps,\pm}_t(t)} \nu^\pm D\ue : D\boldsymbol{\varphi}\, dx
  \end{align*}
  for all $\boldsymbol{\varphi}\in V(\Omega)$ and $t\in (0,T)$.  Adding $\ue$ afterwards to $\we_\eps$ yields the desired solution. Since
  \begin{equation*}
    \|\mathbf{\tilde{f}}\|_{L^2(0,T; V(\Omega)')}\leq \|\mathbf{f}\|_{L^2(0,T; V(\Omega)')}+ C\left(\|\ue\|_{L^2(0,T;H^1)}+\|\partial_t \ue\|_{L^2(0,T;V(\Omega)')}\right)
  \end{equation*}
  one also obtains \eqref{eq:EstimLinStokesMCF1}.

  Now let $\ue\equiv 0$. Choosing $\boldsymbol{\varphi}=\we_\eps$ in \eqref{eq:WeakTwoPhaseStokes} and testing \eqref{eq:linTwoPhase5} with $|\partial_s X_\eps (0,s,t)|\Delta_{\Gamma_\eps} h_\eps$ we obtain
\begin{align*}
  &\frac{d}{dt}\left( \int_\Omega \frac{|\we_\eps(t)|^2}2 \,dx + \int_{\T^1} \frac{|\nabla S_\eps|^2|\partial_s h_\eps(t)|^2}2 |\partial_s X_\eps (0,s,t)|\,ds \right) + \sum_{\pm}\int_{\Omega^\pm} 2\nu^\pm |D\we_\eps|^2\, dx\\
  &\qquad +
  \kappa_\eps  \int_{\T^1} |\partial_s X_\eps (0,s,t)||\nabla S_\eps|^4|\partial_s^2 h_\eps(t)|^2\, ds= \weight{\mathbf{f},\boldsymbol{\varphi}}_{V(\Omega)',V(\Omega)} -\sum_{\pm}\int_{\Omega^\pm(t)}\we^\pm_\eps \cdot \nabla \ve_0^\pm\cdot \we_\eps^\pm\,d x\\
  &\qquad +  \int_{\T^1}(\tilde{a}_\eps \partial_s h_\eps+\tilde{b}_\eps h_\eps)\partial_s^2 h_\eps \, ds +\int_{\T^1} \left(\tilde{c}_\eps\partial_s h_\eps(t)+ \tilde{d}_\eps h_\eps(t)\right)\partial_s h_\eps(t)\, dx\\
  &\quad \leq \left(\|\mathbf{f}\|_{L^2(\Omega)}+ \|\we_\eps\|_{L^2(\Omega)} + \|\mathbf{a}\|_{H^{-\frac12}(\Gamma_t^\eps)} + \|(h_\eps,\partial_s h_\eps)\|_{L^2(\T^1)} \right)\left(\|\we_\eps\|_{H^1(\Omega)}+ \|\partial_s h_\eps\|_{L^2(\T^1)}\right)
\end{align*}
for some smooth and uniformly bounded $\tilde{a}_\eps, \tilde{b}_\eps, \tilde{c}_\eps, \tilde{d}_\eps \colon \T^1\times [0,T]\to \R$, where we note that
\begin{align*}
  \int_{\Gamma_t^\eps} \Delta_{\Gamma_\eps} h_\eps \circ S_\eps \no_\eps \cdot \we_\eps \,d\sigma&=
                                                                                                   \int_{\T^1}(\no_\eps \cdot \we_\eps)(X_\eps(0,s,t)) \Delta_{\Gamma_\eps} h_\eps(s,t)   |\partial_s X_\eps (0,s,t)| \,ds \quad\text{and}\\
  \int_{\T^1}(\tilde{a}_\eps \partial_s h_\eps+\tilde{b}_\eps h_\eps)\partial_s^2 h_\eps \, ds &= -\int_{\T^1}\left((\partial_s\tilde{a}_\eps) \frac{|\partial_s h_\eps|^2}2+\partial_s(\tilde{b}_\eps h_\eps)\partial_s h_\eps\right) \, ds.
  \end{align*}
  Hence Young's and Gronwall's inequality yield the desired estimate \eqref{eq:EstimLinStokesMCF1}.

  \medskip
  
  \noindent
  \emph{Proof of \eqref{eq:EstimLinStokesMCF2}:} Now assume additionally that $\mathbf{f}\in L^2(0,T;L^2(\Omega)^2)$, $\we_0\in V(\Omega)$ and $h_0\in H^{\frac32}(\T^1)$. (Note that we do not reduce to the case $\ue\equiv 0$ in this case since this is not compatible with the assumed regularity for $\ue$.)  The estimate of $\|h_\eps\|_{H^1(0,T;H^{\frac12}(\T^1))\cap L^2(0,T;H^{\frac52}(\T^1))}$ follows directly from
  \eqref{eq:hEstim2} for $r=\frac12$ and \eqref{eq:EstimLinStokesMCF1} using the equation \eqref{eq:linTwoPhase5}. Hence
  \begin{align*}
    &\|\Delta_{\Gamma_\eps}h\circ S_\eps\|_{H^{\frac14}(0,T;L^2(\Gamma_t^{\eps}))\cap L^2(0,T;H^\frac12(\Gamma_t^{\eps})) }\\
    &\quad \leq \frac{C}{\kappa_\eps}\left(\|\mathbf f\|_{L^2((0,T)\times \Omega)}+\|g\|_{H^{\frac14}(0,T;L^2)\cap L^2(0,T;H^\frac12)}+\|\we_0\|_{H^1}+\|h_0\|_{H^1} \right)
  \end{align*}
  Now the estimate of $\|\we_\eps\|_{H^1(0,T;L^2)}+ \|\we_\eps\|_{L^2(0,T;H^2(\Omega^{\eps,\pm}_t))}$ follows from standard estimates for the two-phase Stokes system, cf.~e.g.~\cite{PruessSimonettMovingInterfaces} for the case that the interface $\Gamma^\eps_t$ is independent of $t\in (0,T)$. The result in the present case that $\Gamma_t^\eps$ evolves smoothly with respect to $t$ can be shown by the same perturbation argument as in the proof of Theorem~\ref{solvingkorder} in the appendix.
\end{proof}

\subsection{Spectral Estimate}\label{subsec:spectral_estimate}
For the spectral estimate in $\eps$-coordinates as in Section \ref{subsec:Coordinates} let $\eps_1>0$ be as in Theorem~\ref{th_coord} and assume that \eqref{eq:condSepsneps} hold true.
Moreover, we consider the rescaled variable
\begin{align}
	\rho_\eps(x,t):=\frac{d_\eps(x,t)}{\eps}\quad\text{ for }(x,t)\in\overline{\Gamma(3\delta)}.
\end{align}
Finally, we assume the following structure of the approximate solution: let
\begin{align}
	\tilde{c}^A_\eps(x,t):=\theta_0(\rho_\eps(x,t))+ \eps^\mu p_\eps(S_\eps,t)\theta_1(\rho_\eps)+O(\eps^2),
\end{align}
where $\mu\in[1,2)$ and $p_\eps:\T^1\times[0,T_0]\rightarrow\R$ is measurable with $\|p_\eps\|_\infty \leq C$ and $\theta_1\in L^\infty(\R)$ with
\begin{align}
\int_\R f'''(\theta_0)(\theta_0')^2\theta_1\,d\rho=0.
\end{align}
We set
\begin{align}\label{eq_cA_eps_def}
c^A_\eps:=\zeta(\frac{d_0}{\delta}) \tilde{c}^A_\eps+ (1-\zeta(\frac{d_0}{\delta}))(\chi_{\Omega^+(t)}-\chi_{\Omega^-(t)})\quad\text{ in }\Omega,
\end{align}
where $\zeta:\R\rightarrow\R$ is a smooth cutoff-function with $\zeta=1$ for $|r|\leq 2$ and $\zeta=0$ for $|r|\geq \frac52$.

The following spectral estimate will be a key ingredient for the proof of convergence.
\begin{lem}[Spectral Estimate]\label{thm:Spectral}~\\
Let the above assumptions in this section hold. Then there are some uniform $C_L,c_L>0$, $\eps_0\in (0,\eps_1]$ such that for every $\psi\in H^1(\Om)$, $t\in[0,T_0]$, and $\eps\in (0,\eps_0]$ we have
  \begin{equation*}
    \int_{\Om}\left(|\nabla\psi|^2+ \frac{f''(c^A_\eps(.,t))}{\eps^2}\psi^2\right)\,dx\geq -C_L\|\psi\|_{L^2(\Om)}^2 + \|\nabla \psi\|_{L^2(\Om\backslash\Gamma^\eps_t(\frac{3\delta}{2}))}^2 + c_L\|\nabla_{{\btau_\eps}}\psi\|_{L^2(\Gamma^\eps_t(\frac{3\delta}{2}))}^2,
  \end{equation*}
where $\nabla_{\btau_\eps} \psi$ is as in \eqref{eq:defnTauEps} and $\Gamma^\eps_t(\frac{3\delta}{2}):=\Gamma^\eps(\frac{3\delta}{2}))\cap\left(\R^2\times\{t\}\right)$ with $X_\eps$, $\Gamma^\eps(\frac{3\delta}{2}))$, and $\eps_1$ are as in Theorem \ref{th_coord}.
\end{lem}
\begin{rem}
	Because of \eqref{eq_Gamma_eps} and $|\nabla\psi|\geq c|\nabla_{\tau_\eps}\psi|$ (see e.g.~the proof below), the result also holds for $\Gamma_t(2\delta)$ instead of $\Gamma^\eps_t(\frac{3\delta}{2})$ with a possibly smaller constant $c_L$.
\end{rem}
\begin{proof*}{of Lemma~\ref{thm:Spectral}}
	First, due to \eqref{eq_Gamma_eps} and the definition \eqref{eq_cA_eps_def} of $c^A_\eps$, we obtain that for $\eps>0$ small it holds
	\[
	f''(c^A_\eps)\geq 0\quad\text{ in }\Gamma_\eps(\tfrac{3\delta}{2}).
	\]
	Therefore let us first consider the integral over $\Gamma_\eps(\frac{3\delta}{2})$: we can transform it into $(d_\eps,S_\eps)$-coordinates and get
	\[
	\int_{\Gamma_\eps(\frac{3\delta}{2})}\left(|\nabla\psi|^2+ \frac{f''(c^A_\eps)}{\eps^2}\psi^2\right)\sd x=\int_{\T^1}\int_{-\frac{3\delta}{2}}^{\frac{3\delta}{2}}\left[|\nabla\psi|^2\circ X_\eps + \frac{f''(\tilde{c}^A_\eps)}{\eps^2}\psi_\eps^2\right] J_\eps\,dr\,ds,
	\]
	where we have set $\psi_\eps:=\psi\circ X_\eps$ and $J_\eps$ is defined in \eqref{J_eps}. Via the chain rule we have the following transformation identity:
	\begin{align}
		|\nabla\psi|^2\circ X_\eps = (\nabla_{(r,s)}\psi_\eps)^\top
		\left.\begin{pmatrix}
			|\nabla d_\eps|^2 & \nabla d_\eps\cdot\nabla S_\eps\\
			\nabla d_\eps\cdot\nabla S_\eps & |\nabla S_\eps|^2
		\end{pmatrix}\right|_{X_\eps}
	\nabla_{(r,s)}\psi_\eps.
	\end{align}
    Therefore the asymptotics \eqref{eq:condSepsneps} together with Young's inequality yields for $\eps$ small
    \[
    |\nabla\psi|^2\circ X_\eps \geq (1-C\eps^2)|\partial_r\psi_\eps|^2+\frac{1}{2}|\nabla_{\tau_\eps}\psi|^2\circ X_\eps.
    \]
    Altogether we obtain
    \begin{align}\notag
    \int_{\Omega}\left(|\nabla\psi|^2+ \frac{f''(c^A_\eps(.,t))}{\eps^2}\psi^2\right)\sd x&\geq \|\nabla\psi\|_{L^2(\Omega\backslash\Gamma_t^\eps(\frac{3\delta}{2}))}^2+ \frac{1}{2} \|\nabla_{\tau_\eps}\psi\|_{L^2(\Gamma_t^\eps(\frac{3\delta}{2}))}^2 - C\|\psi\|^2_{L^2(\Gamma_t^\eps(\frac{3\delta}{2}))}+\\
    + (1-C\eps^2) &\int_{\T^1}\int_{-\frac{3\delta}{2}}^{\frac{3\delta}{2}}\left[|\partial_r\psi_\eps|^2 + \frac{f''(\tilde{c}^A_\eps(.,t))}{\eps^2}\psi_\eps^2\right] J_\eps\,dr\,ds.\label{eq_spectral_proof}
    \end{align}
    The last term on the right hand side of \eqref{eq_spectral_proof} can be treated by well-known scaling and perturbation arguments as well as the spectral properties of differential operators on the real line similar to Chen \cite{ChenSpectrumAC}. More precisely, except for the dependency of $J_\eps$ on $\eps$ (which is not a problem because it is uniform in $\eps$) the abstract 1D-spectral estimates in Moser~\cite[Section 5.1.3]{MoserAdv} are applicable after rescaling and yield the desired estimate. This shows the spectral estimate in Lemma \ref{thm:Spectral}.
\end{proof*}

	Furthermore, we need more refined estimates of spectral decomposition type:
\begin{cor}\label{cor:SpectralDecomp}
 Let the previous assumptions be valid and let
$t\in\left[0,T_0\right]$ and $\psi\in H^1(\Gamma_t^\eps(\frac{3\delta}{2}))$, where $\Gamma_t^\eps(\frac{3\delta}{2})=\{x\in \Omega:(x,t)\in\Gamma^\eps(\frac{3\delta}{2})\}$ with $\Gamma^\eps(\frac{3\delta}{2})$ from Theorem \ref{th_coord}. Moreover, let $\Lambda_{\eps}\in\mathbb{R}$ be such that
\begin{equation}
\label{def Lambda energy}
\int_{\Gamma_t^\eps(\frac{3\delta}{2})} \left|\nabla\psi\right|^{2}+\frac{f''\left(c^A_\eps(.,t)\right)}{\eps^2}\psi^{2}\sd x\leq\frac{\Lambda_{\eps}}{\eps}
\end{equation}
and denote $I_{\eps}:=\left(-\frac{3\delta}{2\eps},\frac{3\delta}{2\eps}\right)$. Then, for $\eps>0$ small enough, there exist functions $Z_\eps\in H^{1}(\T^1)$,
$\psi^{\mathbf{R}}_\eps\in H^{1}(\Gamma_t^\eps(\frac{3\delta}{2}))$
and smooth $\Psi_\eps\colon I_{\eps}\times\T^1\to \R$ such that
\begin{equation}
\psi(x)=\eps^{-\frac{1}{2}}Z_\eps(s)\left(\beta_\eps\theta_{0}'(\rho)+\Psi_\eps(\rho,s)\right)+\psi^{\mathbf{R}}_\eps(x) \quad \text{ for all } x\in\Gamma^\eps_t(\tfrac{3\delta}{2}) \label{decompose u}
\end{equation}
where
\begin{equation}\label{eq_variable_abbrev}
  s=S_\eps(x,t),\quad \rho= \frac{d_\eps(x,t)}\eps
\end{equation}
for almost all $x\in \Gamma_t^\eps(\frac{3\delta}{2})$ 
and $\beta_\eps:=\|\theta_0'\|_{L^2(I_\eps)}^{-1}$. Moreover,
\begin{equation}
\Vert \psi^{\mathbf{R}}_\eps\Vert _{L^{2}\left(\Gamma_t^\eps(\frac{3\delta}{4})\right)}^{2}\leq C\left(\eps\Lambda_{\eps}+\eps^{2}\left\Vert \psi\right\Vert _{L^{2}\left(\Gamma_t^\eps(\frac{3\delta}{2})\right)}^{2}\right),\label{f2u estimate-1}
\end{equation}
\begin{equation}\label{f2u estimate-2}
\Vert Z_\eps\Vert _{H^{1}(\T^1)}^{2}+\Vert \nabla_{\btau_\eps}\psi\Vert _{L^{2}(\Gamma_t^\eps(\frac{3\delta}{2}))}^{2}+\Vert \psi^{\mathbf{R}}_\eps\Vert _{H^{1}\left(\Gamma_t^\eps(\frac{3\delta}{2})\right)}^{2}
\le C\left(\left\Vert \psi\right\Vert _{L^{2}\left(\Gamma_t^\eps(\frac{3\delta}{2})\right)}^{2}+\frac{\Lambda_{\eps}}{\eps}\right),
\end{equation}
and with $J_\eps$ from \eqref{J_eps}
\begin{equation}\label{f1u estimate}
\sup_{s\in\T^1}\left(\int_{I_{\eps}}\left(\Psi_\eps(\rho,s)^{2}+\partial_\rho\Psi_\eps(\rho,s)^{2}\right)J_\eps\left(\eps\rho,s,t\right)\sd \rho\right)\leq C\eps^{2}.
\end{equation}
\end{cor}
\begin{proof}
	One can proceed similar to Abels, Marquardt \cite[Corollary~2.12]{AbelsMarquardt1}. Here one uses the transformation into the $\eps$-coordinates from Theorem \ref{th_coord} and spectral properties of 1D-differential operators on the real line similar to Chen \cite{ChenSpectrumAC}, cf.~also Moser \cite[Section 5.1.3]{MoserAdv}. This yields the result with $\|Z_\eps J_\eps(0,.,t)^{\frac{1}{2}}\|_{H^1(\T^1)}$ instead of $\|Z_\eps\|_{H^1(\T^1)}$ on the left hand side of \eqref{f2u estimate-2}. However, the additional factor is not a problem, because one can control $J_\eps(0,.,t)^{-\frac{1}{2}}$ in $C^1(\T^1)$ independent of $\eps$ using the form \eqref{J_eps} and the assumptions on $d_\eps$ and $S_\eps$. Hence with the chain rule we obtain $\|Z_\eps\|_{H^1(\T^1)}\leq C \|Z_\eps J_\eps(0,.,t)^{\frac{1}{2}}\|_{H^1(\T^1)}$ and the result follows.
\end{proof}

\begin{rem}
  For $u\in H^1(\Gamma^\eps_t(\tfrac{3\delta}{2}))$ let us introduce the $\eps$-dependent norms
  \begin{align*}
    \|u\|_{V_\eps} = &\inf\left\{ \|Z\|_{H^1(\T^1)}+ \|v\|_{H^1(\Gamma^\eps_t(\frac{3\delta}{2}))}+\frac{1}{\eps}\|v\|_{L^2(\Gamma^\eps_t(\frac{3\delta}{2}))}: Z\in H^1(\T^1), v\in H^1(\Gamma^\eps_t(\tfrac{3\delta}{2})),\right.\\
    &\quad \qquad \left. u(x)=\frac{1}{\sqrt{\eps}} Z(s)\theta_0'(\rho) + v(x)\text{ for all }x\in\Gamma^\eps_t(\tfrac{3\delta}{2}) \right\},
  \end{align*}
  with the abbreviations from \eqref{eq_variable_abbrev}. Corollary~\ref{cor:SpectralDecomp} yields
  \begin{equation*}
    \|u\|_{V_\eps}^2+\|\nabla_{\tau_\eps} u \|_{L^2(\Gamma^\eps_t(\frac{3\delta}{2}))}^2 \leq C\left(\int_{\Gamma^\eps_t(\frac{3\delta}{2})} |\nabla u|^2 +\frac1{\eps^2} f''(c^A_\eps(.,t))u^2\, dx + \|u\|_{L^2(\Gamma^\eps_t(\frac{3\delta}{2}))}^2 \right).
  \end{equation*}
  Here note that $\beta_\eps$ from Corollary~\ref{cor:SpectralDecomp} is bounded uniformly for $\eps$ small and in order to obtain the estimate one has to take care of the $\Psi_\eps$-term from Corollary \ref{cor:SpectralDecomp}. However, this can be done by using the estimates in Corollary \ref{cor:SpectralDecomp} and a rescaling argument.

  We note that for every $\eps>0$ the norm $\|.\|_{V_\eps}$ is equivalent to the standard norm in $H^1(\tfrac{3\delta}{2}))$ (with $\eps$-dependent constants). For the estimates of some critical remainder terms the choice of this norm will be essential. To estimate such remainder terms the following lemma will be used.
\end{rem}
\begin{lem}\label{lem:EstimMeanValueFree}
  Fix $t\in[0,T_0]$. Let $u\in H^1(\Gamma^\eps_t(\frac{3\delta}{2}))$ and $r_\eps\colon \Gamma^\eps_t(\frac{3\delta}{2})\to \R$ be a finite sum of terms of the form
    \begin{equation*}
    a(\rho) w_\eps(S_\eps),
    \end{equation*}
	where $a\in \mathcal{R}_{0,\alpha}$, $w_\eps\in L^2(\T^1)$ and such that
    \begin{equation}\label{eq:MeanValueFree}
    \int_\R r_\eps(\rho,s,t) \theta'_0(\rho) \, d\rho=0\qquad \text{for all }s\in\T^1.
  \end{equation}
    Then there are constants $C>0$, $\eps_0\in (0,\eps_1]$ independent of $t\in [0,T_0]$ such that for $\eps\in (0,\eps_0]$
  \begin{equation*}
  \left|\int_{\Gamma^\eps_t(\frac{3\delta}{2})} r_\eps u\,dx\right| \leq C\eps^{\frac{3}{2}}\|w_\eps\|_{L^2(\T^1)}\|u\|_{V_\eps}.
  \end{equation*}
\end{lem}
\begin{proof}
  This can be done in the analogous way as in \cite[Lemma 2.11]{AbelsFei}.
\end{proof}
\begin{rem}
  If we define dual norm
  \begin{equation*}
    \|f\|_{V_\eps'} := \sup_{\|\varphi\|_{V_\eps}\leq 1} |\weight{f,\varphi}|\qquad \text{for }f\in (H^1(\tfrac{3\delta}{2})))',
  \end{equation*}
  the lemma states that
  \begin{equation*}
    \left\|\int_{\Gamma^\eps_t(\frac{3\delta}{2})} g_\eps \cdot\,dx\right\|_{V_\eps'} \leq C\eps^{\frac{3}{2}}\|w_\eps\|_{L^2(\T^1)}.
  \end{equation*}
\end{rem}

\section{Formally Matched Asymptotics}\label{sec:MatchedAsymptotics}

In this section we will discuss the construction of the approximate solutions except some higher order terms, which will be added in the next section. In comparision with previous works the main difference is that we obtain an expansion in terms of integer powers of $\eps^{\frac12}$, which means we consider expansions in terms of $\eps^k$ with $k\in \frac12\N_0$.

First of all we note that
\begin{align*}
\Div(\nabla c_\e\otimes\nabla c_\e)=\frac{1}{2}\nabla\big(|\nabla c_\e|^2\big)+\Delta c_\e\nabla c_\e.
\end{align*}
Therefore we  can rewrite \eqref{eq:NSAC1}-\eqref{eq:NSAC3} as
\begin{alignat}{2}\label{eq:NSAC1-new}
  \partial_t \ve_\e +\ve_\e\cdot \nabla \ve_\e-\Div(2\nu(c_\e)D\ve_\e)  +\nabla p_\varepsilon & = -\varepsilon\Delta c_\e\nabla c_\e,\\\label{eq:NSAC2-new}
  \Div \ve_\e& = 0,
  \\ \label{eq:NSAC3-new}
 \partial_t c_\e +\ve_\e\cdot \nabla c_\e & =\e^{\frac{1}{2}}\Delta c_{\epsilon}-\e^{-\frac{3}{2}}f'(c_\e)
\end{alignat}
in $\Omega\times (0,T_0)$ by replacing $p_\eps$ by $p_\eps+\frac12|\nabla c_\eps|^2$.

\subsection{The Outer Expansion \label{subsec:The-Outer-Expansion}}

We assume that in $\Omega^{\pm}\backslash\Gamma$
the solutions of  (\ref{eq:NSAC1-new})-(\ref{eq:NSAC3-new})
have the expansions
\begin{align*}
  c_{\eps}(x,t)  &\approx\sum_{k\in\frac12\N_0}\eps^{k}c_{k}^{\pm}(x,t), \  \ \mathbf{v}_{\eps}(x,t) \approx\sum_{k\in\frac12\N_0}\eps^{k}\mathbf{v}_{k}^{\pm}(x,t),\ \\
  p_{\eps}(x,t)  &\approx\sum_{k\in\frac12\N_{-2}}\eps^{k}p_{k}^{\pm}(x,t),
\end{align*}
where $\N_{-2}=\N_0\cup \{-1,-2\}$ and  $c_{k}^{\pm}$, $\mathbf{v}_{k}^{\pm}$ and
$p_{k}^{\pm}$ are smooth functions defined in $\Omega^{\pm}$. Here $\varphi_\eps(x,t)\approx \sum_{k\geq 0,k\in\frac12\N_0} \eps^k \varphi_k^\pm (x,t)$ for $\varphi_\eps=c_\eps, \ve_\eps$ (analogously for $p_\eps$) is understood in the sense that for any $N\in \frac12\N_0$ we have
\begin{equation*}
  \varphi_\eps(x,t)- \sum_{ k\in\frac12\N_0,k\leq N} \eps^k\varphi^\pm_k(x,t) = O(\eps^{N+\frac12}) \qquad \text{in }\Omega^\pm
\end{equation*}
and the same if $\varphi_\eps$ and $\varphi_k$ are replaced by $\partial_t^j\partial_x^\alpha \varphi_\eps$ and $\partial_t^j\partial_x^\alpha \varphi_k^\pm$, respectively, for any $j\in\N_0$, $\alpha \in \N_0^n$.

Plugging this ansatz into \eqref{eq:NSAC1-new}-\eqref{eq:NSAC3-new} and \eqref{eq:NSAC4}, using a Taylor expansion for $f'$ and $\nu$, the Dirichlet boundary condition for $c_\eps$, and matching the $\mathcal{O}(\eps^{-1}),\mathcal{O}(\eps^{-\frac12})$ terms one obtains in a standard manner (cf. e.g. \cite[Appendix]{AbelsFei})
\begin{equation}
c_{0}^{\pm}=\pm1, \quad c_k^\pm = 0 \quad \text{for }k\geq \frac12, \quad\nabla p_{-1}^{\pm}=\nabla p_{-\frac12}^{\pm}  =0\qquad \text{in }\Omega^\pm\label{eq:c0out}
\end{equation}
and
\begin{alignat}{2}
 \partial_t\mathbf{v}_{k}^{\pm}+\mathbf{v}_{0}^{\pm}\cdot\nabla \mathbf{v}_{k}^{\pm}+\mathbf{v}_{k}^{\pm}\cdot\nabla \mathbf{v}_{0}^{\pm}  -\nu^{\pm}\Delta\ve_{k}^{\pm}+\nabla p_{k}^{\pm}  &=-\sum_{j\in\frac12 \N, \frac12\leq j\leq k-\frac12}\mathbf{v}_{j}^{\pm}\cdot\nabla \mathbf{v}_{k-j}^{\pm}&\ & \text{in }\Omega^\pm, \label{eq:stokes-Outerk}\\
 \operatorname{div}\mathbf{v}_{k}^{\pm} & =0&\ & \text{in }\Omega^\pm,\label{eq:Outervpdefine}\\
 \label{bccondition:vek}
{\ve_k^-|_{\partial\Omega}}&= 0&\ & \text{on }\partial\Omega
\end{alignat}
for every $k\geq 0$.
For simplicity we take
\begin{align}
 p_{-1}^{\pm}= p_{-\frac12}^{\pm} =0.\label{expression:pressure-Outer-1}
\end{align}

\begin{rem}
  \label{Outer-Rem}
  As in \cite{AbelsFei,StokesAllenCahn} we extend $(c_{k}^{\pm},\mathbf{v}_{k}^{\pm},p_{k}^{\pm})$, $k\geq 0$, defined on $\Omega^{\pm}$, to $\Omega^{\pm}\cup\Gamma(3\delta)$ such that $\Div \ve_k^{\pm}=0$ in $\Omega^{\pm}\cup\Gamma(3\delta)$ for all $k\geq 0$. We refer to \cite[Remark~A.1]{AbelsFei} for the details.
\end{rem}

For the following we define
\begin{align}
  \mathbf{W}_{k}^{\pm}(x,t) & =\partial_t\mathbf{v}_{k}^{\pm}(x,t)+\sum_{j\in \frac12 \N_0, j\leq k}\mathbf{v}_{k-j}^{\pm}\cdot\nabla \mathbf{v}_{j}^{\pm}-\nu^\pm \Delta\mathbf{v}_{k}^{\pm}(x,t)+\nabla p_{k}^{\pm}(x,t) 
  \label{eq:Wkpm}
\\ \mathbf{W}^{\pm}  &=\sum_{k\in \frac12\N_0}\eps^{k}\mathbf{W}_{k}^{\pm},\label{eq:Wpm}
\end{align}
for $(x,t)\in\Omega^{\pm}\cup\Gamma(2\delta)$.
Because of \eqref{eq:stokes-Outerk}, it holds $\mathbf{W}_{k}^{\pm}(x,t)=0$ for all $(x,t)\in\overline{\Omega^{\pm}}.$

\subsection{The Inner Expansion \label{subsec:The-Inner-Expansion}}

Close to the interface $\Gamma$ we introduce a stretched variable
\begin{equation}
\rho=\rho(x,t):=\frac{d_\eps(x,t)}{\eps}\qquad \text{for all }(x,t)\in\Gamma(3\delta)\label{eq:rhoeps}
\end{equation}
for $\eps\in\left(0,1\right)$, where $d_{\varepsilon}\colon\Gamma(3\delta)\to \R$ satisfies
\begin{equation}
  |\nabla d_\eps(x,t)|^2 \approx 1\qquad \text{for all }(x,t)\in\Gamma(3\delta),\label{distanceapp}
\end{equation}
which has to be understood as $|\nabla d_\eps(x,t)|^2=1+ O(\eps^{N+\frac12})$ for any $N\in \frac12\N_0$ similary as before.
Formally, $d_\eps$ is the signed distance function to $\Gamma^\varepsilon$, which is the $0$-level set of $c_{\eps}$. Moreover, we assume the asymptotic expansion
\begin{equation*}
  d_\eps(x,t) \approx \sum_{k\in\frac12\N_0} \eps^kd_k(x,t) \qquad \text{for }(x,t)\in\Gamma(3\delta)
\end{equation*}
understood in the same way as before, where $d_0(x,t)=d_{\Gamma_t}(x)$ for all $(x,t)\in\Gamma(3\delta)$. Here and in the following we assume already that $(\ve_0^\pm,p_0^\pm,\Gamma)$ is a smooth solution of \eqref{eq:Limit1}-\eqref{eq:Limit7}, although these equations can also be derived throughout the formal expansion.
Since for the asymptotic expansion as $\eps \to 0$ only small values of $\eps>0$ matter, we may assume that
\begin{equation}\label{assump:deps}
|d_\eps(x,t)-d_0(x,t)-\sqrt{\eps} d_{1/2}(x,t)|\leq M_0\eps\qquad\text{for all }(x,t)\in \Gamma(3\delta)
\end{equation}
for some $M_0>0$.
Moreover, we choose $\eta\colon \mathbb{R}\rightarrow[0,1]$ such
that $\eta=0$ in $(-\infty,-1]$, $\eta=1$ in $[1,\infty)$, $\eta-\frac{1}{2}$ is odd
and $\eta'\geq0$ in $\mathbb{R}$. Then we have by integration by parts
\begin{equation}\label{eq:defbarnu}
  \int_{-\infty}^{+\infty}\nu(\theta_0)\eta'(\rho)d\rho = \left.\left[\nu(\theta_0(\rho)) \left(\eta(\rho)-\frac12\right)\right]\right|_{\rho=-\infty}^\infty= \frac{\nu^++\nu^-}2=: \overline{\nu}
\end{equation}
since $\nu'(\theta_0)$ is even by the assumptions on $\nu'$. Furthermore, we define
\begin{align*}
\eta^{\eps,\pm}(\rho, x,t)&=\eta(-M-1\pm\rho\mp \eps^{-\frac12}d_{1/2}(x,t))\quad \text{for }\rho\in\R, (x,t)\in \Gamma(3\delta).
\end{align*}
\begin{rem}
  \label{WU}  In the following we will insert terms $\mathbf{W}^{\pm}\eta^{\eps,\pm}$ in the equation to ensure some matching conditions.
We have to make sure that these terms vanish if $\rho= \frac{d_\eps(x,t)}\eps$.
Because of \eqref{assump:deps}, we have for $\rho=\frac{d_{\eps}(x,t)}{\eps}$
and $(x,t)\in\Gamma(3\delta)$ with $d_{\Gamma}(x,t)\geq0$
that $\rho-\eps^{-\frac12}d_{1/2}(x,t)\geq-M$. Hence $\eta^{\eps,-}(\rho,x,t)=0$
and, since $(x,t)\in\overline{\Omega^{+}}$, we have $\mathbf{W}^{+}(x,t)=0$. Altogether
\[
\mathbf{W}^{+}\eta^{\eps,+}+\mathbf{W}^{-}\eta^{\eps,-}=0\quad \text{in }\overline{\Omega^+}.
\]
In the same way one shows this in  $\overline{\Omega^-}$.
\end{rem}

For the inner expansion we use the ansatz
\begin{align*}
  c_{\eps}(x,t)  =\hat{c}_{\eps}\big(\rho,x,t\big),\quad
\mathbf{v}_{\eps}(x,t)=\hat{\mathbf{v}}_{\eps}\big(\rho,x,t\big),\quad
p_{\eps}(x,t)  =\hat{p}_{\eps}\big(\rho,x,t\big)
\end{align*}
for $(x,t)\in \Gamma(3\delta)$, where $\rho\in\R $ is as in \eqref{eq:rhoeps}, and use it in the expansion of \eqref{eq:NSAC1-new}-\eqref{eq:NSAC3-new}
for smooth $\hat{c}_{\eps},\hat{p}_{\eps}\colon \mathbb{R}\times\Gamma(3\delta)\rightarrow\mathbb{R}$,
$\hat{\mathbf{v}}_{\eps}\colon \mathbb{R}\times\Gamma(3\delta)\rightarrow\mathbb{R}^{2}$.

We use that
\begin{align}\nonumber
&\Div(2\nu(c_\eps)D\ve_\eps)= \eps^{-2}\partial_{\rho}\big(\nu(\hat{c}_\eps)\partial_{\rho}\hat{\ve}_\eps\big)-\eps^{-1}\partial_{\rho}\big(\nu(\hat{c}_\eps)\Div\hat{\ve}_\eps\big)\nabla d_{\eps}\\
&\qquad+\eps^{-1}\partial_{\rho}\big(2\nu(\hat{c}_\eps)D\hat{\ve}_\eps\big)\big)\cdot\nabla d_{\eps}
+\eps^{-1}\Div\big(2\nu(\hat{c}_\eps)D_d\hat{\ve}_\eps\big)+\Div\big(2\nu(\hat{c}_\eps)D\hat{\ve}_\eps\big),\label{divexpress}
\end{align}
provided \eqref{eq:NSAC2-new} holds,  where
\begin{align*}
 D\hat{\ve}_\eps&=\frac 12(\nabla\hat{\ve}_\eps+(\nabla\hat{\ve}_\eps)^T),\quad D_d\hat{\ve}_\eps=\frac 12\big(\partial_\rho\hat{\ve}_\eps\otimes \nabla d_{\eps}+(\partial_\rho\hat{\ve}_\eps\otimes\nabla d_{\eps})^T\big).
\end{align*}
Moreover, we note that
\begin{align*}
\nabla c_{\varepsilon}&=\varepsilon^{-1}\partial_{\rho}\cte\nabla d_{\varepsilon}+\nabla \cte,
    \\ \Delta c_{\varepsilon}&=\varepsilon^{-2}\partial_{\rho\rho}\cte
+2\varepsilon^{-1}\nabla \partial_{\rho}\cte\cdot\nabla d_{\varepsilon}+\varepsilon^{-1}\partial_{\rho}\cte\Delta d_{\varepsilon}+\Delta\cte,
\end{align*}
where we have used \eqref{distanceapp}. This yields
\begin{align}
\varepsilon\Delta c_{\varepsilon}\nabla c_{\varepsilon}
 =\varepsilon^{-2}\partial_{\rho}\cte\partial_{\rho\rho}\cte\nabla d_{\varepsilon}
+\varepsilon^{-1}\mathbb{A}_{\varepsilon}+\mathbb{B}_{\varepsilon}+\varepsilon\mathbb{C}_{\varepsilon},\label{delta-nabla}
\end{align}
where
\begin{align}
\mathbb{A}_{\varepsilon}&=\partial_{\rho\rho}\cte\nabla \cte+2\nabla \partial_{\rho}\cte\cdot\nabla d_{\varepsilon}\partial_{\rho}\cte\nabla d_{\varepsilon}
+\big(\partial_{\rho}\cte\big)^2\Delta d_{\varepsilon}\nabla d_{\varepsilon},\nonumber \\
\mathbb{B}_{\varepsilon}&=2\nabla \partial_{\rho}\cte\cdot\nabla d_{\varepsilon}\nabla\cte+\partial_{\rho}\cte\Delta d_{\varepsilon}\nabla \cte+\partial_{\rho}\cte\Delta \cte\nabla d_{\varepsilon},
\nonumber \\
\mathbb{C}_{\varepsilon}&=\Delta\cte\nabla\cte.\nonumber
\end{align}
Hence for $\rho$ as in \eqref{eq:rhoeps} the system  (\ref{eq:NSAC1-new})-(\ref{eq:NSAC3-new}) is equivalent to
\begin{align}
\partial_{\rho}\big(\nu(\cte)\partial_{\rho}\hat{\ve}_\eps\big)=& 2\partial_{\rho}\cte\partial_{\rho\rho}\cte\nabla d_{\eps}
+ \eps\bigg(\partial_{\rho}\hat{\ve}_\eps\partial_{t}d_{\eps}+\hat{\ve}_\eps\cdot\nabla d_{\eps}\partial_{\rho}\hat{\ve}_\eps+\partial_{\rho}\hat{p}_\eps\nabla d_{\eps}+\mathbb{A}_{\eps}\nonumber\\
&\quad-\partial_{\rho}\big(2\nu(\hat{c}_\eps)D\hat{\ve}_\eps\big)\cdot\nabla d_{\eps}
-\Div\big(2\nu(\hat{c}_\eps)D_d\hat{\ve}_\eps\big)+\partial_{\rho}\big(\nu(\hat{c}_\eps)\Div\hat{\ve}_\eps\big)\nabla d_{\eps}\bigg)
              \nonumber\\
 &+\eps^2\bigg(\partial_{t}\hat{\ve}_\eps+\hat{\ve}_\eps\cdot\nabla \hat{\ve}_\eps +\nabla\hat{p}_\eps
 -\Div\big(2\nu(\hat{c}_\eps)D\hat{\ve}_\eps\big)+\mathbf{W}^{+}\eta^{\eps,+}+\mathbf{W}^{-}\eta^{\eps,-}\bigg)
              \nonumber\\&+\eps^2\mathbb{B}_{\eps}+\eps^3\mathbb{C}_{\eps} +\big(\nu(\theta_0)\eta'(\rho)\big)'\mathbf{u}_{\eps}\big(d_\eps-\eps\rho\big) +\eps\mathbf{l}_{\eps}\eta'(\rho)\big(d_\eps-\eps \rho\big),\label{modieq:innerstokes} \\
\partial_{\rho}\hat{\ve}_\eps\cdot\nn_\eps=&-\eps\Div_x\hat{\ve}_\eps+\mathbf{u_{\eps}}\cdot \nabla d_{\eps}\eta'(\rho)\big(d_{\eps}-\eps \rho\big),\label{eq:innerdiv}\\
\partial^2_{\rho}\hat{c}_\eps-f'(\cte)=&\eps^{\frac12}\bigg(\partial_{\rho}\hat{c}_\eps\partial_t d_{\eps}+\partial_{\rho}\hat{c}_\eps\hat{\ve}_\eps\cdot\nabla d_{\eps}\bigg)
-\eps\bigg(\partial_{\rho}\hat{c}_\eps\Delta d_{\eps}+2\nabla\partial_{\rho}\hat{c}_\eps\cdot\nabla d_{\eps}\bigg)
\nonumber\\&+\eps^{\frac32}\bigg(\partial_t\hat{c}_\eps+\hat{\ve}_\eps\cdot\nabla\hat{c}_\eps\bigg)
-\eps^2\Delta\hat{c}_\eps+\varepsilon^{\frac{1}{2}}\hat{g}_{\eps}\left(d_{\varepsilon}-\eps\rho\right).\label{modieq:innerAL}
\end{align}
We note that by the definition of $\mathbf{W}^\pm$ the right-hand side of \eqref{modieq:innerstokes} converges exponentially to zero as $|\rho|\to\infty$.
 Here  we introduced
$\mathbf{u}_{\eps}(x,t)$
and $\mathbf{l}_{\eps}(x,t)$ for $(x,t)\in\Gamma(3\delta)$ in a similar manner as in \cite{AlikakosLimitCH}. $\mathbf{l}_\eps$ will ensure  the compatibility
conditions in $\Gamma(3\delta)\backslash\Gamma$ for \eqref{modieq:innerstokes} and $\mathbf{u}_\eps$ will ensure the matching conditions for $\ve_\eps$ on $\Gamma(3\delta)\backslash\Gamma$. Moreover, an  auxiliary function $\hat{g}_{\eps}(\rho,x,t)=\eps g_{\eps}\eta'(\rho)-\theta_0'(\rho)\hat{\phi}_\eps$ is introduced, which is used to satisfy the compatibility
conditions in $\Gamma(3\delta)\backslash\Gamma$ for \eqref{modieq:innerAL}. Note that the extra terms vanish on the relevant set
 \[
   S^{\eps}:=\big\{ (\rho,x,t)\in\mathbb{R}\times\Gamma(3\delta):\rho=\tfrac{d_{\varepsilon}(x,t)}{\eps}\big\}.
 \]
By the definition of $\hat{g}_\eps$ we have
\begin{align*}
\varepsilon^{\frac{1}{2}}\hat{g}_{\eps}\left(d_{\varepsilon}-\eps\rho\right)=-\varepsilon^{\frac{5}{2}}\eta'(\rho)\rho g_{\eps}+\varepsilon^{\frac{3}{2}}\big(\eta'(\rho) g_{\eps}d_{\varepsilon}+\theta_0'(\rho)\rho\phi^\varepsilon\big)-\varepsilon^{\frac{1}{2}}\theta_0'(\rho)\hat{\phi}^\varepsilon d_{\varepsilon}.
\end{align*}
More precisely,  we choose the auxiliary
functions to have expansions of the form
\begin{align*}
\hat{\phi}_\eps(\rho,x,t)=\hat{\phi}_{0}(\rho,x,t)+\eps^{\frac{1}{2}}\hat{\phi}_{\frac{1}{2}}(\rho,x,t)\quad \text{in }\Gamma(3\delta)
 \end{align*} with
  \begin{align}\hat{\phi}_0&=\frac{\partial_t d_{\Gamma}+\hat{\ve}_0\cdot\nabla d_{\Gamma}}{ d_{\Gamma}},\label{definition:phi0} \\
  \hat{\phi}_{\frac{1}{2}}&=\frac{\partial_t d_{\frac{1}{2}}+\hat{\ve}_0\cdot\nabla d_{\frac{1}{2}}+\hat{\ve}_{\frac{1}{2}}\cdot\nabla d_{\Gamma}-\hat{\phi}_0d_{\frac{1}{2}}-\Delta d_{\Gamma}}{ d_{\Gamma}},\label{definition:phi12}
  \end{align}
  and
\begin{align}
\mathbf{u}_{\eps}(x,t) \approx\sum_{k\in \frac12\N_0}\mathbf{u}_{k}(x,t)\eps^{k},\quad  g_{\eps}(x,t)  \approx\sum_{k\in \frac12\N_0}g_{k}(x,t)\eps^{k}, \quad
\mathbf{l}_{\eps}(x,t) \approx\sum_{k\in \frac12\N_0}\mathbf{l}_{k}(x,t)\eps^{k}\label{eq:zusatzerme}
\end{align}
for $(x,t)\in\Gamma(3\delta)$. We note that  $\hat{\phi}_{0},\hat{\phi}_{\frac{1}{2}}$ are defined on $\Gamma$ as limits $d_\Gamma\to 0$, which exist due to \eqref{eq:Limit5}, \eqref{evolution law:d12}, and since it will turn out that $\hat{\ve}|_{\Gamma}=\ve_0^\pm|_{\Gamma}$, cf.\ \eqref{inn-0-ve-express} below. In particular,
\begin{alignat}{2}
  \hat{\phi}_0&=\nabla\big(\partial_t d_{\Gamma}+\hat{\ve}_0\cdot\nabla d_{\Gamma}\big)\cdot\nabla d_{\Gamma}=\nabla (\hat{\ve}_0\cdot\nabla d_{\Gamma})\cdot\nabla d_{\Gamma}&\quad &\text{on}\ \Gamma,\label{gamma definition:phi0} \\
 \hat{\phi}_{\frac{1}{2}}&=\nabla\big(\partial_t d_{\frac{1}{2}}+\hat{\ve}_0\cdot\nabla d_{\frac{1}{2}}+\hat{\ve}_{\frac{1}{2}}\cdot\nabla d_{\Gamma}-\hat{\phi}_0d_{\frac{1}{2}}-\Delta d_{\Gamma}\big)\cdot\nabla d_{\Gamma}&\quad &\text{on}\ \Gamma\label{gamma definition:phi12}
\end{alignat}
since $\nabla d_\Gamma \cdot \nabla \partial_t d_\Gamma = \frac12 \partial_t |\nabla d_\Gamma|^2=0 $.
 Here, in order to obtain \eqref{eq:NSAC1-new}-\eqref{eq:NSAC3-new} (approximately), the equations above only have to hold in
 $
   S^{\eps}=\big\{ (\rho,x,t)\in\mathbb{R}\times\Gamma(3\delta):\rho=\tfrac{d_{\varepsilon}(x,t)}{\eps}\big\}.
 $
But in the following we consider them as ordinary differential equations in $\rho\in\mathbb{R}$,
where $(x,t)\in \Gamma(3\delta)$ are seen as fixed parameters. Thus we require from now on that   \eqref{modieq:innerstokes}-\eqref{modieq:innerAL}
are fulfilled even for all $(\rho,x,t)\in \mathbb{R}\times\Gamma(3\delta)$.

Furthermore, we assume that we have the expansions
\begin{align*}
  \hat{c}_{\eps}(\rho,x,t)  &\approx\sum_{k\in\frac12\N_0}\eps^{k}\hat{c}_{k}(\rho,x,t),\quad
  \hat{\mathbf{v}}_{\eps}(\rho,x,t) \approx\sum_{k\in\frac12\N_0}\eps^{k}\hat{\mathbf{v}}_{k}(\rho,x,t),\quad\\
\hat{p}_{\eps}(\rho,x,t) &\approx\sum_{k\in\frac12\N_{-2}}\eps^{k}\hat{p}_{k}(\rho,x,t).
\end{align*}
understood in the same way as before. Actually, in the expansion it turns out that $\hat{c}_0=\theta_0$ and $\hat{c}_{\frac{1}{2}}=\hat{c}_1=0$. To simplify the following presentation we already assume $\hat{c}_{\frac{1}{2}}=\hat{c}_1=0$.
As usual we normalize $\hat{c}_{k}$
such that
\begin{align}
  \hat{c}_{k}(0,x,t)=0\qquad \text{for all }(x,t)\in\Gamma(3\delta), k\geq 0.\label{innnormalize}
\end{align}

In order to match the inner and outer expansions, we require that
for all $k$ the so-called \emph{inner-outer matching conditions}
\begin{align}
\sup_{(x,t)\in\Gamma(3\delta)}\left|\partial_{x}^{m}\partial_{t}^{n}\partial_{\rho}^{l}\left(\varphi (\pm\rho,x,t)-\varphi^{\pm}(x,t)\right)\right| & \leq Ce^{-\alpha\rho},\label{eq:matchcon}
\end{align}
where $\varphi=\hat{c}_{k},\hat{\mathbf{v}}_{k}$ with $k\geq0$ and $\hat{p}_{k}$ with $k\geq-1$
hold for constants $\alpha,C>0$ and all $\rho>0$, $m,n,l\geq0$.

For the non-linear terms $\Phi=\nu,f'$ we use the expansion
\begin{align}
\Phi(c_\eps)& =\Phi(\hat{c}_{0})+\sum_{k=\frac32}^{N+2}\eps^k \Phi'(\hat{c}_{0})\hat{c}_{k}
+\sum_{k=\frac{5}{2}}^{N+2}\eps^{k}\Phi_{k-1}(\hat{c}_{0},\hat{c}_{\frac{3}{2}},\ldots,\hat{c}_{k-1})
\nonumber\\&\quad+\eps^{N+\frac52}\Phi_{N+\frac{3}{2}}(c_\eps, \hat{c}_{0},\hat{c}_{\frac{3}{2}},\ldots,\hat{c}_{N+\frac{3}{2}})
\label{inner:nu}
\end{align}
where we have used $\hat{c}_{1/2}=\hat{c}_1=0$. Here $\Phi_{k-1}(\hat{c}_{0},\hat{c}_{\frac{3}{2}},\ldots,\hat{c}_{k-1})$ and $\Phi_{N+\frac{3}{2}}(c_\eps, \hat{c}_{0},\hat{c}_{\frac{3}{2}},\ldots,\hat{c}_{N+\frac{3}{2}})$ are polynomials in $\hat{c}_{\frac32},\ldots, \hat{c}_{k-1}$ with coefficients that depend smoothly on $\hat{c}_0$ and $(\hat{c}_0,c_\eps)$, respectively.

Using this expansion in \eqref{distanceapp} we obtain 
\begin{align*}
  |\nabla d_\eps|^2\approx &\underbrace{|\nabla d_0|^2}_{=1} + 2\eps^{\frac12} \nabla d_0\cdot \nabla d_{1/2}
  +\sum\limits_{ k\geq1,k\in\frac12\N_0}\eps^{k}\sum\limits_{0\leq i\leq k,i\in\frac12\N_0}\nabla d_{i}\cdot \nabla d_{k-i}.
\end{align*}
Hence, in order to satisfy \eqref{distanceapp} (up to higher order terms in $\eps$)  in ${\Gamma(3\delta)}$ we choose $d_k$, $k=\frac12, 1, \frac32,\ldots$ successively such that
\begin{alignat}{2}\label{eq:d12}
  \nabla d_0\cdot \nabla d_{1/2} &=0,&\qquad&\text{in }\Gamma(3\delta),\\  \label{eq:dk}
  \nabla d_0\cdot \nabla d_{k} &=-\frac{1}{2}\sum\limits_{\frac{1}{2}\leq i\leq k-\frac{1}{2},i\in\frac12\N_0}\nabla d_{i}\cdot \nabla d_{k-i}\quad \text{for }  k\geq1,\  k\in\frac12\N_0,&\qquad&\text{in }\Gamma(3\delta).
\end{alignat}
Furthermore,  we choose $d_{1/2}$ such that
\begin{align}
 \partial_t d_{\frac{1}{2}}+\hat{\ve}_0\cdot\nabla d_{\frac{1}{2}}+\hat{\ve}_{\frac{1}{2}}\cdot\nabla d_{\Gamma}-\hat{\phi}_0d_{\frac{1}{2}}-\Delta d_{\Gamma}&=0\qquad \text{on } \Gamma,\label{evolution law:d12}
\end{align}
which  will ensure that \eqref{definition:phi12} is well-defined and we can choose $\hat{c}_{1}=0$, cf.\ \eqref{modieq:innerAL1} below. In what follows (see Corollary \ref{independent-1} below) we will find that $\hat{\ve}_0,\hat{\ve}_{\frac{1}{2}}\cdot\nabla d_{\Gamma}$ and $\hat{\phi}_0$ will be independent of $\rho$ on $\Gamma$.

To proceed we use that for $\mathbb{D}= \mathbb{A}, \mathbb{B}, \mathbb{C}$
\begin{equation*}
  \mathbb{D}_\eps(\rho,x,t)=\sum_{k\in\frac12\N_0}\eps^{k}\mathbb{D}_{k}(\rho,x,t)\quad \text{for all } \rho\in\R, (x,t)\in \Gamma(3\delta)
  \end{equation*}
since $\hat{c}_\eps$, $d_\eps$ and their derivatives have a corresponding expansion.

Matching the
$O(\eps^0)$-terms in the  Allen-Cahn equation \eqref{modieq:innerAL}, we find
\begin{equation}
  \label{eq:c0}
\hat{c}_{0}(\rho,x,t)=\theta_0(\rho)\quad \text{for all }\rho \in\R, (x,t)\in \Gamma(3\delta).
\end{equation}
 Matching the
$O(\eps^0)$-terms in the transformed momentum equation \eqref{modieq:innerstokes} and the divergence equation \eqref{eq:innerdiv}, we derive the following ordinary differential equations
in $\rho$:
\begin{align}
  \partial_{\rho}\bigg(\nu(\theta_0)\big(\partial_{\rho}\hat{\ve}_0-\mathbf{u}_0d_{\Gamma}\eta'\big)\bigg)&=\big(2\theta_0'\theta_0''+\partial_{\rho}\hat{p}_{-1}\big)\nabla d_{\Gamma},\label{modieq:innerstokes0}\\
   \big(\partial_{\rho}\hat{\ve}_0-\mathbf{u}_0d_{\Gamma}\eta'\big)\cdot\nabla d_{\Gamma}&=0,\label{modieq:innerdiv0}
\end{align}
where the right-hand side vanishes for the choice $p_{-1}(\rho)= \frac12\theta_0'(\rho)^2$.
Matching the
$O(\eps^{\frac{1}{2}})$-order terms in the  Allen-Cahn equation \eqref{modieq:innerAL}, we find
\begin{align*}
\partial^2_{\rho}\hat{c}_{\frac{1}{2}}-f''(\theta_0)\hat{c}_{\frac{1}{2}}=\theta_0'\big(\partial_t d_{\Gamma}+\hat{\ve}_0\cdot\nabla d_{\Gamma}-\hat{\phi}_0d_{\Gamma}\big)=0,
\end{align*}
which is compatible with the choice
$\hat{c}_{\frac{1}{2}}=0$ above. Here we have used \eqref{definition:phi0} and \eqref{innnormalize}.
Then matching the
$O(\eps^{\frac{1}{2}})$-order terms in the transformed momentum equation and the divergence equation, we obtain the following ordinary differential equations with respect to $\rho$:
  \begin{align}
\partial_{\rho}\bigg(\nu(\theta_0)\big(\partial_{\rho}\hat{\ve}_{\frac{1}{2}}-(\mathbf{u}_{\frac{1}{2}}d_{\Gamma}+\mathbf{u}_0d_{\frac{1}{2}})\eta'\big)\bigg)&=\left(2\theta_0'\theta_0''+\partial_{\rho }\hat{p}_{-1}\right)\nabla d_{\frac{1}{2}}+\partial_{\rho }\hat{p}_{-\frac{1}{2}}\nabla d_{\Gamma},\label{modieq:innerstokes12}\\
\big(\partial_{\rho}\hat{\ve}_{\frac{1}{2}}-(\mathbf{u}_{\frac{1}{2}}d_{\Gamma}+\mathbf{u}_0d_{\frac{1}{2}})\eta'\big)\cdot\nabla d_{\Gamma}&=\left(-\partial_{\rho}\hat{\ve}_{0}+\mathbf{u}_{0}d_{\Gamma}\eta'\right)\cdot\nabla d_{\frac{1}{2}}.\label{modieq:innerdiv12}
\end{align}
Furthermore comparing the
$O(\eps)$-order terms in the  Allen-Cahn equation \eqref{modieq:innerAL}, we have
\begin{align}
\partial^2_{\rho}\hat{c}_{1}-f''(\theta_0)\hat{c}_{1}=\theta_0'\big(\partial_t d_{\frac{1}{2}}+\hat{\ve}_0\cdot\nabla d_{\frac{1}{2}}+\hat{\ve}_{\frac{1}{2}}\cdot\nabla d_{\Gamma}-\hat{\phi}_0d_{\frac{1}{2}}-\hat{\phi}_{\frac{1}{2}}d_{\Gamma}-\Delta d_{\Gamma}\big)=0,\label{modieq:innerAL1}
\end{align}
which again justifies the choice
$\hat{c}_{1}=0$. 

Then matching the
$O(\eps)$-order terms in the transformed momentum equation and the divergence equation, we obtain the following ordinary differential equations
  \begin{align}
&\partial_{\rho}\bigg(\nu(\theta_0)\big(\partial_{\rho}\hat{\ve}_{1}-(\mathbf{u}_{1}d_{\Gamma}+\mathbf{u}_{0}d_{1})\eta'\big)\bigg)=2\theta_0'\theta_0''\nabla d_{1}+\partial_{\rho}\hat{\ve}_0\partial_{t}d_{\Gamma}+\hat{\ve}_0\cdot\nabla d_{\Gamma}\partial_{\rho}\hat{\ve}_0\nonumber\\
&\quad-\partial_{\rho}\big(2\nu(\theta_0)D\hat{\ve}_0\big)\cdot\nabla d_{\Gamma}
-\Div\big(2\nu(\theta_0)D_d\hat{\ve}_0\big)+\partial_{\rho}\big(\nu(\theta_0)\Div\hat{\ve}_0\big)\nabla d_{\Gamma}\nonumber\\
    &\quad +\sum\limits_{i\in\{-1,-\frac{1}{2},0\}}\partial_{\rho}\hat{p}_i\nabla d_{-i}+\mathbb{A}_0+\nabla \hat{p}_{-1}
    +\big(\nu(\theta_0)\eta'\big)'\big(\mathbf{u}_{\frac{1}{2}}d_{\frac{1}{2}}-\mathbf{u}_0\rho\big)+\mathbf{l}_{0}d_{\Gamma}\eta',\label{modieq:innerstokes1}
  \end{align}
and
  \begin{align}
&\big(\partial_{\rho}\hat{\ve}_{1}-(\mathbf{u}_{1}d_{\Gamma}+\mathbf{u}_{0}d_{1})\eta'\big)\cdot\nabla d_{\Gamma}=-\sum\limits_{i\in\{0,\frac{1}{2}\}}\partial_{\rho}\hat{\ve}_{i}\nabla d_{1-i}-\Div_x\hat{\ve}_0-\mathbf{u}_0\cdot \nabla d_{\Gamma}\eta'\rho\nonumber\\&\qquad+\eta'\sum\limits_{i\in\{\frac{1}{2},1\}}\mathbf{u}_{0}\cdot \nabla d_{i}d_{1-i}+\eta'\sum\limits_{i\in\{0,\frac{1}{2}\}}\mathbf{u}_{\frac{1}{2}}\cdot\nabla d_{i}d_{\frac{1}{2}-i}.\label{modieq:innerdiv1}
\end{align}
Similarly, comparing the
$O(\eps^{k})$-order terms for $k\geq\frac{3}{2}$ in the transformed momentum equation \eqref{modieq:innerstokes}, the divergence equation \eqref{eq:innerdiv} and the  Allen-Cahn equation \eqref{modieq:innerAL}, we obtain the following ordinary differential equations
\begin{align}
  &\partial_{\rho}\left(\nu(\theta_0)\big(\partial_{\rho}\hat{\ve}_k-(\mathbf{u}_kd_{\Gamma}+\mathbf{u}_0d_{k})\eta'\big)\right)\nonumber\\
  &=-\sum\limits_{i\in\{k-1,k-\frac{1}{2},k\}}\partial_\rho\big(\hat{\nu}_i\partial_{\rho}\hat{\ve}_{k-i}\big)
+2\partial_{\rho}(\partial_{\rho}\hat{c}_k\theta_0')\nabla d_{\Gamma}+2\partial_{\rho}(\partial_{\rho}\hat{c}_{k-\frac{1}{2}}\theta_0')\nabla d_{\frac{1}{2}}+2\partial_{\rho}(\partial_{\rho}\hat{c}_{k-1}\theta_0')\nabla d_{1}
\nonumber\\&\quad+\sum\limits_{i\in\{0,k-1\}}\partial_{\rho}\hat{\ve}_i\partial_td_{k-1-i}+\partial_{\rho}\hat{\ve}_0\sum\limits_{i\in\{0,k-1\}}\hat{\ve}_i\cdot\nabla d_{k-1-i}+\partial_{\rho}\hat{\ve}_{k-1}\hat{\ve}_0\cdot\nabla d_{\Gamma}
\nonumber\\&\quad+\theta_0''\nabla \hat{c}_{k-1}+(\theta_0')^2\big(\Delta d_{\Gamma}\nabla d_{k-1}+\Delta d_{k-1}\nabla d_{\Gamma}\big)+\eta'(\rho)\sum\limits_{i\in\{0,k-1\}}\mathbf{l}_{i}d_{k-1-i}
\nonumber\\&\quad-2\nabla d_{\Gamma}\cdot\sum\limits_{i\in\{0,k-1\}}\partial_\rho(\hat{\nu}_i D\hat{\ve}_{k-1-i})-2\nabla d_{k-1}\cdot\partial_\rho(\hat{\nu}_0 D\hat{\ve}_{0})-2\Div(\hat{\nu}_0 {(D_d\hat{\ve})_{k-1}})-2\Div(\hat{\nu}_{k-1} D_d\hat{\ve}_{0})
\nonumber\\&\quad+\nabla d_{\Gamma}\sum\limits_{i\in\{0,k-1\}}\partial_\rho(\hat{\nu}_i \Div\hat{\ve}_{k-1-i})+\nabla d_{k-1}\partial_\rho(\hat{\nu}_0 \Div\hat{\ve}_{0})+\sum\limits_{i\in\{0,k-1\}}\partial_\rho \hat{p}_i\nabla d_{k-1-i}
\nonumber\\&\quad+\partial_\rho\big(\nu(\theta_0)\eta'\big)\sum\limits_{i\in\{\frac{1}{2},1,k-\frac{1}{2},k-1\}}\mathbf{u}_{i}d_{k-i}-\partial_\rho\big(\nu(\theta_0)\eta'\big)\rho\mathbf{u}_{k-1}
  +\mathcal{R}_{1,k-\frac{3}{2}},\label{modieq:innerstokesk}
\end{align}
\begin{align}
  &\big(\partial_{\rho}\hat{\ve}_k-\eta'(\mathbf{u}_kd_{\Gamma}+\mathbf{u}_0d_{k}))\cdot\nabla d_\Gamma\nonumber\\
  & \quad =-\sum\limits_{i\in\{0,\frac{1}{2},1,k-1,k-\frac{1}{2}\}}\partial_{\rho}\hat{\ve}_{i}\nabla d_{k-i}-\Div\hat{\ve}_{k-1}+\eta'\mathbf{u}_0\cdot\sum\limits_{i\in\{\frac{1}{2},1,k-1,k-\frac{1}{2},k\}}\nabla d_{i}d_{k-i}
\nonumber\\&\qquad+\eta'\mathbf{u}_{\frac{1}{2}}\cdot\sum\limits_{i\in\{\frac12,k-1,k-\frac12\}}\nabla d_{i}d_{k-\frac{1}{2}-i}+\eta'\mathbf{u}_1\cdot\sum\limits_{i\in\{0,k-1\}}\nabla d_{i}d_{k-1-i}
\nonumber\\&\qquad+\eta'\mathbf{u}_{k-\frac{1}{2}}\cdot\sum\limits_{i\in\{0,\frac12\}}\nabla d_{i}d_{\frac12-i}+\eta'\mathbf{u}_{k-1}\cdot\sum\limits_{i\in\{0,1\}}\nabla d_{i}d_{1-i}
\nonumber\\&\qquad-\rho\eta'\sum\limits_{i\in\{0,k-1\}}\mathbf{u}_i\cdot\nabla d_{k-1-i}
  +\mathcal{R}_{2,k-\frac{3}{2}},\label{modieq:innerdivk}
\end{align}
and
\begin{align}
  &\partial^2_{\rho}\hat{c}_{k}-f''(\theta_0)\hat{c}_{k}=\hat{f}'_{k-1}+\theta_0'\partial_t d_{k-\frac{1}{2}}+\partial_{\rho}\hat{c}_{k-\frac{1}{2}}\partial_t d_{\Gamma}+\partial_{\rho}\hat{c}_{k-1}\partial_t d_{\frac{1}{2}}\nonumber\\
  &\quad +\theta_0'\sum\limits_{i\in\{0,\frac{1}{2},k-1,k-\frac{1}{2}\}}\hat{\ve}_i\cdot\nabla d_{k-\frac{1}{2}-i}
+\partial_{\rho}\hat{c}_{k-\frac{1}{2}}\hat{\ve}_0\cdot\nabla d_{\Gamma}+\partial_{\rho}\hat{c}_{k-1}\sum\limits_{i\in\{0,\frac12\}}\hat{\ve}_i\cdot\nabla d_{\frac{1}{2}-i}\nonumber\\&\quad
-\sum\limits_{i\in\{0,k-1\}}\partial_{\rho}\hat{c}_i\Delta d_{k-1-i} - 2\nabla\partial_{\rho}\hat{c}_{k-1}\cdot\nabla d_{\Gamma}-\theta_0'\sum\limits_{i\in\{0,\frac12\}}\hat{\phi}_id_{k-\frac{1}{2}-i}
\nonumber\\&\quad+\eta'g_{k-\frac{3}{2}}d_{\Gamma}+\mathcal{R}_{3,k-\frac{3}{2}},\label{modieq:innerALk}
\end{align}
where \begin{align*}
  {\big(D_d\hat{\ve}\big)_{k-1}=\frac 12\big(\partial_\rho\hat{\ve}_{k-1}\otimes \nabla d_{\Gamma}+(\partial_\rho\hat{\ve}_{k-1}\otimes\nabla d_{\Gamma})^T\big)+\frac 12\big(\partial_\rho\hat{\ve}_0\otimes \nabla d_{k-1}+(\partial_\rho\hat{\ve}_0\otimes\nabla d_{k-1})^T\big)}
\end{align*}
and $\mathcal{R}_{1,k-\frac{3}{2}}$, $\mathcal{R}_{2,k-\frac{3}{2}}$, $\mathcal{R}_{3,k-\frac{3}{2}}$ depend on the terms up to $k-\frac{3}{2}$ order and converge exponentially to zero as $|\rho|\to\infty$ because of the choice of $W^\pm_k$ and $\Div \ve_k^\pm =0$ for all $k\in\frac12\N_0$.

\subsection{Existence of Expansion Terms}
The following two lemmas are used to solve the ordinary differential equations with respect to $\rho$ and can be found in  \cite[Lemma A.2 and Lemma A.3]{AbelsFei}.
\begin{lem}\label{ODEsolver}
Let $U\subset \mathbb{R}^n$ be an open subset and let $A:\mathbb{R}\times U\rightarrow \mathbb{R},(\rho,x)\mapsto A(\rho,x)$ be given and smooth. Assume that there exists $A^\pm(x)$ such that the decay property $A(\pm\rho,x)-A^\pm(x)=O(e^{-\alpha\rho})$ as $\rho\rightarrow\infty$ is fulfilled. Then for every $x\in U$
the system
\begin{equation}\label{eq:ODEsolver}
\begin{split}
&w_{\rho\rho}(\rho,x)-f'(\theta_0(\rho))w(\rho,x)=A(\rho,x)\quad \text{for all } \rho\in\mathbb{R},\\
&w(0,x)=0,\ w(\cdot,x)\in L^\infty(\mathbb{R})
\end{split}
\end{equation}
has a smooth and bounded  solution if and only if
\begin{equation}\label{compatiblity:ODEsolver}
  \int_{\mathbb{R}}A(\rho,x)\theta_0'(\rho)d\rho=0.
\end{equation}
In addition, if the solution exists, then it is unique and satisfies for all $x\in U$
\begin{equation}\label{decay1:ODEsolver}
\partial_\rho^\ell\bigg(w(\pm\rho,x)+\frac{A^\pm(x)}{f''(\pm1)}\bigg)=O(e^{-\alpha\rho})\quad \text{as}~\rho\to \infty,\ l=0,1,2.
\end{equation}
Furthermore, if $A(\rho,x)$ satisfies for all $x\in U$
\begin{align*}
\partial_x^m\partial_\rho^\ell\bigg(A(\pm\rho,x)-A^\pm(x)\bigg)=O(e^{-\alpha\rho})\quad \text{as}~ \rho\rightarrow\infty
\end{align*}
for all $m\in\{0,\cdots,M\}$ and $l\in\{0,\cdots,L\}$,
then
\begin{equation}\label{decay2:ODEsolver}
\partial_x^m\partial_\rho^\ell\bigg(w(\pm\rho,x)+\frac{A^\pm(x)}{f''(\pm1)}\bigg)=O(e^{-\alpha\rho})\quad \text{as}~\rho\to\infty
\end{equation}
for all $m\in\{0,\cdots,M\}$ and $l\in\{0,\cdots,L\}$.
\end{lem}

\begin{lem}\label{modified version of Lemma 2.4}
Let $U\subset \mathbb{R}^n$ be an open subset and let $B:\mathbb{R}\times U\rightarrow \mathbb{R},(\rho,x)\mapsto B(\rho,x)$ be given and smooth. Assume that for all $x\in U$ the decay property $B(\pm\rho,x)=O(e^{-\alpha\rho})$ as $\rho\rightarrow\infty$ is fulfilled.
Then for each $x\in U$ the problem
\begin{align}
\partial_{\rho}\big(\nu(\theta_0)\partial_\rho w(\rho,x)\big)=B(\rho,x)\quad \text{for all } \rho\in \mathbb{R}\label{modified version of Lemma 2.4-1}
\end{align}
has a solution $w(\cdot,x)\in C^2(\mathbb{R})\cap L^\infty(\mathbb{R})$ if and only if
\begin{align}
\int_{\mathbb{R}}B(\rho,x)d\rho=0.\label{modified version of Lemma 2.4-2}
\end{align}
Furthermore, if $w_{\ast}(\rho,x)$ is such a solution, then all the solutions can be written as
\begin{align}
 w(\rho,x)=w_{\ast}(\rho,x)+c(x),\ \rho\in \mathbb{R}\label{modified version of Lemma 2.4-3}
\end{align}
where $c:U\rightarrow \mathbb{R}$ is an arbitrary function.
In particular, if \eqref{modified version of Lemma 2.4-2} holds,
\begin{align}
w_{\ast}(\rho,x)=\int_{0}^\rho\frac{1}{\nu(\theta_0)}\int_{-\infty}^r B(s,x)ds dr,\ \rho\in \mathbb{R}\label{modified version of Lemma 2.4-4}
\end{align}
is a solution.
Additionally, if \eqref{modified version of Lemma 2.4-2} holds for all $x\in U$ and there exist $M,L\in \mathbb{N}$ such that
\begin{align}
\partial_x^m\partial_\rho^lB(\pm\rho,x)=O(e^{-\alpha\rho}) \ \text{as}\ \rho\rightarrow+\infty\label{modified version of Lemma 2.4-5}
\end{align}
for all $m\in\{0,\cdots,M\}$ and $l\in\{0,\cdots,L\}$, then there exists smooth functions $w^+(x)$ and  $w^-(x)$ such that
\begin{align}
D_x^mD_\rho^l\big(w(\pm\rho,x)-w^{\pm}(x)\big)=O(e^{-\alpha\rho}) \ \text{as} \ \rho\rightarrow+\infty\label{modified version of Lemma 2.4-6}
\end{align}
for all $m\in\{0,\cdots,M\}$ and $l\in\{0,\cdots,L+2\}$.
\end{lem}
\begin{rem}
  We note the statements on the asymptotics ``$O(e^{-\alpha\rho})$ as $\rho\to \infty$'' in Lemma~\ref{ODEsolver} and Lemma~\ref{modified version of Lemma 2.4} hold true uniformly with respect to $x\in U$ provided
  $A(\pm\rho,x)-A^\pm(x)=O(e^{-\alpha\rho})$ and $B(\pm\rho,x)=O(e^{-\alpha\rho})$ as $\rho\rightarrow\infty$ hold true uniformly with respect to $x\in U$.
\end{rem}

\subsubsection{Solving lower order terms}

{\it{Solving $\hat{p}_{-1}$, $\ve_0$, $\ue_0$, and $\ve_0^\pm$:}}
Firstly,  using the matching conditions  \eqref{eq:matchcon} and \eqref{innnormalize} we get $\hat{c}_{0}=\theta_0$ satisfying  \eqref{eq:OptProfile}.
Moreover, it follows from \eqref{modieq:innerstokes0} multiplied with $\nabla d_\Gamma$ and \eqref{modieq:innerdiv0} that
\begin{align*}
0=\partial_{\rho}\left(\nu(\theta_0)(\hat{\ve}_0-\mathbf{u}_0d_{\Gamma}\eta'(\rho))\cdot \nabla d_\Gamma\right)=2\theta_0'\theta_0''+\partial_{\rho}\hat{p}_{-1}.
\end{align*}
Hence
\begin{align}
\hat{p}_{-1}(\rho,x,t)&=-(\theta'_0(\rho))^2\quad\text{and} \label{inn-0-ve-2}\\\nonumber
\hat{\ve}_0(\rho,x,t)&=\overline{\ve}_0(x,t)+\mathbf{u}_0(x,t)d_{\Gamma}\big(\eta(\rho)-\tfrac{1}{2}\big)
\end{align}
for all $\rho\in\R$, $(x,t)\in \Gamma(3\delta)$ and some function $\overline{\ve}_0\colon \Gamma(3\delta)\to \R^2$ due to Lemma~\ref{modified version of Lemma 2.4}.
Using the matching conditions  we obtain
\begin{align*}
\mathbf{v}_{0}^+(x,t)=\overline{\ve}_0(x,t)+\frac{1}{2}\mathbf{u}_0(x,t)d_{\Gamma}(x,t),\ \ \mathbf{v}_{0}^-(x,t)=\overline{\ve}_0(x,t)-\frac{1}{2}\mathbf{u}_0(x,t)d_{\Gamma}(x,t)
\end{align*}
 and therefore
\begin{align*}
\overline{\ve}_0(x,t)=\frac{1}{2}\big(\mathbf{v}_{0}^+(x,t)+\mathbf{v}_{0}^-(x,t)\big),\quad \mathbf{u}_0(x,t)d_{\Gamma}(x,t)=\mathbf{v}_{0}^+(x,t)-\mathbf{v}_{0}^-(x,t)
\end{align*}
for all $(x,t)\in \Gamma(3\delta)$. The latter is consistent with \eqref{eq:Limit4} and yields
\begin{align}
\hat{\ve}_0(\rho,x,t)=\mathbf{v}_{0}^+(x,t)\eta(\rho)+\mathbf{v}_{0}^-(x,t)(1-\eta(\rho)),\label{inn-0-ve-express}
\end{align}
and
\begin{eqnarray}\label{formula:u0}
\mathbf{u}_0(x,t)=\left \{
\begin {array}{ll}
\frac{\mathbf{v}_{0}^+(x,t)-\mathbf{v}_{0}^-(x,t)}{d_{\Gamma}(x,t)}& \text{if }(x,t)\in\Gamma(3\delta)\backslash\Gamma,\\
\\
\nn\cdot\nabla\big(\mathbf{v}_{0}^+(x,t)-\mathbf{v}_{0}^-(x,t)\big)& \text{if }(x,t)\in\Gamma.
\end{array}
\right.
\end{eqnarray}
Therefore
\begin{align}
\partial_t d_{\Gamma}+\hat{\ve}_0\cdot\nabla d_{\Gamma}=\partial_t d_{\Gamma}+\mathbf{v}_{0}^-\cdot\nabla d_{\Gamma}+\big(\mathbf{v}_{0}^+-\mathbf{v}_{0}^-\big)\cdot\nabla d_{\Gamma}\eta(\rho)\label{interface equality}
\end{align}
which vanishes on $\Gamma$ due to \eqref{eq:Limit4} and \eqref{eq:Limit5}.

\medskip

\noindent
{\it{Solving $\hat{p}_{-\frac{1}{2}}$ and the $O(\varepsilon^{\frac{1}{2}})$-order terms}:}
It follows from \eqref{modieq:innerstokes12} multiplied with $\nabla d_\Gamma$, \eqref{modieq:innerdiv12}, and \eqref{inn-0-ve-2} that
\begin{align*}
\partial_{\rho}\hat{p}_{-\frac{1}{2}}(\rho,x,t)=0\qquad \text{for all }\rho\in\R,  (x,t)\in \Gamma(3\delta). 
\end{align*}
Hence we can choose $\hat{p}_{-\frac{1}{2}}\equiv 0$. Using Lemma~\ref{modified version of Lemma 2.4} we obtain
\begin{align*}
\hat{\ve}_{\frac{1}{2}}(\rho,x,t)=\overline{\ve}_{\frac{1}{2}}(x,t)+\big(\mathbf{u}_{\frac{1}{2}}(x,t)d_{\Gamma}(x,t)+\mathbf{u}_{0}(x,t)d_{\frac{1}{2}}(x,t)\big)\big(\eta(\rho)-\tfrac{1}{2}\big)
\end{align*}
for all $\rho\in\R$, $(x,t)\in \Gamma(3\delta)$ and some function $\overline{\ve}_{\frac{1}{2}}\colon \Gamma(3\delta)\to \R^2$.
Because of the matching conditions we  get
\begin{align*}
\mathbf{v}_{\frac{1}{2}}^+=\overline{\ve}_{\frac{1}{2}}+\frac{1}{2}\big(\mathbf{u}_{\frac{1}{2}}d_{\Gamma}+\mathbf{u}_{0}d_{\frac{1}{2}}\big),\ \ \mathbf{v}_{\frac{1}{2}}^-=\overline{\ve}_{\frac{1}{2}}-\frac{1}{2}\big(\mathbf{u}_{\frac{1}{2}}d_{\Gamma}+\mathbf{u}_{0}d_{\frac{1}{2}}\big)\quad \text{in }\Gamma(3\delta)
\end{align*}
as well as
\begin{align}
\overline{\ve}_{\frac{1}{2}}=\frac{1}{2}\big(\mathbf{v}_{\frac{1}{2}}^++\mathbf{v}_{\frac{1}{2}}^-\big),\ \ \mathbf{u}_{\frac{1}{2}}d_{\Gamma}+\mathbf{u}_{0}d_{\frac{1}{2}}=\mathbf{v}_{\frac{1}{2}}^+-\mathbf{v}_{\frac{1}{2}}^-\label{modieq:innerstokes12-4}
\end{align}
which immediately imply
\begin{align}
\hat{\ve}_{\frac{1}{2}}(\rho,x,t)=\mathbf{v}_{\frac{1}{2}}^+(x,t)\eta(\rho)+\mathbf{v}_{\frac{1}{2}}^-(x,t)(1-\eta(\rho))\label{modieq:innerstokes12-5}
\end{align}
and
\begin{eqnarray}\label{formula:u12}
\mathbf{u}_{\frac{1}{2}}(x,t)=\left \{
\begin {array}{ll}
\frac{\mathbf{v}_{\frac{1}{2}}^+(x,t)-\mathbf{v}_{\frac{1}{2}}^-(x,t)-\mathbf{u}_{0}(x,t)d_{\frac{1}{2}}(x,t)}{d_{\Gamma}(x,t)}&\text{if } (x,t)\in\Gamma(3\delta)\backslash\Gamma,\\
\\
\nn\cdot\nabla\big(\mathbf{v}_{\frac{1}{2}}^+(x,t)-\mathbf{v}_{\frac{1}{2}}^-(x,t)-\mathbf{u}_{0}(x,t)d_{\frac{1}{2}}(x,t)\big)&\text{if }  (x,t)\in\Gamma
\end{array}
\right.
\end{eqnarray}
for all $(x,t)\in\Gamma(3\delta)$, $\rho\in\R$.

\medskip

\noindent {\it{Solving the $O(\varepsilon)$-order terms}:}
To proceed we give the  following  proposition, which can be found in \cite[Proposition A.5]{AbelsFei}.
\begin{prop}\label{prop:1}
There hold
\begin{align}
\int_{-\infty}^{\infty}\hat{\ve}_0\cdot\nabla d_{\Gamma}\partial_{\rho}\hat{\ve}_{0}d\rho&=\frac{1}{2}\big(\ve_0^++\ve_0^-\big)\cdot\nabla d_{\Gamma}\big(\ve_0^+-\ve_0^-\big)=0 &\quad&\text{on } \Gamma,\label{preeq:inner1-4}\\
\int_{-\infty}^{\infty}\Div\big(2\nu(\theta_0)D_d\hat{\ve}_0\big)\big)d\rho&=\overline{\nu}\widetilde{\Div}\hat{\ve}_0=\overline{\nu}\mathbf{u}_0 &&\text{on } \Gamma,\label{preeq:inner1-5}\\
\int_{-\infty}^{\infty}\mathbb{A}_0d\rho&=\sigma\Delta d_{\Gamma}\nabla d_{\Gamma}= -\sigma H \no &&\text{on }\Gamma,\label{preeq:inner1-6}\\
\mathbf{u}_0\cdot\nn&=0 &&\text{on } \Gamma,\label{preeq:inner1-66}
\end{align}
where
$
\widetilde{\Div}\hat{\ve}_0=\left((\mathbf{v}_{0}^+-\mathbf{v}_{0}^-)\cdot\nabla\right)\nabla d_{\Gamma}+(\mathbf{v}_{0}^+-\mathbf{v}_{0}^-)\Delta d_{\Gamma}+(\nabla d_{\Gamma}\cdot\nabla)(\mathbf{v}_{0}^+-\mathbf{v}_{0}^-)
$
and $\ol{\nu}$ is defined as in \eqref{eq:defbarnu}.
\end{prop}

We need to point out that the first equalities in \eqref{preeq:inner1-4}-\eqref{preeq:inner1-6} hold  not only on $\Gamma$ but also in $\Gamma(3\delta)$. Moreover, because of \eqref{eq:Limit4} 
\begin{equation}\label{formula:secondderivative d}
  (\llbracket\ve_{0}^\pm\rrbracket\cdot\nabla)\nabla d_{\Gamma}=\big(\llbracket\ve_{0}^\pm\rrbracket\cdot\nn\big)\Delta d_{\Gamma}\nn=0,\quad \text{on }\Gamma.
\end{equation}
\begin{cor}\label{independent-1}
  $\hat{\ve}_0,\hat{\ve}_{\frac{1}{2}}\cdot\nabla d_{\Gamma}$ and $\hat{\phi}_0$ are independent of $\rho$ on $\Gamma$. Moreover, we can rewrite the evolution law  \eqref{evolution law:d12}  for $d_{1/2}$ as
  \begin{align}
    \partial_t d_{\frac{1}{2}}+\ve_0^{\pm}\cdot\nabla d_{\frac{1}{2}}+\ve_{\frac{1}{2}}^{\pm}\cdot\nabla d_{\Gamma}&-\left((\nabla d_{\Gamma}\cdot\nabla)\mathbf{v}_{0}^{\pm}\cdot\nabla  d_{\Gamma}\right)d_{\frac{1}{2}}-\Delta d_{\Gamma}=0\qquad \text{on } \Gamma.\label{newevolution law:d12}
  \end{align}
\end{cor}
\begin{proof} It follows from \eqref{inn-0-ve-express} that \begin{align}
\hat{\ve}_0(\rho,x,t)=\mathbf{v}_{0}^+(x,t)=\mathbf{v}_{0}^-(x,t)\quad \text{on}\ \Gamma.\label{independent-2}
\end{align}
By \eqref{modieq:innerstokes12-4} and \eqref{modieq:innerstokes12-5} one has
\begin{align*}
\hat{\ve}_{\frac{1}{2}}(\rho,x,t)=\mathbf{u}_{0}d_{\frac{1}{2}}\eta(\rho)+\mathbf{v}_{\frac{1}{2}}^-(x,t)\quad \text{on}\ \Gamma.
\end{align*}
Together with \eqref{preeq:inner1-66} this leads to
\begin{align}
\hat{\ve}_{\frac{1}{2}}(\rho,x,t)\cdot\nabla d_{\Gamma}=\mathbf{v}_{\frac{1}{2}}^-(x,t)\cdot\nabla d_{\Gamma}\quad \text{on}\ \Gamma.\label{independent-3}
\end{align}
 According to \eqref{gamma definition:phi0} one has
 \begin{align}
 \hat{\phi}_0&=\mathbf{u}_{0}\cdot\nabla d_{\Gamma}\eta(\rho)+(\nabla d_{\Gamma}\cdot\nabla)\mathbf{v}_{0}^-(x,t)\cdot\nabla  d_{\Gamma}= (\nabla d_{\Gamma}\cdot\nabla)\mathbf{v}_{0}^-(x,t)\cdot\nabla  d_{\Gamma}\quad \text{on}\ \Gamma.\label{independent-3}
 \end{align}
This implies the statement.
\end{proof}
\begin{rem}\label{rem:Solvabilityd12}
  We note that solvability of \eqref{newevolution law:d12} together with a system for $\ve_{\frac12}$ is given by Theorem~\ref{solvingkorder} in the appendix and will be discussed later.
\end{rem}

In order to apply Lemma \ref{modified version of Lemma 2.4} to \eqref{modieq:innerstokes1} the equation
\begin{align}\label{formula:lg0}
&\big(p_0^+-p_0^-+\sigma\Delta d_{\Gamma}\big)\nabla d_{\Gamma}+\big(\ve_0^+-\ve_0^-\big)\partial_t d_{\Gamma}+\frac{1}{2}\big(\ve_0^++\ve_0^-\big)\cdot\nabla d_{\Gamma}\big(\ve_0^+-\ve_0^-\big)\nonumber\\&\quad-2\big(\nu^+ D\ve_0^+-\nu^- D\ve_0^-\big)\cdot\nabla d_{\Gamma}-\overline{\nu}\widetilde{\Div}\hat{\ve}_0+\mathbf{l}_{0}d_{\Gamma}+\overline{\nu}\mathbf{u}_0=0
\end{align}
has to be satisfied on $\Gamma(3\delta)$ because of Proposition~\ref{prop:1}.
Using \eqref{preeq:inner1-4}-\eqref{preeq:inner1-6} we obtain that  \eqref{formula:lg0} on $\Gamma$ is equivalent to
\begin{align}\label{formula:lg0000}
2\llbracket \nu^\pm D\ve_0^\pm \rrbracket\nn -\llbracket  p_0^\pm \rrbracket\nn =\sigma\Delta d_{\Gamma}\nn,\ \ \text{on}\ \Gamma,
\end{align}
i.e., the balance of normal stresses \eqref{eq:Limit3} has to hold, which is true by our assumptions.  In order to obtain \eqref{formula:lg0} on $\Gamma(3\delta)$ we define
\begin{eqnarray}\label{formula:l0}
\mathbf{l}_{0}=\left \{
\begin {array}{ll}
\frac{\big(p_0^--p_0^+-\sigma\Delta d_{\Gamma}\big)\nabla d_{\Gamma}+\big(\ve_0^--\ve_0^+\big)\partial_t d_{\Gamma}-\frac{1}{2}\big(\ve_0^++\ve_0^-\big)\cdot\nabla d_{\Gamma}\big(\ve_0^+-\ve_0^-\big)}{d_{\Gamma}}\\ \qquad+\frac{2\big(\nu^+ D\ve_0^+-\nu^- D\ve_0^-\big)\cdot\nabla d_{\Gamma}+\overline{\nu}\big(\widetilde{\Div}\hat{\ve}_0-\mathbf{u}_0\big)}{d_{\Gamma}}& \text{in }\Gamma(3\delta)\backslash\Gamma,\\
\\
\nn\cdot\nabla\bigg(\big(p_0^--p_0^+-\sigma\Delta d_{\Gamma}\big)\nabla d_{\Gamma}+\big(\ve_0^--\ve_0^+\big)\partial_t d_{\Gamma}\\ \qquad -\frac{1}{2}\big(\ve_0^++\ve_0^-\big)\cdot\nabla d_{\Gamma}\big(\ve_0^+-\ve_0^-\big)+2\big(\nu^+ D\ve_0^+-\nu^- D\ve_0^-\big)\cdot\nabla d_{\Gamma}\\ \qquad+\overline{\nu}\big(\widetilde{\Div}\hat{\ve}_0-\mathbf{u}_0\big)\bigg)& \text{on }\Gamma.
\end{array}
\right.
\end{eqnarray}
Then the solution of \eqref{modieq:innerstokes1} is given by
\begin{align}\label{formula:ve1}
\hat{\ve}_1-(\mathbf{u}_0d_{1}+\mathbf{u}_1d_{\Gamma})\eta(\rho)&=\ve_1^-+\mathbf{u}_{\frac{1}{2}}d_{\frac{1}{2}}\eta(\rho)+\int_{-\infty}^\rho\frac{1}{\nu(\theta_0(r))}\int_{-\infty}^r\mathbf{V}^{0}(s,x,t)\,ds\,dr\nonumber\\
&=:\ve_1^-+\mathbf{W}^{0}\qquad \text{in }\R\times \Gamma(3\delta)
\end{align}
because of the matching conditions, where $\mathbf{V}^{0}$ consists of the terms up to zero order. Passing to the limit $\rho\to \infty$ yields
\begin{align}
\mathbf{u}_0d_{1}+\mathbf{u}_1d_{\Gamma}+\mathbf{u}_{\frac{1}{2}}d_{\frac{1}{2}}+\mathbf{W}^{0}|_{\rho=+\infty}=\mathbf{v}_{1}^+-\mathbf{v}_{1}^-,\label{modieq:innerstokes1-4}
\end{align}
which immediately implies that
\begin{align}
\hat{\ve}_{1}(\rho,x,t)=\mathbf{v}_{1}^+(x,t)\eta(\rho)+\mathbf{v}_{1}^-(x,t)(1-\eta(\rho))+\mathbf{W}^{0}(\rho,x,t)-\mathbf{W}^{0}(x,t,+\infty)\eta(\rho)\label{modieq:innerstokes1-5}
\end{align}
for all $\rho\in\R$, $(x,t)\in\Gamma(3\delta)$ and
\begin{eqnarray}\label{formula:u1}
\mathbf{u}_{1}=\left \{
\begin {array}{ll}
\frac{\mathbf{v}_{1}^+-\mathbf{v}_{1}^--\mathbf{u}_{0}d_{1}-\mathbf{u}_{\frac{1}{2}}d_{\frac{1}{2}}-\mathbf{W}^{0}|_{\rho=+\infty}}{d_{\Gamma}}& \text{in }\Gamma(3\delta)\backslash\Gamma,\\
\\
\nn\cdot\nabla\big(\mathbf{v}_{1}^+-\mathbf{v}_{1}^--\mathbf{u}_{0}d_{1}-\mathbf{u}_{\frac{1}{2}}d_{\frac{1}{2}}-\mathbf{W}^{0}|_{\rho=+\infty}\big)&\text{on }  \Gamma.
\end{array}
\right. 
\end{eqnarray}
In order to determine $\hat{p}_0$ we multiply \eqref{modieq:innerstokes1} with $\nabla d_\Gamma$, use \eqref{modieq:innerdiv1}, and obtain
\begin{align}
  &\partial_{\rho}\bigg(\hat{p}_0+\hat{\ve}_0\cdot\nabla d_{\Gamma}\partial_t d_{\Gamma}+\hat{\ve}_0\cdot\nabla d_{\Gamma}\hat{\ve}_{0}\cdot\nabla d_{\Gamma} -\big(2\nu(\theta_0
    )D\hat{\ve}_0:\nabla d_{\Gamma}\otimes\nabla d_{\Gamma}\big)\nonumber\\&\qquad-\int_{-\infty}^\rho\Div\big(2\nu(\theta_0(r))D_d\hat{\ve}_0(r,\cdot)\big)\cdot\nabla d_{\Gamma}\,dr
+\nu(\theta_0)\Div\hat{\ve}_0\nonumber\\&\qquad+\mathbf{l}_0d_{\Gamma}\eta\cdot\nabla d_{\Gamma}
  -\nu(\theta_0)\mathcal{A}^0+\int_{-\infty}^\rho\widetilde{\mathbf{V}}^0(r,\cdot)\cdot\nabla d_{\Gamma}dr\bigg)=0,
\end{align}
where $\mathcal{A}^{k-1}$ and $\widetilde{\mathbf{V}}^{k-1}$ consist of some terms up to $0$ order. Integrating this equation on $(-\infty,\rho)$ and using the matching condition for $\hat{p}_0$ determines $\hat{p}_0$.

In summary the equations for $(\hat{\ve}_1, \ue_1, \hat{p}_0)$ are solvable if $(\ve_0^\pm, p_0^\pm, (\Gamma_t)_{t\in [0,T_0]})$ solves the sharp interface limit system \eqref{eq:Limit1}-\eqref{eq:Limit5} and we have:
\begin{lem}[The zeroth order terms]
  \label{zeroorder}~\\
Let $\hat{p}_{-1}= (\theta_0')^2$, let $(\ve_0^\pm, p_0^\pm, (\Gamma_t)_{t\in [0,T_0]})$ be the solution of \eqref{eq:Limit1}-\eqref{eq:Limit5}, and let $(\mathbf{v}_0^{\pm},p_0^{\pm})$ be extended to $\Omega^{\pm}\cup\Gamma(3\delta)$ as in Remark~\ref{Outer-Rem}. Moreover, we define
$c_{0}^{\pm}(x,t)=\pm1$ for all $(x,t)\in \Omega^{\pm}\cup\Gamma(3\delta)$ and $\hat{c}_{0},\hat{\mathbf{v}}_{0}, \ue_0, \mathbf{l}_0$ by \eqref{eq:c0}, \eqref{inn-0-ve-express}, \eqref{formula:u0}, and \eqref{formula:l0}, respectively.
Then the outer equations \eqref{eq:c0out},  \eqref{eq:stokes-Outerk},
\eqref{eq:Outervpdefine} (for $k=0$), the inner equations \eqref{modieq:innerstokesk},
\eqref{modieq:innerdivk}, \eqref{modieq:innerALk}
(for $k=0$), the inner-outer matching conditions \eqref{eq:matchcon}
(for $k=0$) are satisfied on $\Gamma(3\delta)$. Finally, the compatibility condition for \eqref{modieq:innerstokes1}, which is equivalent to \eqref{formula:lg0}, is satisfied on $\Gamma(3\delta)$.
\end{lem}

\subsubsection{Solving the Higher Order Terms}

\noindent
{\it Determining $\hat{c}_{\frac{3}{2}}$, the evolution law of $d_1$ on $\Gamma$, and $g_0$:} Taking $k=\frac{3}{2}$ in \eqref{modieq:innerALk} one has
\begin{align}
\partial^2_{\rho}\hat{c}_{\frac{3}{2}}-f''(\theta_0)\hat{c}_{\frac{3}{2}}=&\theta_0'(\rho)\partial_t d_{1}+\theta_0'(\rho)\big(\hat{\ve}_0\cdot\nabla d_{1}+\hat{\ve}_{\frac{1}{2}}\cdot\nabla d_{\frac{1}{2}}+\hat{\ve}_1\cdot\nabla d_{\Gamma}\big)-\theta_0'(\rho)\Delta d_{\frac{1}{2}}
\nonumber\\&-\theta_0'(\rho)\big(\hat{\phi}_0d_{1}+\hat{\phi}_{\frac{1}{2}}d_{\frac{1}{2}}\big)+\eta'(\rho)g_0d_{\Gamma}+\theta_0'(\rho)\rho\hat{\phi}_{0}.\label{modieq:innerAL32}
\end{align}
The compatibility condition \eqref{compatiblity:ODEsolver} for \eqref{modieq:innerAL32}  is equivalent to
\begin{align}
\mathcal{D}_1+\sigma^{-1}\sigma_1g_0d_{\Gamma}=0\qquad \text{in }\Gamma(3\delta),\label{evolution law:d1}
\end{align}
where $\sigma_1:=\int_{-\infty}^{\infty}\theta_0'(\rho)\eta'(\rho)d\rho$
\begin{align*}
\mathcal{D}_1:=\partial_t d_{1}+\frac{\ve_0^{+}+\ve_0^{-}}{2}\cdot\nabla d_{1}+\frac{\ve_1^{+}+\ve_1^{-}}{2}\cdot\nabla d_{\Gamma}-\sigma^{-1}\int_{-\infty}^{+\infty}\big(\theta_0'(\rho)\big)^2\hat{\phi}_0(\rho,\cdot)\, d\rho \, d_{1}+D_{\frac{1}{2}}
\end{align*}
and $D_{\frac{1}{2}}$ depends on the terms up to order $\frac{1}{2}$, which were determined before. Here we have used the definition of $\hat{\ve}_0$ and $\int_{\R}\theta'(\rho)(\eta(\rho)-\tfrac12)\, d\rho=0$ since $\eta-\tfrac12$ is odd.
On $\Gamma$ the latter equation is satisfied if and only if $d_1$ solves the evolution equation
\begin{align}
\partial_t d_{1}+\ve_0^{\pm}\cdot\nabla d_{1}+\frac{\ve_1^{+}+\ve_1^{-}}{2}\cdot\nabla d_{\Gamma}-\hat{\phi}_0 d_{1}+D_{\frac{1}{2}}=0\ \ \text{on}\ \Gamma.\label{evolution law:d1gamma}
\end{align}
 In order to satisfy the compatibility condition on $\Gamma(3\delta)\setminus \Gamma$ we define
\begin{eqnarray}\label{formula:g0}
g_0=\left \{
\begin {array}{ll}
-\frac{\sigma\sigma_1^{-1}\mathcal{D}_1}{d_{\Gamma}}& \text{in }\Gamma(3\delta)\backslash\Gamma,\\
\\
-\sigma\sigma_1^{-1}\nn\cdot\nabla\mathcal{D}_1& \text{on }  \Gamma.
\end{array}
\right.
\end{eqnarray}

\noindent
{\it Determining $\hat{c}_{k}$ for $k \geq 2$, the evolution law of $d_{k-\frac{1}{2}}$ on $\Gamma$, and $g_{k-\frac{3}{2}}$:} First of all, we  rewrite \eqref{modieq:innerALk} as
\begin{align}
&\partial^2_{\rho}\hat{c}_{k}-f''(\theta_0)\hat{c}_{k}=\mathcal{S}_{1,k-\frac{3}{2}}(\theta_0,\hat{c}_{\frac{3}{2}},\cdots,\hat{c}_{k-\frac{1}{2}},d_0,\cdots,d_{k-1},\hat{\ve}_0,\cdots,\hat{\ve}_{k-\frac{3}{2}},g_0,\cdots,g_{k-2},\hat{\phi}_0,\hat{\phi}_{\frac{1}{2}})
\nonumber\\&\quad +\theta_0'\partial_t d_{k-\frac{1}{2}}+\theta_0'\sum\limits_{j=0,j\in\frac12\N_0}^{k-\frac{1}{2}}\hat{\ve}_j\cdot\nabla d_{k-\frac12-j}-\theta_0'\Delta d_{k-1}-\theta_0'\hat{\phi}_0d_{k-\frac{1}{2}}+\eta'g_{k-\frac{3}{2}}d_{0},\label{modieq:innerALk-high}
\end{align}
where $\mathcal{S}_{1,k-\frac{3}{2}}$ depends on lower order terms which are known by the induction hypothesis.
Then the compatibility condition \eqref{compatiblity:ODEsolver} for \eqref{modieq:innerALk-high}  is equivalent to
\begin{align}
\mathcal{D}_{k-\frac{1}{2}}+\sigma^{-1}\sigma_1g_{k-\frac{3}{2}}d_{\Gamma}=0\qquad \text{on }\Gamma(3\delta),\label{evolution law:h0}
\end{align}
where
\begin{align*}
\mathcal{D}_{k-\frac{1}{2}}&=\sigma^{-1}\int_{-\infty}^{+\infty}\theta_0'(\rho)\mathcal{S}_{1,k-\frac{3}{2}}\, d\rho+\partial_t d_{k-\frac{1}{2}}-\Delta d_{k-1}
\nonumber\\&\quad +\sigma^{-1}\int_{-\infty}^{+\infty}\big(\theta_0'(\rho)\big)^2\bigg(\sum\limits_{j=0,j\in\frac12\N_0}^{k-\frac{1}{2}}\hat{\ve}_j\cdot\nabla d_{k-\frac12-j}-\hat{\phi}_0d_{k-\frac{1}{2}}\bigg)d\rho.
\end{align*}
It is satisfied if $d_{k-\frac{1}{2}}$ solves
\begin{align}
&\partial_t d_{k-\frac{1}{2}}+\sigma^{-1}\int_{-\infty}^{+\infty}\big(\theta_0'(\rho)\big)^2\bigg(\sum\limits_{j=0,j\in\frac12\N_0}^{k-\frac{1}{2}}\hat{\ve}_j\cdot\nabla d_{k-\frac12-j}-\hat{\phi}_0d_{k-\frac{1}{2}}\bigg)d\rho\nonumber\\&-\Delta d_{k-1}+\sigma^{-1}\int_{-\infty}^{+\infty}\theta_0'(\rho)\mathcal{S}_{1,k-\frac{3}{2}}d\rho=0\quad \text{on }\Gamma\label{evolution law:dk12gamma}
\end{align}
and
\begin{eqnarray}\label{formula:gk32}
g_{k-\frac{3}{2}}=\left \{
\begin {array}{ll}
-\frac{\sigma\sigma_1^{-1}\mathcal{D}_{k-\frac{1}{2}}}{d_{\Gamma}}& \text{in } \Gamma(3\delta)\backslash\Gamma,\\
\\
-\sigma\sigma_1^{-1}\nn\cdot\nabla\mathcal{D}_{k-\frac{1}{2}}& \text{on }\Gamma.
\end{array}
\right.
\end{eqnarray}

\noindent
{\it Determining $\hat{\ve}_{k}$, $k\geq \frac32$, the jump conditions on $\Gamma$ and $\mathbf{l}_{k-1}$:}
Firstly, the compatibility condition \eqref{modified version of Lemma 2.4-2} for  \eqref{modieq:innerstokesk} is equivalent to
\begin{align}
  & \big(\ve_{0}^+- \ve_{0}^-\big)\partial_td_{k-1}+\sigma\Delta d_{k-1}\nabla d_\Gamma+\big(\ve_{0}^+- \ve_{0}^-\big)\frac{\ve_{0}^++\ve_{0}^-}{2}\cdot\nabla d_{k-1}
    \nonumber\\
 &\quad+\big(\sigma\Delta d_\Gamma+p_{0}^+- p_{0}^--2\big(\nu^+ D\ve_{0}^+- 2\nu^- D\ve_{0}^-\big)\big)\nabla d_{k-1}
 +\big(\ve_{0}^+- \ve_{0}^-\big)\int_{-\infty}^{\infty}\eta'\hat{\ve}_{k-1}\cdot\nabla d_{\Gamma}d\rho\nonumber\\
 &\quad+\int_{-\infty}^{\infty}\partial_{\rho}\hat{\ve}_{k-1}\big(\partial_td_{\Gamma}+\hat{\ve}_0\cdot\nabla d_{\Gamma}\big)d\rho-\int_{-\infty}^{\infty}\Div\big(2\nu(\theta_0){\big(D_d\hat{\ve}\big)_{k-1}}\big)\big)d\rho
   \nonumber\\
 &\quad-\big( 2\nu^+ D\ve_{k-1}^+- 2\nu^- D\ve_{k-1}^-\big)\nabla d_{\Gamma}+(p_{k-1}^+-p_{k-1}^-)\nabla d_{\Gamma}+\mathbf{l}_{0}d_{k-1}+\mathbf{l}_{k-1}d_{\Gamma}\nonumber\\
 & =\mathcal{S}_{2,k-\frac{3}{2}}(\theta_0,\hat{c}_{\frac{3}{2}},\cdots,\hat{c}_{k-1},d_0,\cdots,d_{k-2},\hat{\ve}_0,\cdots,\hat{\ve}_{k-\frac{3}{2}},\mathbf{l}_0,\cdots,\mathbf{l}_{k-\frac{3}{2}}).\label{oldmodieq:innerstokes1-solvability-k2}
\end{align}
Here
 $\mathcal{S}_{2,k-\frac{3}{2}}$ depends on low order terms, which were determined before.

If it is satisfied,
the solution to \eqref{modieq:innerstokesk} is given by
\begin{align}\label{formula:vek}
\hat{\ve}_k-(\mathbf{u}_0d_{k}+\mathbf{u}_kd_{\Gamma})\eta&=\ve_k^-+(\mathbf{u}_{\frac{1}{2}}d_{k-\frac{1}{2}}+\mathbf{u}_{k-\frac{1}{2}}d_{\frac{1}{2}})\eta+\int_{-\infty}^\rho\frac{1}{\nu(\theta_0(r))}\int_{-\infty}^r\mathbf{V}^{k-1}(s,x,t)\,ds\,dr\nonumber\\
&=:\ve_k^-+(\mathbf{u}_{\frac{1}{2}}d_{k-\frac{1}{2}}+\mathbf{u}_{k-\frac{1}{2}}d_{\frac{1}{2}})\eta+\mathbf{W}^{k-1},
\end{align}
where $\mathbf{V}^{k-1}$ consists of terms up to $k-1$ order.
By taking $\rho\to+\infty$ in \eqref{formula:vek} and the matching conditions for $\hat{\ve}_k$ we obtain
\begin{align*}
\ve_k^+-\ve_k^-=\mathbf{u}_0d_{k}+\mathbf{u}_kd_{\Gamma}+\mathbf{u}_{\frac{1}{2}}d_{k-\frac{1}{2}}+\mathbf{u}_{k-\frac{1}{2}}d_{\frac{1}{2}}+\mathbf{W}^{k-1}|_{\rho=+\infty},
\end{align*}
which yields
\begin{align}\label{jumpcondition:vek}
{\llbracket \ve_k^{\pm}\rrbracket\cdot\nn}&=\big(\mathbf{u}_{\frac{1}{2}}d_{k-\frac{1}{2}}+\mathbf{u}_{k-\frac{1}{2}}d_{\frac{1}{2}}+\mathbf{W}^{k-1}|_{\rho=+\infty}\big)\cdot\nn+\mathbf{u}_0\cdot\nn d_{k}=:\hat{a}_{k-\frac{1}{2}}\quad \text{on}\ \Gamma,
\end{align}
 where we have used $\mathbf{u}_0\cdot\nn|_{\Gamma}=0$ due to Proposition~\ref{prop:1},  and
\begin{align}\label{jumpcondition:vek-0}
\llbracket \ve_k^{\pm}\rrbracket\cdot\btau&=\big(\mathbf{u}_{\frac{1}{2}}d_{k-\frac{1}{2}}+\mathbf{u}_{k-\frac{1}{2}}d_{\frac{1}{2}}+\mathbf{W}^{k-1}|_{\rho=+\infty}\big)\cdot\btau+\mathbf{u}_0\cdot\btau d_{k}
=:\check{a}_{k-\frac{1}{2}}+\mathbf{u}_0\cdot\btau d_{k}
\end{align}
on $\Gamma$.
Since by the induction hypothesis one assumes that the compatibility condition \eqref{oldmodieq:innerstokes1-solvability-k2} for $k-1$ instead of $k$ is already satisfied, one obtains 
\begin{align}\label{newformula:vek-1}
\hat{\ve}_{k-1}=&\ve_{k-1}^+\eta(\rho)+\ve_{k-1}^-(1-\eta(\rho))+(1-\eta(\rho))\int_{-\infty}^\rho\frac{1}{\nu(\theta_0(r))}\int_{-\infty}^r\mathbf{V}^{k-2}(s,\cdot)\,ds\,dr
\nonumber\\&-\eta(\rho)\int_\rho^{+\infty}\frac{1}{\nu(\theta_0(r))}\int_{-\infty}^r\mathbf{V}^{k-2}(s,\cdot)\,ds\,dr,
\end{align}
where $\mathbf{V}_{-1}\equiv 0$.
Inserting this we can rewrite  \eqref{oldmodieq:innerstokes1-solvability-k2} as
\begin{align}
\mathcal{J}_{k-1}+\mathbf{l}_{k-1}d_{\Gamma}=\tilde{\mathcal{S}}_{k-\frac{3}{2}},\label{newmodieq:innerstokes1-solvability-k2}
\end{align}
where
\begin{align}
\mathcal{J}_{k-1}&=\big(\ve_{0}^+- \ve_{0}^-\big)\partial_td_{k-1}+\sigma\Delta d_{k-1}\nabla d_\Gamma+\big(\ve_{0}^+- \ve_{0}^-\big)\frac{\ve_{0}^++\ve_{0}^-}{2}\cdot\nabla d_{k-1}
    \nonumber\\
 &+\big(\sigma\Delta d_\Gamma+p_{0}^+- p_{0}^--2\big(\nu^+ D\ve_{0}^+- 2\nu^- D\ve_{0}^-\big)\big)\nabla d_{k-1}+\mathbf{l}_{0}d_{k-1}
 \nonumber\\
 &+\big(\ve_{0}^+- \ve_{0}^-\big)\frac{\ve_{k-1}^++\ve_{k-1}^-}{2}\cdot\nabla d_{\Gamma}+ \big(\ve_{k-1}^+- \ve_{k-1}^-\big)\partial_td_{\Gamma}+\big(\ve_{k-1}^+- \ve_{k-1}^-\big)\frac{\ve_{0}^++\ve_{0}^-}{2}\cdot\nabla d_{\Gamma}
\nonumber\\&-\overline{\nu}\left(( \ve_{k-1}^+- \ve_{k-1}^-)\cdot\nabla\right) \nabla d_{\Gamma}-\overline{\nu}\big( \ve_{k-1}^+- \ve_{k-1}^-\big)\Delta d_{\Gamma}-\overline{\nu}(\nabla d_{\Gamma}\cdot \nabla)\big( \ve_{k-1}^+- \ve_{k-1}^-\big)
  \nonumber\\&{-\overline{\nu}\left(( \ve_{0}^+- \ve_{0}^-)\cdot\nabla \right)\nabla d_{k-1}-\overline{\nu}\big( \ve_{0}^+- \ve_{0}^-\big)\Delta d_{k-1}-\overline{\nu}\nabla\big( \ve_{0}^+- \ve_{0}^-\big)\nabla d_{k-1}}
 \nonumber\\&-\big( 2\nu^+ D\ve_{k-1}^+- 2\nu^- D\ve_{k-1}^-\big)\nabla d_{\Gamma}+(p_{k-1}^+-p_{k-1}^-)\nabla d_{\Gamma}.
\end{align}
Here we have used
\begin{align*}
 \int_{-\infty}^{\infty}\Div\big(2\nu(\theta_0)\big(D_d\hat{\ve}\big)_{k-1}\big)\big)d\rho&=\overline{\nu}\left(( \ve_{k-1}^+- \ve_{k-1}^-)\cdot\nabla\right) \nabla d_{\Gamma}+\overline{\nu}\big( \ve_{k-1}^+- \ve_{k-1}^-\big)\Delta d_{\Gamma}\nonumber\\&\quad +\overline{\nu}(\nabla d_{\Gamma}\cdot \nabla)\big( \ve_{k-1}^+- \ve_{k-1}^-\big)
 +\overline{\nu}\left(( \ve_{0}^+- \ve_{0}^-)\cdot\nabla \right)\nabla d_{k-1} \nonumber\\&\quad +\overline{\nu}\big( \ve_{0}^+- \ve_{0}^-\big)\Delta d_{k-1}+\overline{\nu}\nabla\big( \ve_{0}^+- \ve_{0}^-\big)\nabla d_{k-1},
\end{align*}
which is shown in the same way as \eqref{preeq:inner1-5}.

Therefore \eqref{oldmodieq:innerstokes1-solvability-k2} is satisfied if
\begin{align}
\tilde{\mathcal{S}}_{k-\frac{3}{2}}=&\sigma\Delta d_{k-1}\nabla d_\Gamma+\big(\sigma\Delta d_\Gamma+\llbracket p_{0}^\pm\rrbracket-2\llbracket\nu^\pm D\ve_{0}^\pm\rrbracket{-\overline{\nu}\llbracket\nabla\ve_{0}^\pm\rrbracket\big)\nabla d_{k-1}}+\mathbf{l}_{0}d_{k-1}
 \nonumber\\
 &+ \llbracket\ve_{k-1}^\pm\rrbracket\big(\partial_td_{\Gamma}+\ve_{0}^\pm\cdot\nabla d_{\Gamma}-\overline{\nu}\Delta d_{\Gamma}-\overline{\nu}\nabla^2d_{\Gamma}\big)
\nonumber\\&-\overline{\nu}\llbracket\nabla\ve_{k-1}^\pm\rrbracket\cdot\nabla d_{\Gamma}-2\llbracket\nu^\pm D\ve_{k-1}^\pm\rrbracket\nabla d_{\Gamma}+\llbracket p_{k-1}^\pm\rrbracket\nabla d_{\Gamma} \quad \text{on }\Gamma
\label{jumpcondition:lk-1}
\end{align}
and
\begin{eqnarray}\label{formula:lk-1}
\mathbf{l}_{k-1}=\left \{
\begin {array}{ll}
\frac{\tilde{\mathcal{S}}_{k-\frac{3}{2}}-\mathcal{J}_{k-1}}{d_{\Gamma}}& \text{in } \Gamma(3\delta)\backslash\Gamma,\\
\\
\nn\cdot\nabla\big(\tilde{\mathcal{S}}_{k-\frac{3}{2}}-\mathcal{J}_{k-1}\big)& \text{on } \Gamma.
\end{array}
\right.
\end{eqnarray}
In order to satisfy the matching conditions for $\hat{\ve}_k$ we define  $\mathbf{u}_{k}$ by
\begin{eqnarray}\label{formula:uk}
\mathbf{u}_k=\left \{
\begin {array}{ll}
\frac{\mathbf{v}_k^{+}-\mathbf{v}_k^{-}-\mathbf{u}_{0}d_{k}-\mathbf{u}_{\frac{1}{2}}d_{k-\frac{1}{2}}-\mathbf{u}_{k-\frac{1}{2}}d_{\frac{1}{2}}-\mathbf{W}^{k-1}|_{\rho=+\infty}}{d_{\Gamma}}&\ \text{on } \Gamma(3\delta)\setminus\Gamma,\\
\nn\cdot\nabla\big(\mathbf{v}_k^{+}-\mathbf{v}_k^{-}-\mathbf{u}_{0}d_{k}-\mathbf{u}_{\frac{1}{2}}d_{k-\frac{1}{2}}-\mathbf{u}_{k-\frac{1}{2}}d_{\frac{1}{2}}-\mathbf{W}^{k-1}|_{\rho=+\infty}\big)&\ \text{on } \Gamma.
\end{array}
\right.
\end{eqnarray}
Then
\begin{align}\label{newformula:vek}
\hat{\ve}_k(\rho,x,t)=&\ve_k^+(x,t)\eta(\rho)+\ve_k^-(x,t)(1-\eta(\rho))+(1-\eta(\rho))\int_{-\infty}^\rho\frac{1}{\nu(\theta_0(r))}\int_{-\infty}^r\mathbf{V}^{k-1}(s,x,t)\,ds\,dr
\nonumber\\&-\eta\int_\rho^{+\infty}\frac{1}{\nu(\theta_0(r))}\int_{-\infty}^r\mathbf{V}^{k-1}(s,x,t)\,ds\,dr,
\end{align}
satisfies the inner-outer matching conditions \eqref{eq:matchcon}.

In order to determine $\hat{p}_{k-1}$ we multiply \eqref{modieq:innerstokesk} by $\nabla d_{\Gamma}$ and use \eqref{modieq:innerdivk}. This yields
\begin{align}
  &\partial_{\rho}\bigg(\hat{p}_{k-1}+\hat{\ve}_{k-1}\cdot\nabla d_{\Gamma}\partial_t d_{\Gamma}+\hat{\ve}_{k-1}\cdot\nabla d_{\Gamma}\hat{\ve}_{0}\cdot\nabla d_{\Gamma} -\big(2\nu(\theta_0
    )D\hat{\ve}_{k-1}:\nabla d_{\Gamma}\otimes\nabla d_{\Gamma}\big)\nonumber\\&\qquad-\int_{-\infty}^\rho\Div\big(2\nu(\theta_0(r)){\big(D_d\hat{\ve}\big)_{k-1}}\big)(r,\cdot)\big)\cdot\nabla d_{\Gamma}\,dr
+\nu(\theta_0)\Div\hat{\ve}_{k-1}\nonumber\\&\qquad+\mathbf{l}_{k-1}d_{\Gamma}\eta\cdot\nabla d_{\Gamma}
  -\nu(\theta_0)\mathcal{A}^{k-1}+\int_{-\infty}^\rho\widetilde{\mathbf{V}}^{k-1}(r,\cdot)\cdot\nabla d_{\Gamma}dr\bigg)=0,\label{modieq:innerstokesk-new}
\end{align}
where $\mathcal{A}^{k-1}$ and $\widetilde{\mathbf{V}}^{k-1}$ consist of some terms up to $k-1$ order.
Thus
\begin{align}
\hat{p}_{k-1}&=\nu(\theta_0(\rho))\mathcal{A}^{k-1}-\nu(\theta_0(\rho))\Div\hat{\ve}_{k-1}
               +\big(\ve_{k-1}^+\eta+\ve_{k-1}^-(1-\eta)\big)\cdot\nabla d_{\Gamma}\partial_t d_{\Gamma}\nonumber\\&\ -\hat{\ve}_{k-1}\cdot\nabla d_{\Gamma}\partial_t d_{\Gamma} +2\nu(\theta_0(\rho))D\hat{\ve}_{k-1}:\nabla d_{\Gamma}\otimes\nabla d_{\Gamma}+p_{k-1}^+\eta+p_{k-1}^-(1-\eta)\nonumber\\
&\  -\left(2\nu^+ D\ve_{k-1}^+\eta+ 2\nu^- D\ve_{k-1}^-(1-\eta)\right):\nabla d_{\Gamma}\otimes\nabla d_{\Gamma}
\nonumber\\&\ +\eta\int_{+\infty}^\rho\Div\big(2\nu(\theta_0(r)){\big(D_d\hat{\ve}\big)_{k-1}}\big)\big)\cdot\nabla d_{\Gamma}\,dr+(1-\eta)\int_{-\infty}^\rho\Div\big(2\nu(\theta_0(r){\big(D_d\hat{\ve}\big)_{k-1}}\big)\big)\cdot\nabla d_{\Gamma}\,dr
\nonumber\\&\ +\eta\int_\rho^{+\infty}\widetilde{\mathbf{V}}^{k-1}(r,\cdot)\cdot\nabla d_{\Gamma}\, dr-(1-\eta)\int_{-\infty}^\rho\widetilde{\mathbf{V}}^{k-1}(r,\cdot)\cdot\nabla d_{\Gamma}\,dr \quad \text{in }\Gamma(3\delta),
\label{modieq:innerstokesk-new1}
\end{align}
which satisfies  the inner-outer matching conditions \eqref{eq:matchcon}. In summary we have:

\begin{lem}[The $k$-th order terms]
 \label{k-thorder}~\\
Let $k\geq\frac12$  and  all functions with negative index be supposed to be zero.  Then there are
smooth functions
\[
\hat{\mathbf{v}}_{k},\mathbf{v}_{k}^{\pm},\mathbf{u}_{k},\mathbf{l}_{k-1},\hat{c}_{k},c_{k}^{\pm},g_{k-\frac{3}{2}},d_{k},\hat{p}_{k-1},p_{k}^{\pm},
\]
which are bounded on their respective domains, such that for the $k$-th
order the outer equations  \eqref{eq:c0out},  \eqref{eq:stokes-Outerk},
\eqref{eq:Outervpdefine}, the inner equations \eqref{modieq:innerstokesk},
\eqref{modieq:innerdivk}, \eqref{modieq:innerALk},
the inner-outer matching conditions \eqref{eq:matchcon} are satisfied. Moreover, $(\mathbf{v}_{k}^{\pm},
p_{k}^{\pm},d_{k})$ satisfies
\begin{alignat}{2}
\partial_t\mathbf{v}_{k}^{\pm}-\nu^\pm\Delta\ve_{k}^{\pm}+\nabla p_{k}^{\pm}  &=0  &\ &\text{in }\Omega^{\pm},\label{system:korder-1}\\
\operatorname{div}\mathbf{v}_{k}^{\pm} & =0   &\ &\text{in }\Omega^{\pm},\label{system:korder-2}\\
{\llbracket \ve_k^\pm\rrbracket\cdot\nn}& =\hat{a}_{k-\frac{1}{2}}   &\ &  \text{on }\Gamma,\label{system:korder-3}\\
{\llbracket \ve_k^\pm\rrbracket\cdot\btau}& =\check{a}_{k-\frac{1}{2}} -\mathbf{u}_0\cdot\btau d_{k} &\ &  \text{on }\Gamma,\label{system:korder-33}\\
\llbracket 2\nu^\pm D\ve_{k}^\pm -p_{k}^\pm \tn{I}\rrbracket\no_{\Gamma_t}+\ol{\nu}\llbracket \nabla\ve_{k}^\pm \rrbracket\no_{\Gamma_t} &+\ol{\nu}\llbracket \ve_{k}^\pm \rrbracket\Delta d_{\Gamma}  &&\nonumber\\+\ol{\nu}\llbracket \ve_k^\pm\rrbracket\cdot\nn_{\Gamma_t}\Delta d_{\Gamma}\nn_{\Gamma_t}&= \sigma\Delta d_{k}\no_{\Gamma_t} -\overline{\nu}\llbracket\nabla\ve_{0}^\pm\rrbracket\nabla d_{k}+\mathbf{l}_{0}d_{k} &\ &  \text{on }\Gamma,\label{system:korder-4}\\
\partial_t d_{k}+\ve_0^{\pm}\cdot\nabla d_{k}+\tfrac{\ve_k^{+}+\ve_k^{-}}{2}\cdot\no_{\Gamma_t}-\phi_0d_{k}&=b_{k-\frac{1}{2}}&\quad& \text{on}\ \Gamma, \label{system:korder-5}\\
\mathbf{v}_k^{-} & =\ol{a}_k\no_{\partial\Omega},  &\ & \text{on }\partial\Omega,\label{system:korder-6}
\end{alignat}
where $\ol{a}_k = \frac1{|\partial\Omega|}\int_{\Gamma_t} (\check{a}_{k-\frac{1}{2}} -\mathbf{u}_0\cdot\btau d_{k})\,d\sigma $,
as well as \eqref{eq:dk}.
Here  \eqref{system:korder-4} and \eqref{system:korder-5} come from  \eqref{jumpcondition:lk-1} and \eqref{evolution law:dk12gamma}(with $k$ instead of  $k-1$, $k-\frac12$, respectively) and $\hat{a}_{k-\frac12}$, $\check{a}_{k-\frac12}$, $b_{k-\frac{1}{2}}$ depend only on the terms up to $k-\frac{1}{2}$ order.
Furthermore, the compatibility condition \eqref{jumpcondition:lk-1} is satisfied for $k$ instead of $k-1$ and $\mathbf{v}_{k}^{\pm}$,  $c_{k}^{\pm}$ and
$p_{k}^{\pm}$ are extended onto $\Omega^{\pm}\cup\Gamma(3\delta)$ as in Remark~\ref{Outer-Rem}.
\end{lem}
\begin{proof}
  The lemma is proved by mathematical induction with respect to $k\in \frac12\N$, where the beginning of the induction is given by Lemma~\ref{zeroorder}. In Theorem \ref{solvingkorder} in the appendix we will show solvability of the system \eqref{system:korder-1}-\eqref{system:korder-6}, which will be smooth due to Remark~\ref{rem:Solvability}. In the induction hypothesis we assume that
\begin{align*}
\{
(\hat{\mathbf{v}}_{i},\mathbf{v}_{i}^{\pm},\mathbf{u}_{i},\mathbf{l}_{i-1},\hat{c}_{i},g_{i-\frac{3}{2}},d_{i},\hat{p}_{i-1},p_{i}^{\pm}): 0\leq i\leq k-\tfrac12
\}
\end{align*}
are known and satisfy the statements of the lemma with $i$ instead of $k$ for all $0\leq i\leq k-\frac12$.
Then we obtain the terms for  $i=k$  by the following four steps:

\smallskip

\noindent\textbf{Step 1:} By the induction hypothesis $\mathcal{S}_{1,k-\frac32}$ is known. Since $d_{k-\frac12}$ solves \eqref{system:korder-5} with $k-\frac12$ instead of $k$, the compatibility condition for \eqref{modieq:innerALk} on $\Gamma$ is satisfied. Moreover, defining $g_{k-\frac32}$ by \eqref{formula:gk32} the  compatibility condition for \eqref{modieq:innerALk} are satisfied on $\Gamma(3\delta)$.  Hence we can determine $\hat{c}_k$ as the solution of \eqref{modieq:innerALk} for all $\rho\in\R$, $(x,t)\in \Gamma(3\delta)$.

\smallskip

\noindent\textbf{Step 2:} We have seen that the compatibility condition \eqref{oldmodieq:innerstokes1-solvability-k2} for solving \eqref{modieq:innerstokesk} is equivalent to \eqref{jumpcondition:lk-1} and \eqref{formula:lk-1}. Here \eqref{jumpcondition:lk-1} is satisfied since $(\ve^\pm_{k-1}, p^\pm_{k-1}, d_{k-1})$ solve \eqref{system:korder-1}-\eqref{system:korder-6} by assumption in which  we have used \eqref{formula:secondderivative d} to rewrite  \eqref{jumpcondition:lk-1} as \eqref{system:korder-4}. Moreover, if we define $\mathbf{l}_{k-1}$ by \eqref{formula:lk-1}, the compatibility conditions for \eqref{modieq:innerstokesk} are satisfied. Now we can determine $\hat{\ve}_{k}$ by \eqref{modieq:innerstokesk} on $\Gamma$ uniquely. (Note that \eqref{modieq:innerstokesk} determines $\hat{\ve}_k$ on $\Gamma(3\delta)\setminus \Gamma$ up to $\ue_k$, which is not determined yet.) Moreover, this determines $\hat{p}_{k-1}$ by \eqref{modieq:innerstokesk-new1} on $\Gamma(3\delta)$ since \eqref{formula:lk-1} holds and $\mathcal{A}^{k-1}$, $\widetilde{\mathbf{V}}^{k-1}$ are known.

\smallskip

\noindent\textbf{Step 3:} Since $\hat{p}_{k-1}$ is determined, $\mathcal{S}_{1,k-1}$ is known on $\Gamma$ and we can determine $(\ve^\pm_k, p^\pm_k, d_k)$ as solution of \eqref{system:korder-1}-\eqref{system:korder-6}, cf.\ Theorem~\ref{solvingkorder} in the appendix.

\smallskip

\noindent\textbf{Step 4:} Using that $(\ve_k^\pm, p^\pm_k)$ are known, $\ue_k$ is now determined uniquely by \eqref{formula:uk} and we can determine $\hat{\ve}_k$ by \eqref{modieq:innerstokesk} on $\Gamma(3\delta)$ uniquely.

\smallskip

\noindent\textbf{Step 5:} Using that $d_k$ is determined on $\Gamma$, one can integrate \eqref{eq:dk} in normal direction to determine $d_k$ uniquely on $\Gamma(3\delta)$.

Finally, we note that $(\hat{\ve}_k, \hat{p}_{k-1}, \hat{c}_k)$ satisfy the matching conditions on $\Gamma(3\delta)$ by construction, in particular because of the choice of $\ue_k$.
\end{proof}

\subsection{Summary of the Construction}
The result of this section can be summarized as follows:
\begin{thm}\label{thm:Approx1}
Let $N\in \frac12\N$. Then there are smooth $(\tilde{c}_A^{in},\tilde{\ve}_A^{in},\tilde{p}_A^{in})$ defined in $\Gamma(3\delta)$ and smooth $(c_A^\pm, \tilde{\ve}_A^\pm, \tilde{p}_A^\pm)$ defined on $\ol{\Omega}\times [0,T_0]$ such that:
\begin{enumerate}
\item \emph{Inner expansion:} In $\Gamma(3\delta)$ we have
    \begin{alignat}{2}\nonumber
      \partial_t \tilde{\ve}_A^{in}+ \tilde{\ve}_A^{in}\cdot \nabla \tilde{\ve}_A^{in} -\Div (2\nu(\tilde{c}_A^{in})D \tilde{\ve}_A^{in}) +\nabla \tilde{p}_A^{in}&= -\eps \Div (\nabla \tilde{c}_A^{in}\otimes \nabla \tilde{c}_A^{in}) + R_\eps,
      \\\nonumber
      \Div \tilde{\ve}_A^{in} &=\,  
      G_\eps,\\
      \partial_t \tilde{c}_A^{in} + \tilde{\ve}_A^{in}\cdot \nabla \tilde{c}_A^{in} &= \eps^{\frac12}\Delta \tilde{c}_A^{in} -\eps^{-\frac32} f'(\tilde{c}_A^{in})+ s_\eps,
\end{alignat}
where
\begin{alignat}{2}\label{eq:RemainderApprox1}
  \|(R_\eps, \partial_tG_\eps , s_\eps)\|_{L^\infty(\Gamma(3\delta)))}&\leq C\eps^{N+1},
  \\\label{eq:RemainderApprox2}
 \|G_\eps \|_{L^\infty(\Gamma(3\delta))}&\leq C\eps^{N+2}.
  \end{alignat}
\item \emph{Outer expansion:}  In $\Omega^\pm$ we have $c_A^\pm \equiv \pm 1$ and
  \begin{alignat}{2}
  \partial_t \tilde\ve_A^\pm+ \tilde\ve_A^\pm\cdot \nabla \tilde\ve_A^\pm -\nu^\pm \Delta \tilde\ve_A^\pm +\nabla \tilde{p}_A^\pm&= R^\pm_\eps,
\\\nonumber
\Div \tilde\ve_A^\pm &=\, 0,\\
\tilde\ve_A^\pm|_{\partial\Omega} &=\, \ol{a}_\eps \no_{\partial\Omega}\quad \text{on }\partial\Omega\times [0,T_0],
\end{alignat}
where $\ol{a}_\eps\colon [0,T]\to \R$ is smooth and
\begin{equation*}
  \|R_\eps^\pm \|_{L^\infty(\Omega\times [0,T_0])}\leq C\eps^{N+2}\qquad \text{for all }\eps\in (0,1).
\end{equation*}
\item \emph{Matching condition:} For every $\beta\in \N_0^n$ we have for some  $\alpha>0$, $C(M)>0$
    \begin{equation*}
  \begin{split}
    &\| \p_x^{\beta}( \tilde\ve_A^{in}-\tilde\ve_A^{+}
  \chi_+-\tilde\ve_A^{-}\chi_-)\|_{L^\infty(\Gamma(3\delta)\setminus \Gamma(\delta) )}\leq C(M) e^{-\frac{\alpha\delta}{2\eps}},\\
  &\| \p_x^{\beta}(  \tilde{p}_A^{in}-\tilde{p}_A^{+}
  \chi_+-\tilde{p}_A^{-}\chi_-)\|_{L^\infty(\Gamma(3\delta)\setminus \Gamma(\delta) )}\leq C(M) e^{-\frac{\alpha\delta}{2\eps}},\\
  &\|   \p_x^{\beta}(\tilde{c}_A^{in}-c_A^{+}
  \chi_+-c_A^{-}\chi_-)\|_{L^\infty(\Gamma(3\delta)\setminus \Gamma(\delta) )}\leq C(M) e^{-\frac{\alpha\delta}{2\eps}}
  \end{split}
\end{equation*}
for all $\eps \in (0,1)$.
\end{enumerate}
  \end{thm}
\begin{proof}
We define
\begin{align*}
  c_A^{\pm}(x,t)  &=\sum_{k\in \frac12 \N_0, k\leq N+2}\eps^{k}c_{k}^{\pm}(x,t),\quad \tilde{\mathbf{v}}_A^{\pm}(x,t) =\sum_{k\in \frac12 \N_0, k\leq N+2}\eps^{k}\mathbf{v}_{k}^{\pm}(x,t), \\
  \tilde{p}_A^{\pm}(x,t) &= \sum_{k\in \frac12 \N_{-1}, k\leq N+2}\eps^{k}p_{k}^{\pm}(x,t),
\end{align*}
and
\begin{align*}
  \tilde{c}_A^{in}(x,t)  &=\sum_{k\in \frac12 \N_0, k\leq N+2}\eps^{k}\hat{c}_{k}(\tfrac{d_\eps}{\eps},x,t),\quad \tilde{\mathbf{v}}_A^{in}(x,t) =\sum_{k\in \frac12 \N_0, k\leq N+2}\eps^{k}\hat{\mathbf{v}}_{k}(\tfrac{d_\eps}{\eps},x,t), \\
  \tilde{p}_A^{in}(x,t) &= \sum_{k\in \frac12 \N_{-1}, k\leq N+2}\eps^{k}\hat{p}_{k}(\tfrac{d_\eps}{\eps},x,t),
\end{align*}
where
\begin{equation}
  d_A(x,t)=\sum_{k\in \frac12 \N_0, k\leq N+2}\eps^{k}d_k(x,t).\label{def:dA}
\end{equation}
as well as $\ol{a}_\eps(t) = \sum_{k\in \frac12 \N_0, k\leq N+2}\eps^{k}\ol{a}_{k}(t)$.
From the construction one can verify the statements of Theorem~\ref{thm:Approx1} in the same way as e.g.\ in \cite[Section~4]{AbelsMarquardt2}.
\end{proof}
\begin{rem}
  We note that $d_A$ defined in \eqref{def:dA} satisfies
  \begin{equation}\label{eq:ErrordA}
    |\nabla d_A|^2= 1+ O(\eps^{N+3})\qquad \text{as }\eps\to 0
  \end{equation}
  with respect to $C^k(\Gamma(3\delta))$ for every $k\in\N$ by the construction \eqref{eq:dk}.
\end{rem}

\section{Refined Approximate Solutions}\label{sec:ApproxSolutions}

In this section we refine the approximate solutions constructed in the previous section by adding a few terms to obtain:
\begin{thm}\label{thm:approx}
  Let $M>0$, $\ue=\ue(\eps)\in L^2(0,T_\eps;H^1(\Omega)^2\cap L^2_\sigma(\Omega))$ be given for some $T_\eps \in (0,T_0]$, $\eps\in (0,1)$. Moreover, let
  \begin{align*}
    P_M(\ue) =
    \begin{cases}
      \ue &\text{if } \|\ue \|_{L^2(0,T_\eps; H^1(\Omega))}+ \|\ue \|_{H^{\frac12}(0,T_\eps; L^2(\Omega))} \leq M,\\
      \frac{M \ue}{\|\ue \|_{L^2(0,T_\eps; H^1)}+ \|\ue \|_{H^{\frac12}(0,T_\eps; L^2)}} &\text{else}.
    \end{cases}
  \end{align*}
Then there are $c_A\in H^1(0,T_\eps;L^2(\Omega))\cap L^2(0,T_\eps;H^2(\Omega))$, $p_A\in L^2(0,T_\eps;H^1(\Omega))$, and $\ve_A\in H^1(0,T_\eps; V(\Omega)')\cap L^2(0,T_\eps;H^1(\Omega)^2)$, $\we_\eps\in H^1(0,T_\eps;L^2(\Omega))\cap L^2 (0,T_\eps;H^1(\Omega))$ such that
  \begin{alignat}{1}\label{eq:ApproxS1}
    \partial_t \ve_A+ \ve_A\cdot \nabla \ve_A -\Div (2\nu(c_A)D \ve_A) +\nabla p_A&= -\eps \Div (\nabla c_A\otimes \nabla c_A) + R_\eps,
\\\label{eq:ApproxS2}
      \Div \ve_A &=\,                          G_\eps,\\\label{eq:ApproxS3}
      \partial_t c_A + \big(\ve_A+\eps^{N+\frac14}((P_M(\ue)+\we_\eps)|_{X_\eps(0,S_\eps,t)}-\we_\eps)\big)\cdot \nabla c_A &= \eps^{\frac12}\Delta c_A -\eps^{-\frac32} f'(c_A)
      +s_\eps,\\
      (\ve_A,c_A)|_{\partial\Omega}&= (0,-1),
\end{alignat}
where $s_\eps=s_\eps^1+s_\eps^2$ with $\supp s^2_\eps \subseteq \Gamma(2\delta)$, $\Div\we_\eps=0$ and 
\begin{alignat}{2}
  \|R_\eps\|_{L^2(0,T; (H^1(\Omega))'))}&\leq C(M) \eps^{N+\frac14}(T^{\frac12}+\eps^{\frac14}),\label{estimation:Reps}\\
  \|G_\eps\|_{H^1(0,T_\eps;L^2(\Omega))}&\leq C(M)\eps^{N+1},\label{estimation:Geps}\\
  \|\we_\eps\|_{L^2(0,T_\eps;H^1(\Omega))}+\|\we_\eps\|_{L^\infty(0,T_\eps;L^2(\Omega))}&\leq C(M)\label{eq_estimate_wA}\\
  \|s_\eps^1\|_{L^2(\Omega\times (0,T_\eps))}\leq C(M)\eps^{N+\frac12},\quad  \|s_\eps^2\|_{L^2(0,T_\eps; (X_\eps)')}&\leq C(M)\eps^{N+\frac34},\label{estimation:Seps}\\\label{estimation:Seps'}
  \|s_\eps^2\|_{L^2(\Omega\times(0,T_\eps))}&\leq C(M)\eps^{N-\frac14},
  \end{alignat}
uniformly in $T\in (0,T_\eps]$  for some $C(M)>0$ independent of $\eps\in (0,1)$.
  Moreover, $c_A\equiv \pm 1$ in $\Omega^\pm\setminus \Gamma(3\delta)$, $\partial_t c_A$, $\nabla c_A$ are supported in $\ol{\Gamma(3\delta)}$ and $\supp s_\eps \subseteq \ol{\Gamma(5\delta/2)}$. 
\end{thm}
\begin{rem}\label{th_rem_U}
	Later we will choose $\ue=\frac{\we_1}{\eps^{N+\frac14}}$, cf.~also \eqref{def:ue}. 
\end{rem}

Let $(\tilde{c}_A^{in},\tilde{\ve}_A^{in},\tilde{p}_A^{in})$, $(c_A^\pm, \tilde{\ve}_A^\pm, \tilde{p}_A^\pm)$, and $\overline{a}_\eps (t)$ be as in Theorem~\ref{thm:Approx1}, i.e., the inner and outer pieces of the approximate solution of the Navier-Stokes/Allen-Cahn system constructed in the previous section. Moreover, let $d_A$ be as in \eqref{def:dA} and define
\begin{equation}
  \label{eq:deps}
  d_\eps(x,t)= d_{\Gamma}(x,t)+ \theta(d_\Gamma(x,t)) (d_A(x,t)-d_{\Gamma}(x,t))\quad \text{for all }(x,t)\in \Gamma(3\delta),
\end{equation}
where $\theta\in C^\infty_0(\R)$ with $\theta(r)=1$ if $r\in [-\tfrac{5\delta}2, \tfrac{5\delta}2]$ and $\supp \theta \subseteq (-3\delta,3\delta)$. Moreover, let
\begin{equation}
  \label{eq:Seps}
  S_\eps(x,t)= S_0(x,t)+ \eps^{\frac12} S_{1/2}(x,t)+ \eps S_1(x,t)+ \eps^{\frac32} S_{3/2}(x,t)\quad \text{for all }(x,t)\in \Gamma(3\delta),
\end{equation}
where $S_{1/2}, S_1, S_{3/2}$ are determined such that
\begin{equation*}
  \nabla S_\eps (x,t)\cdot \nabla d_\eps(x,t)= O(\eps^2) \qquad \text{in }\Gamma(\tfrac{5\delta}2)
\end{equation*}
with respect to any $C^k$-norm, $k\in\N$. Since $\nabla S_0\cdot \nabla d_\Gamma=0$, this leads to the system of first order partial differential equations
\begin{align*}
  \nabla S_j\cdot \no_{\Gamma} &= -\sum_{k=0,\ldots, j-\frac12} \nabla S_k\cdot \nabla d_{k-j}\qquad \text{in } \Gamma(\tfrac{5\delta}2)
\end{align*}
for $j=\frac12,1,\frac32$, which can be solved together with $S_j|_{\Gamma}=0$ by integration in normal direction/the method of characteristics. Moreover, we extend $S_j$ to $S_j\colon \Gamma(3\delta)\to \T^1$ such that $\supp S_j \subseteq \Gamma(\delta')$ for some $\delta'\in (\tfrac{5\delta}2,3\delta)$.  
Then the assumptions \eqref{eq_d_eps}-\eqref{eq_tilde_deps_0} are satisfied with $\eta=\frac12$. Moreover, let $\rho$ be defined as in \eqref{eq:StretchedVariable} in the following. Since $d_\eps=d_A$ in $\Gamma(\tfrac{5\delta}2)$, the definition of $\rho$ coincides with the definition of $\rho$ in Section~\ref{sec:MatchedAsymptotics}, proof of Theorem~\ref{thm:Approx1}, respectively, in $\Gamma(\tfrac{5\delta}2)$. In the following we will only use the identities from  Section~\ref{sec:MatchedAsymptotics} in the latter domain.

We will now define the refined approximate solution as
\begin{alignat*}{1}
  c_{A}(x,t)&=\zg c_A^{in}(x,t)+(1-\zg )\(c_A^+
  \chi_+ +c_A^-\chi_- \),\\
     \ve_A(x,t)&=\zg \ve_A^{in}(x,t)+(1-\zg )\left(\ve_A^{+}(x,t)
  \chi_+ +\ve_A^-(x,t)\chi_-\right){- \mathbf{N}\bar{a}_\eps(t)},\\
  p_A(x,t)&=\zg p_A^{in}(x,t)+(1-\zg )
  \left(p_A^{+}(x,t)
  \chi_+ +p_A^{-}(x,t)\chi_- \right),
\end{alignat*}
    where $\zeta\colon \R\to [0,1]$ is smooth such that $\supp \zeta \subseteq [-\frac{5\delta}2,\frac{5\delta}2]$ and $\zeta\equiv 1$ on $[-2\delta,2\delta]$, $\mathbf{N}\colon \Omega \to \R^2$ is a smooth vector field such that $\mathbf{N}|_{\partial \Omega}= \no_{\partial\Omega}$ and $\supp \mathbf{N}\cap \Gamma(3\delta)=\emptyset$,  $c_A^{\pm}=\pm 1$ (as before) and $\chi_\pm=\chi_{\O^\pm_t}(x)$, and we use the following refined ansatz for the inner and outer expansion
\begin{equation}\label{eq:innerexpan'}
\begin{split}
  c_A^{in} (x,t)&=
  \tilde{c}_A^{in}(\rho,x,t)+\eps^{N-\frac34} \theta_0'(\rho) \bar{h}_{N,\eps}(S_\eps(x,t),t),\\
  \bar{h}_{N,\eps}(s,t)&=h_{N-\frac34,\eps}(s,t)+\sqrt{\eps}h_{N-\frac14,\eps}(s,t),\\
  \ve_A^{in} (x,t)&=\tilde{\ve}_A^{in}(x,t) + \eps^{N+\frac14}\left(\we_\eps^+(x,t)\chi_{\Omega^+}+\we_\eps^-(x,t)\chi_{\Omega^-}\right)\\
  p_A^{in} (x,t)&=\tilde{p}_A^{in}(x,t)+ \eps^{N+\frac14}\left(q_\eps^+(x,t)\chi_{\Omega^+}+q_\eps^-(x,t)\chi_{\Omega^-}\right),\\
    \ve_A^{\pm} (x,t)&=\tilde{\ve}_A^{\pm}(x,t) +\eps^{N+\frac14}\we^\pm_\eps(x,t),\quad
  p_A^{\pm} (x,t)=\tilde{p}_A^{\pm}(x,t)+\eps^{N+\frac14}q^\pm_\eps(x,t).
\end{split}
\end{equation}
We note that we do not use $\rho$-dependent terms in the extra-terms in $\ve_A^{in}$ and $p_A^{in}$ of order $\eps^{N+\frac14}$. This ansatz differs significantly from the construction of the other terms. It turned out that it not only simplifies the treatment of several remainder terms. It also provides sufficiently good remainder estimates for our analysis, which we could not obtain before.
Here it is essential that $(\we_\eps,q_\eps, h_{N-\frac34,\eps})$ solve the following linearized two-phase flow system
    \begin{alignat}{2}\label{eq:ApproxStokesMCF1}
      \partial_t \we_\eps^\pm +\we_\eps^\pm \cdot \nabla \ve_0^\pm+ \ve_0^\pm \cdot \nabla \we_\eps^\pm -\nu^\pm \Delta \we_\eps^\pm &+ \nabla q_\eps^\pm = 0 &\quad & \text{in }\Omega_t^{\eps,\pm}, t\in (0,T_\eps),\\
      \Div \we_\eps^\pm &= 0 &\ & \text{in }\Omega_t^{\eps,\pm}, t\in (0,T_\eps),\\\label{eq:ApproxStokesMCF2}
      \llbracket\we_\eps\rrbracket =0,\quad   \llbracket\nu D\we_\eps-q_\eps\tn{I}\rrbracket\cdot \no_{\Gamma^\eps_t} &= \sigma \Delta^{\Gamma_\eps}h_{N-\frac34,\eps} \no_{\Gamma^\eps_t} &\ &\text{on }\Gamma_t^\eps, t\in (0,T_\eps),\\\label{eq:ApproxStokesMCF3}
      \we_\eps^-|_{\partial\Omega} &= 0 &&\text{on }\partial\Omega\times (0,T_\eps),\\\label{eq:ApproxStokesMCF4}
      \we_\eps^\pm|_{t=0} &=0&& \text{in } \Omega^\pm_0 
    \end{alignat}
    together with
      \begin{align}\label{eq:ApproxStokesMCF5}
        &\partial_t h_{N-\frac34,\eps}+\no_{\Gamma^\eps_t} \cdot \we_\eps|_{X_\eps(0,S_\eps,t)}-(\partial_tS_\eps+v_0\cdot\nabla S_\eps)|_{X_\eps(0,S_\eps,t)} \partial_sh_{N-\frac34,\eps}\\\nonumber
        &\qquad -\sqrt{\eps} |\nabla S_0|^2|_{X_\eps(0,S_\eps,t)}\partial_s^2h_{N-\frac34,\eps}
  +a_\eps(S_\eps,t) h_{N-\frac34,\eps} = - (P_M(\ue)\cdot\nabla d_\eps)|_{X_\eps(0,S_\eps,t)}
  \end{align}
  on $\T^1\times (0,T_\eps)$ and $h_{N-\frac34,\eps}|_{t=0}=0$, where $\we_\eps^\pm= \we_\eps|_{\Omega^{\eps,\pm}}$, $q_\eps^\pm= q|_{\Omega^{\eps,\pm}}$. Here $a_\eps$ is determined in the proof of Theorem~\ref{thm:hN12} below. This system can be considered as a linearization of \eqref{eq:Limit1}-\eqref{eq:Limit6} if $\sqrt{\eps} H_{\Gamma_t}$ was added to the right-hand side of \eqref{eq:Limit5} and $\Omega^\pm_t, \Gamma_t$ was replaced by $\Omega^{\eps,\pm}_t, \Gamma_t^\eps$. The function $h_{N-\frac14,\eps}$ will be determined in the proof of Theorem~\ref{thm:hN12} below.

  As in Section~\ref{sec:MatchedAsymptotics} we extend $\we^\pm_\eps$ and $q_\eps^\pm$ to $\Omega\times (0,T_\eps)$ such that $\Div \we_\eps^\pm =0$ in $\Omega\times (0,T_\eps)$ and $\we_\eps^\pm \in H^1(0,T_\eps;L^2(\Omega))\cap L^2 (0,T_\eps;H^2(\Omega))$, $q_\eps^\pm\in L^2(0,T_\eps;H^1(\Omega))$ in a bounded manner.
  Because of Theorem~\ref{thm:LinStokesMCF}, we have the uniform bounds
    \begin{align}\nonumber
      &\|h_{N-\frac34,\eps}\|_{L^\infty(0,T_\eps;H^1(\T^1))}  + \eps^{\frac12} \|h_{N-\frac34,\eps}\|_{H^1(0,T_\eps;H^{\frac12})\cap L^2(0,T_\eps;H^{\frac52})}\\\nonumber
      &+\eps^{\frac14} \|\partial_s^2 h_{N-\frac34,\eps}\|_{L^2((0,T_\eps)\times \T^1)}+\|\we_\eps\|_{H^1(0,T_\eps;V(\Omega)')}+ \|\we_\eps\|_{L^2(0,T_\eps;H^1(\Omega))}\\\label{eq:EstimLinStokesMCF}
     & +\eps^{\frac12}\|\we_\eps\|_{H^1(0,T_\eps;L^2(\Omega))}+ \eps^{\frac12} \|\we_\eps\|_{L^2(0,T_\eps;H^2(\Omega^{\eps,\pm}_t))}
    \qquad\leq C(M),
    \end{align}
    where $C(M)$ does not depend on $T_\eps$.
  The estimate \eqref{eq_estimate_wA} follows from the well-known embedding $L^2(0,T_\eps;V(\Om))\cap H^1(0,T_\eps;V(\Om)')\hookrightarrow C^0([0,T_\eps],L^2(\Om))$, where the embedding constant is uniform due to \eqref{eq:ApproxStokesMCF4}. 

For the following we denote $u_A^{in}:= c_A^{in}- \tilde{c}_A^{in}$ and use $\ve_A^{in}=\tilde{\ve}_A^{in}+\eps^{N+\frac14}\we_\eps$, where $\we_\eps = \we_\eps^+\chi_{\Omega^{\eps,+}}+\we_\eps^-\chi_{\Omega^{\eps,-}}$. Then we obtain in a straight forward manner
\begin{align}\nonumber
  &  \partial_t c_A^{in}+ \ve_A^{in}\cdot \nabla c_A^{in}-\sqrt{\eps}\Delta c_A^{in} + \eps^{-\frac32}f'(c_A^{in})\\\nonumber
  &= \partial_t u_A^{in}+ \ve_A^{in}\cdot \nabla u_A^{in}-\sqrt{\eps}\Delta u_A^{in} + \eps^{-\frac32}f''(\tilde{c}_A^{in}) u_A^{in}+\tilde{s}_A^\eps+ \eps^{N+\frac14}\we_\eps\cdot \nabla \tilde{c}_A^{in}+ O(\eps^{N+\frac12})\\\label{eq:cAIdentity}
    &= \partial_t u_A^{in}+ \ve_A^{in}\cdot \nabla u_A^{in}-\sqrt{\eps}\Delta u_A^{in} + \eps^{-\frac32}f''(\tilde{c}_A^{in}) u_A^{in}+ \eps^{N+\frac14}\we_\eps\cdot \nabla \tilde{c}_A^{in} + O(\eps^{N+\frac12})
\end{align}
in $L^2(0,T_\eps;L^2(\Gamma_t(2\delta))):=L^2(\Gamma(2\delta)\cap((0,T_\eps)\times\R^2))$. Here $\tilde{s}_A^\eps$ is a term that is quadratic in $u_A^{in}$ times $\eps^{-\frac32}$. Hence $\tilde{s}_A^\eps$ is $O(\eps^{2N-\frac32-\frac32+\frac12})=O(\eps^{N+\frac12})$ in $L^2(0,T_\eps;L^2(\Gamma_t(2\delta)))$ if $N\geq 3$ due to \eqref{eq:EstimLinStokesMCF}.

For the first terms we have:
\begin{thm}\label{thm:hN12}
	Let $u_A^{in}=\eps^{N-\frac34} \theta_0'(\rho) \bar{h}_{N,\eps}(S_\eps(x,t),t)$ and $h_{N-\frac34,\eps}$  be as before and define
	\[
	\mathcal{R}_\eps:=\partial_t u_A^{in}+ \ve_A^{in}\cdot \nabla u_A^{in}-\sqrt{\eps}\Delta u_A^{in} + \eps^{-\frac{3}{2}}f''(\tilde{c}_A^{in}) u_A^{in}+\eps^{N+\frac14}(P_M(\ue)+\we_\eps)|_{X_\eps(0,S_\eps,t)}\cdot \nabla \tilde{c}_A^{in},
	\]
	where $\ue=\ue(\eps)$ is uniformly bounded in $L^2(0,T_\eps,H^1(\O)^2)\cap H^{\frac{1}{2}}(0,T_\eps;L^2(\O)^2)$ for small $\eps$, cf.~Remark \ref{th_rem_U}. Then there is a choice of $h_{N-\frac{1}{4},\eps}\in X_{T_\eps,0}$ with bounds as in Theorem \ref{thm:ParabolicEqOnSurface} for $\kappa=\sqrt{\eps}$, $r=0$ such that
	\[
	\mathcal{R}_\eps=\eps^{N-\frac34} g_\eps(\rho,S_\eps,t) +O(\eps^{N+\frac{1}{2}})\quad\text{ in }L^2(0,T_\eps;L^2(\Gamma_t(2\delta))),
	\]
	where $g_\eps$ satisfies the conditions of Lemma \ref{lem:EstimMeanValueFree}.
\end{thm}
\begin{proof}
  First of all, let us recall where all the appearing terms come from. By \eqref{eq:innerexpan'}, it holds $c_A^{in}=\tilde{c}_A^{in}+u_A^{in}$ and $\ve_A^{in}=\tilde{\ve}_A^{in}+\eps^{N+\frac14}\hat{\we}_{\eps}$, where the terms
  \[
  \tilde{c}_A^{in}=\theta_0(\rho)+\eps^{3/2}c_{3/2}(\rho,x,t)+\eps^2 c_2+...
  \] 
  and $\tilde{\ve}_A^{in}=\hat\ve_0(\rho,x,t)+\eps^{1/2}\hat\ve_{1/2}+...$ stem from the inner expansion in Section \ref{subsec:The-Inner-Expansion} and are smooth.
  Here $u_A^{in}=\eps^{N-\frac34} \theta_0'(\rho) \bar{h}_{N,\eps}(S_\eps(x,t),t)$ is as in \eqref{eq:innerexpan'} with the expansion $\bar{h}_{N,\eps}=h_{N-\frac34,\eps}+\eps^{\frac12}h_{N-\frac14,\eps}$. Note that $h_{N-\frac34,\eps}$ will be determined by a coupled equation as in Theorem \ref{thm:LinStokesMCF} and $h_{N-\frac14,\eps}$ will be determined by Theorem \ref{thm:ParabolicEqOnSurface} for $\kappa_\eps=\sqrt{\eps}$. Thus these functions will satisfy the uniform estimates in there for $\kappa=\sqrt{\eps}$. Hence we can already assume the latter estimates to hold and disregard some unimportant higher order terms in the following. Finally, recall the properties and expansion form of $d_\eps,S_\eps$ from above. 
  
  We compute all terms with the chain rule and use Taylor for the $f''$-term. This yields
  \begin{align*}
  	\mathcal{R}_\eps&=\eps^{N-\frac34}\left(\theta_0''(\rho)\frac{\partial_td_\eps}{\eps} \bar{h}_{N,\eps}+\theta_0'(\rho)(\partial_t\bar{h}_{N,\eps}+\partial_tS_\eps \partial_s\bar{h}_{N,\eps})\right)\\
  	&+\eps^{N-\frac34}\tilde{\ve}_A^{in}\cdot\left(\theta_0''(\rho)\frac{\nabla d_\eps}{\eps}\bar{h}_{N,\eps}+ \theta_0'(\rho)\nabla S_\eps \partial_s\bar{h}_{N,\eps}\right)\\
                        &-\eps^{N-\frac14}\left(\theta_0'''(\rho)\frac{|\nabla d_\eps|^2}{\eps^2}\bar{h}_{N,\eps}+2\theta_0''(\rho)\frac{\nabla d_\eps\cdot\nabla S_\eps}{\eps}\partial_s\bar{h}_{N,\eps}
                          +\theta_0''(\rho)\frac{\Delta d_\eps}{\eps} \bar{h}_{N,\eps} + \theta_0'(\rho)\Delta_{\Gamma_\eps}\bar{h}_{N,\eps}\right)\\
  	&+ \eps^{-\frac32}\eps^{N-\frac34}\theta_0'(\rho) \bar{h}_{N,\eps} \left(f''(\theta_0)+f'''(\theta_0)\left[\eps^\frac32 c_\frac32+\eps^2 c_2\right]  \right)\\
  	&+\eps^{N+\frac14}(P_M(\ue)+\we_\eps)|_{X_\eps(0,S_\eps,t)}\cdot \frac{\nabla d_\eps}{\eps} \theta_0'(\rho) +O(\eps^{N+\frac12})\quad\text{ in }L^2(\Gamma(2\delta)\cap((0,T_\eps)\times\R^2)),
  \end{align*}
  where we used that terms of the form $\eps^N a(\rho)b_\eps(S_\eps,t)$ with $a\in \mathcal{R}_{0,\alpha}$ and $b_\eps\in L^2(\T^1\times(0,T_\eps))$ uniformly bounded with respect to small $\eps$ are $O(\eps^{N+\frac12})$ in $L^2(\Gamma(2\delta))$ with Lemma~\ref{lem:rescale2}. 
  Moreover, we used that $(P_M(\ue)+\we_\eps)|_{X_\eps(0,.)}$ is bounded in $L^2((0,T_\eps)\times\T^1)$, the estimates for $h_{N-\frac34,\eps}$ and $h_{N-\frac14,\eps}$ as well as $N\geq2$ to replace $\ve_A^{in}$ by $\tilde{\ve}_A^{in}$ up to the error in $L^2(\Gamma(2\delta))$ above. 
  
  Note that the remaining $\ue$-term is critical and sits at order $O(\eps^{N-\frac34})$. Actually, the point of the theorem is to generate this term. Therefore the prefactor of $\bar{h}_{N,\eps}$ was chosen to be $\eps^{N-\frac34}$ such that the contribution is also in that order. There are several problems one has to overcome. First, there are contributions of $d_\eps$-terms into orders below the critical one. Since $|\nabla d_\eps|^2=1+O(\eps^2)$, we can cancel the lowest order contribution of this term with the $f''(\theta_0)$-term. Moreover, $\nabla d_\eps\cdot \nabla S_\eps$ is of order $O(\eps^2)$. By taking a close look at Section \ref{subsec:The-Inner-Expansion}, one can infer that $\partial_td_\eps+\tilde{\ve}_A^{in}\cdot\nabla d_\eps-\sqrt{\eps}\Delta d_\eps$ only gives a contribution of order $O(\eps)$ because there is some cancellation. These terms are of sum structure with $\sqrt{\eps}$-spacing. Furthermore, one has to expand some remaining terms in $\mathcal{R}_\eps$ depending on $(x,t)$ or $(\rho,x,t)$ into $(\rho,S_\eps,t)$. This can be done by transforming the $(x,t)$-part with $X_\eps$, using Taylor expansion in the first variable and $d_\eps=\eps\rho$. Therefore we use $\hat\ve_0(\rho,x,t)=\ve_0^-(x,t)+(\ve_0^- -\ve_0^+)(x,t)\eta(\rho)$. Because of $\ve_0^-=\ve_0^+$ on $\Gamma$, the second part is improved by the order $\sqrt{\eps}$ due to a Taylor expansion. Finally, the functions in the expansion of $\bar{h}_{N,\eps}$ should be obtained by solving equations of the form mentioned above, of course with the goal to just leave remainders as stated in the theorem. The goal to have a remainder $r_\eps$ suitable for Lemma \ref{lem:EstimMeanValueFree} as stated in the theorem leads to the desired equation for $h_{N-\frac34,\eps}$ in order to resolve the order $\eps^{N-\frac34}$. The remaining terms contribute formally to the order $\eps^{N-\frac14}$. For this order we intend to use $h_{N-\frac14,\eps}$. However, one has to take care since not all terms in \eqref{eq:h1} for $h_{N-\frac14,\eps}$ scale as the right hand side with respect to $\kappa=\sqrt{\eps}$ in the $L^2$-norm. More precisely, the first two terms with $\partial_t$ and $\partial_s$ are scaling worse on their own (but only at the amount of $\eps^{\frac14}$), the others are fine. Hence in the application here we need the same prefactor (depending on $\rho$) for those two terms in the equations we require. By having a look at $\mathcal{R}_\eps$, we see that $\theta_0'$ is the desired prefactor. Hence all terms with derivatives of $h_{N-\frac14,\eps}$ either have the same $\rho$-prefactor $\theta_0'$ or contribute to the order $O(\eps^N)$ or higher. Thus we obtain an equation of the form as in Theorem \ref{thm:ParabolicEqOnSurface} for $h_{N-\frac14,\eps}$. More precisely we have
  \begin{align*}
  \partial_t h_{N-\frac14,\eps}-(\partial_tS_\eps+\ve_0\cdot\nabla S_\eps)|_{X_\eps(0,S_\eps,t)} \partial_sh_{N-\frac14,\eps} &-\sqrt{\eps} \Delta_{\Gamma_\eps}h_{N-\frac14,\eps}+\tilde{a}_\eps(S_\eps,t) h_{N-\frac14,\eps} = \tilde{g}_\eps,
  \end{align*}
  where $\tilde{a}_\eps$ is smooth (uniformly bounded with respect to small $\eps$, derivatives as well) and $\tilde{g}_\eps$ is uniformly bounded in $L^2(\Gamma(2\delta)\cap((0,T_\eps)\times\R^2))$. By rewriting the $\partial_t h_{N-\frac14,\eps}$ and $\partial_sh_{N-\frac14,\eps}$-terms with the above equation, and estimating the $O(\eps^N)$ remainders with the aid of Lemma~\ref{lem:rescale2}, 
  we finally get remainders as stated in the theorem. Altogether, this yields the claim.
\end{proof}

  \medskip
    
  \noindent
  \begin{proof*}{of Theorem~\ref{thm:approx}}
    First, we estimate $\overline{a}_\eps\colon (0,T_0)\to \R$. To this end we define
    \begin{equation*}
      \tilde{\ve}_A(x,t)=\zg \tilde{\ve}_A^{in}(x,t)+(1-\zg )\left(\tilde{\ve}_A^{+}(x,t)
        \chi_+ +\tilde{\ve}_A^-(x,t)\chi_-\right).
    \end{equation*}
    Then
    \begin{align*}
      \Div \tilde{\ve}_A &= \zg G_\eps + O(e^{-\frac{\alpha\delta}{2\eps}}),\quad
      \partial_t \Div \tilde{\ve}_A = \zg (\partial_t G_\eps) + O(e^{-\frac{\alpha\delta}{2\eps}})
    \end{align*}
    because of the matching conditions and $\Div \tilde{\ve}_A^\pm =0$. Since $\overline{a}_\eps = \no_{\partial\Omega} \cdot \tilde{\ve}_A|_{\partial\Omega}$ only depends on $t$,
    \begin{align*}
      \overline{a}_\eps = \frac1{\HS^{d-1}(\partial\Omega)}\int_{\Omega} \Div \tilde{\ve}_A \, dx =O(\eps^{N+2})\quad \text{and}\\
      \partial_t \overline{a}_\eps = \frac1{\HS^{d-1}(\partial\Omega)}\int_{\Omega} \partial_t\Div \tilde{\ve}_A \, dx =O(\eps^{N+1})
    \end{align*}
    in $L^\infty(0,T_0)$ due to \eqref{eq:RemainderApprox1}-\eqref{eq:RemainderApprox2}. 
    The rest of the proof is split into three parts.

    \medskip
    
    \noindent{\bf Part 1: Error in the divergence equation:} Because of $\llbracket\we_\eps\rrbracket=0$ and $\Div \we_\eps^\pm =0$ in $\Omega^{\eps,\pm}$, we have
    \begin{align}\nonumber
      &\Div (\ve_A(x,t) + \mathbf{N}\bar{a}_\eps(t)) \\\nonumber
      &= \zeta(d_\Gamma(x,t)) G_\eps  + \zeta'(d_\Gamma(x,t))\nabla d_\Gamma(x,t)\cdot \left(\ve_A^{in}(x,t) - \ve^+_A(x,t)\chi_{\Omega^+(t)}(x) - \ve^-_A(x,t)\chi_{\Omega^-(t)}(x)\right)\\\label{eq:DivError}
      &= O(\eps^{N+1}) 
    \end{align}
    in $H^1(0,T_\eps;L^2(\Omega))$ 
    because of \eqref{eq:RemainderApprox1}-\eqref{eq:RemainderApprox2} and the matching condition in Theorem~\ref{thm:Approx1}.
 Together with the previous estimates for $\overline{a}_\eps$ this shows \eqref{eq:ApproxS2} for some (different) $G_\eps\colon \Omega \times (0,T_\eps)\to \R$, which is given by the sum of the right-hand side of \eqref{eq:DivError} and $-\Div (\mathbf{N}\bar{a}_\eps(t))$.

     \medskip

     \noindent{\bf Part 2: Error in the linear momentum equation:}
    First of all,
    \begin{align*}
      &\partial_t \ve_A +\ve_A\cdot \nabla \ve_A -\Div(2\nu(c_A)D\ve_A)+ \nabla p_A\\
      &= \zeta(d_\Gamma) \left(\partial_t \ve_A^{in} +\ve_A^{in}\cdot \nabla \ve_A^{in} -\Div(2\nu(c_A^{in})D\ve_A^{in})+ \nabla p_A^{in}\right)\\
      &\quad        + (1-\zeta(d_\Gamma))\sum_{\pm} \left(\partial_t \ve_A^{\pm} +\ve_A^\pm\cdot \nabla \ve_A^{\pm} -\Div(2\nu(c_A^\pm)D\ve_A^{\pm})+ \nabla p_A^{\pm}\right)\chi_\pm+O(\eps^{N+1}) 
    \end{align*}
    in $L^2(\Omega\times (0,T_\eps))^2$ 
    because of the matching conditions and $\partial_t \overline{a}_\eps(t)\mathbf{N}=O(\eps^{N+1})$ in $L^\infty(0,T_0)$. Since by the construction
  \begin{equation*}
    \partial_t \ve_A^{\pm} +\ve_A^\pm\cdot \nabla \ve_A^{\pm} -\Div(2\nu(c_A^\pm)D\ve_A^{\pm})+ \nabla p_A^{\pm} = O(\eps^{N+\frac54}) \quad \text{in }L^2(\Omega^\pm\setminus \Gamma(2\delta)),
  \end{equation*}
  we only have to consider the terms from the inner expansion.
  
Next by the construction of $\hat{\we}_\eps$ and \eqref{eq:ApproxStokesMCF2} we have 
   \begin{align*}
     &-\Div (2\nu(\theta_0(\rho))D\we_\eps) + \nabla q_\eps = -\Div (2\nu(\theta_0(\rho))(\we^+_\eps\chi_{\Omega^{\eps,+}}+\we^-_\eps\chi_{\Omega^{\eps,-}}))+ \nabla q_\eps\\
                                                 &\quad = 2\Div (\nu^+\we^+_\eps\chi_{\Omega^{\eps,+}}+\nu^-\we^-_\eps\chi_{\Omega^{\eps,-}})+ \nabla q_\eps\\
     &\qquad + 2\Div ((\nu(\theta_0)-\nu^+)\chi_{\Omega^{\eps,+}} D\we_\eps^+) + 2\Div ((\nu(\theta_0)-\nu^-)\chi_{\Omega^{\eps,-}} D\we_\eps^-) \\
     &\quad = (-\nu^+\Delta \we_\eps^++\nabla q_\eps^+)\chi_{\Omega^{\eps,+}}+(-\nu^-\Delta \we_\eps^-+\nabla q_\eps^-)\chi_{\Omega^{\eps,-}}
       -\sigma \delta_{\Gamma^\eps_t}\otimes \Delta_{\Gamma_t^\eps} h_\eps\circ S_\eps \no_\eps + O(\eps^{\frac14})
   \end{align*}
   in $L^2(0,T_\eps;(H^1(\Gamma_t(3\delta))^2)')$, where
   \begin{equation*}
     \weight{\delta_{\Gamma^\eps_t}\otimes \Delta_{\Gamma_t^\eps} h_\eps\circ S_\eps \no_\eps,\boldsymbol\varphi} := \int_{\Gamma_t^\eps} \Delta_{\Gamma_t^\eps} h_\eps(S_\eps(x,t))\no_\eps \cdot \boldsymbol\varphi (x)\, d\sigma(x)\quad \text{for all }\boldsymbol\varphi\in H^1(\Omega)^2.
   \end{equation*}
   Here we have used for the second equality that
   \begin{align*}
     \|(\nu(\theta_0)-\nu^\pm) D\we_\eps^\pm\|_{L^2(\Omega^{\eps,\pm})}&\leq \|\nu(\theta_0)-\nu^\pm\|_{L^2(\R_\pm)} \|\sup_{r\in [0,2\delta]}D\we_\eps^\pm(X_\eps(r,.)\|_{L^2((0,T)\times \T^1)}\\
     &\leq C\sqrt{\eps} \|\we_\eps\|_{L^2(0,T;H^1(\Omega^{\eps,\pm}_t))}^{\frac12}\|\we_\eps\|_{L^2(0,T;H^2(\Omega^{\eps,\pm}_t))}^{\frac12} = O(\eps^{\frac14})
   \end{align*}
   due to \eqref{eq:EstimLinStokesMCF1}-\eqref{eq:EstimLinStokesMCF2} as well as $\no_\eps = \no_{\Gamma_t^\eps}+ O(\eps^{N+3})$ due to \eqref{eq:ErrordA}.
   Hence we conclude 
      \begin{align*}
        &\partial_t \we_\eps+\we_\eps\cdot \nabla \ve_0 +\ve_0 \cdot \nabla \we_\eps -\Div (2\nu(\theta_0(\rho))D\we_\eps)+\nabla q_\eps\\
        &=  -\sigma \delta_{\Gamma^\eps_t}\otimes \Delta_{\Gamma_t^\eps} h_\eps\circ S_\eps \no_\eps\\
        &\quad +\sum_{\pm} \chi_{\Omega^{\eps,\pm}} \left(\partial_t \we_\eps^\pm+\we_\eps^\pm \cdot \nabla \ve_0^\pm+ \ve_0^\pm \cdot \nabla \we_\eps^\pm -\nu^\pm \Delta \we_\eps^\pm + \nabla q_\eps^\pm \right)  +O(\eps^{\frac14})  \\
        &=  -\sigma \delta_{\Gamma^\eps_t}\otimes \Delta_{\Gamma_t^\eps} h_\eps\circ S_\eps \no_\eps+O(\eps^{\frac14})
      \end{align*}
      in $L^2(0,T_\eps; (H^1(3\delta))')$.
            Moreover, we have because of Lemma~\ref{lem:Taylor} for $\boldsymbol{\varphi}\in H^1(\Gamma_t(3\delta))^2$
      \begin{alignat*}{1}
        &\left|\sigma\weight{\delta_{\Gamma_t^\eps}\otimes(\Delta_{\Gamma_t^\eps} h_{N-\frac34,\eps})\circ S_\eps \no_\eps, \boldsymbol{\varphi}}- \int_{\Gamma_t^\eps(\frac{11}4\delta)} \frac{\theta'_0(\rho)^2}\eps(\Delta_{\Gamma_t^\eps} h_{N-\frac34,\eps})\circ S_\eps\no_\eps\cdot \boldsymbol\varphi\, dx \right|\\
        & \leq \frac1\eps\int_{-\frac{11}4\delta}^{\frac{11}4\delta}\int_{\T^1}| \theta_0'(\tfrac{r}\eps)^2\no_\eps\cdot \left(\boldsymbol{\varphi}(X_\eps(r,s,t))- \boldsymbol{\varphi}(X_\eps(0,s,t)) \right)  \Delta_{\Gamma_t^\eps} h_{N-\frac34,\eps}(s,t)|J_\eps(r,s,t)\, ds\, dr\\
        &\qquad + C\eps\|h_{N-\frac34,\eps}\|_{H^2(\T^1)}\|\boldsymbol{\varphi}\|_{H^1(\Gamma_t(3\delta))}\\
        & \leq C\left(\int_{-\frac{11}4\delta}^{\frac{11}4\delta}|\tfrac{r}\eps \theta_0'(\tfrac{r}\eps)^2|^2\, dr\right)^{\frac12}\|\boldsymbol{\varphi}\|_{H^1(\Gamma_t(3\delta))}\|\Delta_{\Gamma^\eps_t} h_{N-\frac34,\eps}\|_{L^2(\T^1)} + C(M)\eps\|h_{N-\frac34,\eps}\|_{H^2(\T^1)}\|\boldsymbol{\varphi}\|_{H^1(\Gamma_t(3\delta))}\\
       &\leq C\sqrt{\eps}\|\boldsymbol{\varphi}\|_{H^1(\Gamma_t(3\delta))}\| h_{N-\frac34,\eps}\|_{H^2(\T^1)} + C(M)\eps\|\boldsymbol{\varphi}\|_{H^1(\Gamma_t(3\delta))}\| h_{N-\frac34,\eps}\|_{H^2(\T^1)}
      \end{alignat*}
      since
      $$
      \frac1\eps\int_{-\frac{11}4\delta}^{\frac{11}4\delta}\theta_0'(\tfrac{r}\eps)^2J_\eps(r,s,t)\, dr = \frac{\sigma}{|\nabla S_\eps (X_\eps(0,s,t))|} + O(\eps)
      $$
      due to \eqref{J_eps}, $|\nabla d_\eps|^2= 1 + O(\eps^2)$, $\nabla d_\eps\cdot \nabla S_\eps=O(\eps^2)$, $|\nabla S_\eps(X_\eps(0,s,t)|= |\partial_s X_\eps(0,s,t)|^{-1}+O(\eps^2)$,
      and a Taylor expansion around $r=0$.
      Here $\|\partial_s^2 h_{N-\frac34}\|_{L^2(\T^1\times(0,T_\eps))}=O(\eps^{-\frac14})$ due to \eqref{eq:EstimLinStokesMCF1}.
      Hence we obtain
    \begin{align*}
      &\partial_t \ve_A + \ve_A\cdot \nabla \ve_A -\Div (2\nu(c_A) D\ve_A) + \nabla p_A\\
      &= -\zeta(d_\Gamma)\left(  \eps\Div (\nabla \tilde{c}_A^{in}\otimes \nabla \tilde{c}_A^{in}) +\eps^{N-\frac34}(\theta_0'(\rho))^2(\Delta_{\Gamma_\eps} h_{N-\frac34})\circ S_\eps\right)+ O(\eps^{N+\frac12})
    \end{align*}
    in $L^2(0,T_\eps; (H^1(\Omega)^2)')$
    by using \eqref{eq:RemainderApprox1} and the matching condition in Theorem~\ref{thm:Approx1}.

    Now we use that
    \begin{align*}
      &\eps \nabla c_A^{in}\otimes \nabla c_A^{in} -\eps \nabla \tilde{c}_A^{in}\otimes \nabla \tilde{c}_A^{in} \\
      &= \eps^{N-\frac74}2\theta_0''(\rho)\theta_0'(\rho)\no_\eps\otimes \no_\eps \ol{h}_{N,\eps}(S_\eps(x,t),t) \\
           &\quad + \eps^{N-\frac34}\theta_0'(\rho)^2\left(\nabla^{\Gamma^\eps} \ol{h}_{N,\eps}(S_\eps(x,t),t)\otimes \no_\eps+ \no_\eps\otimes\nabla^{\Gamma^\eps} \ol{h}_{N,\eps}(S_\eps(x,t),t)\right)  \\
      &\quad + \eps^{N-\frac14} \mathbf{r}_\eps(\rho,x,t)\cdot \mathfrak{a}_\eps(s,t)\qquad \text{in } \Gamma(\tfrac52\delta)
    \end{align*}
    for some $\mathbf{r}_\eps \in (\mathcal{R}_{0,\alpha})^N$, $\mathfrak{a}_\eps \in L^\infty(0,T_\eps;L^2(\T^1))^N$ and $N\in\N$ with uniformly bounded norms in $\eps\in (0,\eps_0)$. Here one uses that $\sqrt{\eps}h_{N-\frac14,\eps}, \sqrt{\eps}h_{N-\frac34,\eps} \in L^\infty(0,T_\eps; W^1_4(\T^1))$ are bounded (because of \eqref{eq:EstimLinStokesMCF2} and \eqref{eq:hEstim2}) and that $\partial_s h_{N-\frac14,\eps}, h_{N-\frac34,\eps}$ enter at most quadratically.
    Hence we obtain
     \begin{align*}
       &  -\eps \Div (\nabla c_A^{in}\otimes \nabla c_A^{in}) +\eps \Div (\nabla \tilde{c}_A^{in}\otimes \nabla \tilde{c}_A^{in})\\
       &\quad =  -\eps^{N-\frac{11}4} \partial_\rho (2\theta_0''(\rho)\theta_0'(\rho))(\no_\eps \cdot \no_\eps) \no_\eps \ol{h}_{N,\eps} -\eps^{N-\frac{7}4} 2\theta_0''(\rho)\theta_0'(\rho)\no_\eps\cdot \nabla^{\Gamma_\eps} \ol{h}_{N,\eps} \\
       &\qquad -\eps^{N-\frac74} 2\theta_0''(\rho)\theta_0'(\rho)\Div (\no_\eps\otimes\no_\eps) \ol{h}_{N,\eps} -\eps^{N-\frac{7}4} 2\theta_0''(\rho)\theta_0'(\rho)\left(\no_\eps\cdot \no_\eps \nabla^{\Gamma_\eps} \ol{h}_{N,\eps}+ \no_\eps \cdot \nabla^\Gamma \ol{h}_{N,\eps}\no_\eps \right)\\
       &\qquad - \eps^{N-\frac34}\theta_0'(\rho)^2\Div \left(\nabla^{\Gamma_\eps} \ol{h}_{N,\eps}(S_\eps(x,t),t)\otimes \no_\eps+ \no_\eps\otimes\nabla^{\Gamma_\eps} \ol{h}_{N,\eps}(S_\eps(x,t),t)\right)\\
       &\qquad + O(\eps^{N+\frac14}(\sqrt{T}+\eps^{\frac14}))\\
       &\quad = \nabla \pi_\eps-\eps^{N-\frac74} 2\theta_0''(\rho)\theta_0'(\rho)\Div (\no_\eps)\no_\eps \ol{h}_{N,\eps} \\
       &\qquad - \eps^{N-\frac34}\theta_0'(\rho)^2\Div \left(\nabla^{\Gamma_\eps} \ol{h}_{N,\eps}(S_\eps(x,t),t)\otimes \no_\eps+ \no_\eps\otimes\nabla^{\Gamma_\eps} \ol{h}_{N,\eps}(S_\eps(x,t),t)\right)\\
       &\qquad + O(\eps^{N+\frac14}(\sqrt{T}+\eps^{\frac14}))\\
       &\quad = \nabla \pi_\eps - \eps^{N-\frac34}\theta_0'(\rho)^2\Delta^{\Gamma_\eps} h_{N-\frac34,\eps}(S_\eps(x,t),t)\no_\eps+ O(\eps^{N+\frac14}(\sqrt{T}+\eps^{\frac14}))
     \end{align*}
     in $L^2(0,T; (H^1(\Gamma_t(\tfrac52\delta))'))$ for all $T\in (0,T_\eps]$ since $\no_\eps \cdot \no_\eps = 1+O(\eps^2)$, $\no_\eps\cdot \nabla \no_\eps =O(\eps^2)$, $\no_\eps \cdot \nabla S_\eps = O(\eps^2)$, and $\ol{h}_{N-\frac34,\eps}\in L^\infty(0,T_\eps, H^1(\T^1))$ and $\eps^{\frac38}\ol{h}_{N-\frac34,\eps}\in L^4(0,T_\eps, H^2(\T^1))$ are uniformly bounded, where $\pi_\eps =\eps^{N-\frac74} 2\theta''_0(\rho)\theta_0 (\rho) \ol{h}_{N,\eps}$.
     Now replacing $p_A$ by $p_A+\pi_\eps$ we obtain \eqref{eq:ApproxS1}.
     
\medskip

\noindent{\bf Part 3: Error in the Allen-Cahn equation:} Since $c_A^\pm \equiv \pm 1$, the equation \eqref{eq:ApproxS3} together with \eqref{eq_estimate_wA} and \eqref{estimation:Seps}, follows in a straight forward manner from \eqref{eq:cAIdentity}, Theorem~\ref{thm:hN12} and the matching conditions.
\end{proof*}

\section{Sharp Interface Limit}\label{sec:main_proof}

The proof of our main result Theorem~\ref{thm:main} follows the same steps as in \cite[Section~4]{AbelsFei}. But there are several careful adaptions needed since for our choice of mobility certain estimates ``degenerate''/give worse estimates compared to \cite{AbelsFei} and the construction of the approximate solution is different.

\subsection{The Leading Error in the Velocity}\label{subsec:LeadingErrorVelocity}

For the following let $(c_A, \ve_A, p_A)$ and $(\tilde{c}_A,\tilde{\ve}_A, \tilde{p}_A)$ are given as in Section~\ref{sec:ApproxSolutions}, where $(c_A, \ve_A, p_A)$ still depends on the choice of $\mathbf{u}$, which will be chosen in the following, but $(\tilde{c}_A,\tilde{\ve}_A, \tilde{p}_A)$ are independent of $\ue$.
Moreover, we define $\tilde{\we}:= \ve_\eps-\ve_A$.
Hence we obtain
\begin{equation}\label{eq:w}
\begin{split}
  \partial_t \tilde{\we}{+ \ve_\eps \cdot \nabla \tilde{\we}}-\Div(2\nu(c_\eps) D\tilde{\we})+\nabla q&={- \tilde{\we}\cdot \nabla \ve_A}+\Div(2(\nu(c_\eps)-\nu(c_A)) D\ve_A)\\
  &\quad-\eps\Div (\nabla u\otimes^s \nabla c_A)-\eps\Div (\nabla u\otimes  \nabla u)- R_\eps,\\
  \Div \tilde{\we}&= -G_\eps,\\
  \tilde{\we}|_{t=0}&=\ve_{0,\eps} - \ve_{A}|_{t=0},\\
    \tilde{\we}|_{\partial\Omega}&= 0,
\end{split}
\end{equation}
for some $q\colon \Om\times [0,T_0]\to \R$. Here $u=c_\eps-c_A$, $a\otimes^s b=a\otimes b+b\otimes a$ and $R_\eps$, $G_\eps$ are as in Theorem~\ref{thm:approx}.

In the following we consider the estimates
\begin{subequations}\label{assumptions}
  \begin{align}\label{assumptions1}
    \sup_{0\leq t\leq \tau} \|c_\eps(t) -c_A(t)\|_{L^2(\Omega)} + {\eps^{\frac{1}{4}}}\|\nabla (c_\eps -c_A)\|_{L^2(\Omega\times(0,\tau)\setminus \Gamma^\eps(\frac{3\delta}2))} &\leq R\eps^{\order+\frac12},\\\label{assumptions2}
    {\eps^{\frac{1}{4}}}\|\nabla_{{\btau_\eps}}(c_\eps -c_A)\|_{L^2(\Omega\times(0,T_\eps)\cap \Gamma^\eps(\frac{3\delta}{2}))} +\eps\|\partial_{\no_\eps}(c_\eps -c_A)\|_{L^2(\Omega\times(0,\tau)\cap \Gamma^\eps(\frac{3\delta}{2}))} &\leq R\eps^{\order+\frac12},\\\label{assumptions3}
    \|\nabla(c_\eps -c_A)\|_{L^\infty(0,\tau;L^2(\Omega))}+\eps^{\frac{1}{4}}\|\Delta(c_\eps -c_A)\|_{L^2(\Omega\times(0,\tau))} &\leq R\eps^{\order-\frac32},\\\label{assumptions4}
        \int_0^\tau\int_{\Gamma^\eps_t(\frac{3\delta}{2})}|\nabla u|^2 +\eps^{-2} f''(c_A(x,t)) u^2\, \sd x\, \sd t &\leq R^2\eps^{2N+\frac12}
\end{align}
\end{subequations}
for some $\tau=\tau(\eps) \in (0,T_0]$, $\eps_0\in (0,1]$, and all $\eps \in (0,\eps_0]$, where $R>0$ is chosen such that
 \begin{equation}\label{initial assumption-0}
  \|c_{0,\eps}-c_A|_{t=0}\|_{L^2(\Omega)}^2+ \varepsilon^4\|\nabla(c_{0,\eps}-c_A|_{t=0})\|_{L^2(\Omega)}^2+\|\ve_{0,\eps}-\ve_A|_{t=0}\|_{L^2(\Omega)}^2\leq \frac{R^2}{4}\eps^{2\order+1}e^{-C_LT_0}
\end{equation}
for all $\eps\in (0,1]$, where $C_L>0$ is the constant from the spectral estimate in Theorem~\ref{thm:Spectral}. We note that compared to \cite[Estimates (4.5)]{AbelsFei} there is an additional factor $\eps^{\frac14}$ in front of the norms for $\nabla (c_\eps-c_A)$ in \eqref{assumptions1}, for $\nabla_{\btau_\eps} (c_\eps-c_A)$ in \eqref{assumptions2}, and for $\Delta (c_\eps-c_A)$ in \eqref{assumptions3} as well as a loss by $\sqrt{\eps}$ in \eqref{assumptions4}.

As in \cite[Section~4]{AbelsFei}, we define
\begin{align}
T_{\varepsilon}:=\sup\{\tau\in[0,T_0]: \eqref{assumptions}  \ \text{holds true}\}.\label{def:Teps}
\end{align}
and have $T_{\varepsilon}>0$ because of \eqref{initial assumption-0}.

The main goal of this subsection is to obtain the following bound for the error $\tilde{\we}$ in the velocity, which again is by a factor $\eps^{\frac14}$ worse than the corresponding result in \cite{AbelsFei}:
\begin{theorem}\label{thm:ErrorVelocity}
Let $M>0$, $c_A, \tilde{\we}$ be as in \eqref{eq:w} and $u$ satisfy \eqref{assumptions} for some $R>0$ and $\tau = T_\eps \in (0,T_0]$  and $N\geq 3$. Then there are some $C(R,M)>0$, $C_0(R)$ independent of $\eps \in (0,\eps_1)$, where $\eps_1$ is as in Theorem~\ref{th_coord}, and $T\in (0,T_\varepsilon]$ and $\tilde{\we}=\we_1-\we_0$, where $\we_1\in C([0,T];L^2_\sigma(\Omega))\cap L^2(0,T;H^1_0(\Omega)^2)\cap H^{\frac12}(0,T;L^2_\sigma(\Omega))$, $\we_0\in L^2(0,T;H^2(\Omega)^2)\cap H^1(0,T;L^2(\Omega)^2)$ satisfy
  \begin{align}
    &\|\partial_t \we_1\|_{L^2(0,T;V(\Omega)')}+\|\we_1\|_{H^{\frac{1}{2}}(0,T;L^2(\Omega))}+\|\we_1\|_{L^\infty(0,T;L^2(\Omega))}+\|\we_1\|_{L^2(0,T;H^1(\Omega))}\nonumber\\
    &\qquad\leq C_0(R)\eps^{\order+\frac14}+ C(R,M) \eps^{\order+\frac{1}{4}}(T^{\frac12}+ \eps^{\frac14})\label{w1estimate},\\
    \label{est:w0}
    &\|\we_0\|_{L^2(0,T;H^2(\Omega))}+ \|\partial_t\we_0\|_{L^2(\Omega\times (0,T)))}\leq C(R,M) \eps^{N+1}
  \end{align}
  provided that $\|\ve_{0,\eps}-\ve_A|_{t=0}\|_{H^1(\Omega)}\leq C \eps^{N+\frac12}$ for some $C>0$.
\end{theorem}
\begin{rem}
  Now we choose
\begin{equation}
  \ue = \frac{\we_1}{\eps^{N+\frac14}} \in L^2(0,T_0; H^1_0(\Omega)^2)\cap C([0,T_0];L^2_\sigma(\Omega)).\label{def:ue}
\end{equation}
As in \cite{AbelsFei} this yields a non-linear evolution equation with a globally Lipschitz nonlinearity for $\ue$, which can be solved in the same manner as in \cite[Proof of Lemma~4.2]{StokesAllenCahn}.

  The result shows that, if we choose $M=2C_0(R)$, then there are $T'\in (0,T_0]$ and $\eps_0\in (0,1)$ (depending on $R>0$) such that
  \begin{equation*}
     \|\ue\|_{H^{\frac{1}{2}}(0,T;L^2(\Omega))}+\|\ue\|_{L^\infty(0,T;L^2(\Omega))}+\|\ue\|_{L^2(0,T;H^1(\Omega))}\leq M=2C_0(R)
   \end{equation*}
   provided that $T\in (0,\min(T', T_\eps)]$ and $\eps\in (0,\eps_0]$. In particular $P_M(\ue)=\ue$ in \eqref{eq:ApproxS3}. After the proof of Theorem~\ref{thm:ErrorVelocity} $M$ will be choosen as $M=2C_0(R)$. 
\end{rem}
\begin{proof*}{of Theorem~\ref{thm:ErrorVelocity}} The proof is a variant of the proof of \cite[Theorem~4.1]{AbelsFei}. But there are several careful modifications necessary because of the different powers in the estimates \eqref{assumptions} in the present case and the new $\eps$-dependent coordinates $(d_\eps, S_\eps)$, which are only approximatively orthogonal.

As in \cite{AbelsFei} let $(\we_0,q_0)$ solve the system
\begin{subequations}\label{neweq:w0}
\begin{alignat}{2}
\partial_t \we_0-\Delta\we_0+\nabla q_0&=0 &\qquad&\text{in }\Omega\times (0,T),\\
  \Div \we_0&=G_\eps&\qquad&\text{in }\Omega\times (0,T),\\
  \we_0|_{\partial\Omega}&= 0&\qquad&\text{on }\partial\Omega\times (0,T),\\
  \we_0|_{t=0} &=(\ve_A-\ve_\eps)|_{t=0}&\qquad& \text{in }\Omega,
\end{alignat}
\end{subequations}
where we note that $G_\eps|_{t=0}= \Div (\ve_A-\ve_\eps)|_{t=0}$.
By standard results on strong solutions of the Stokes system one obtains a unique solution $\we_0\in L^2(0,T;H^2(\Omega)^2)\cap H^1(0,T;L^2(\Omega)^2)$, which satisfies \eqref{est:w0}.
Then $\we_1:=\tilde{\we}+\we_0$ is a solution of the modified system
\begin{align}\nonumber
  \partial_t \we_1{+ \ve_\eps \cdot \nabla \we_1}-\Div(2\nu(c_\eps) D\we_1)+\nabla q&=-\eps\Div (\nabla u\otimes^s \nabla c_A)-\eps\Div (\nabla u\otimes  \nabla u)\\\nonumber&\quad+\Div(2(\nu(c_\eps)-\nu(c_A)) D\ve_A)
  +\we_0\cdot \nabla \ve_A\\\label{neweq:w}
  &\quad +\partial_t\we_0+\ve_\eps \cdot \nabla \we_0-\Div(2\nu(c_\eps) D\we_0)- R_\eps,\\\nonumber
  \Div \we_1&= 0,\\\nonumber
    \we_1|_{\partial\Omega}&= 0,\\\nonumber
  \we_1|_{t=0}&=0  
\end{align}
in a weak sense.
Now testing \eqref{neweq:w} with $\we_1$ and using Gronwall's inequality yields
\begin{align}
  &\sup_{0\leq t\leq T}\frac{1}{2}\int_{\Omega} |\we_1|^2(x,t)\sd x  +2\int_{0}^T\int_{\Omega} \nu(c_\eps) |D\we_1|^2\sd x \sd t\nonumber\\&\leq e^{CT}\bigg( \frac{1}{2}\int_{\Omega} |\we_1|^2(x,0)\sd x+\eps\int_{0}^T\bigg|\int_{\Omega}\big(\nabla u\otimes^s  \nabla c_A:\nabla\we_1\big)\sd x\bigg|\sd t\nonumber\\&\quad+2\int_{0}^T\bigg|\int_{\Omega}\big(\big(\nu(c_\eps)-\nu(c_A)\big) D\ve_A:\nabla\we_1\big)\sd x\bigg|\sd t
  \nonumber\\
  &\quad +\varepsilon\int_{0}^T\bigg|\int_{\Omega}\big(\nabla u\otimes\nabla u:\nabla\we_1\big)\sd x\bigg|\sd t
    +\int_{0}^T\bigg|\int_{\Omega}R_\eps\cdot\we_1\sd x\bigg|\sd t
  \nonumber\\&\quad+\int_{0}^T\bigg|\int_{\Omega}(\we_1\cdot \nabla \ve_A)\cdot \we_1\sd x\bigg|\sd t + \int_0^T\bigg| \weight{\partial_t \we_0, \we_1}_{V',V} \bigg|\sd t
 \nonumber\\&\quad+\int_{0}^T\bigg|\int_{\Omega}\big((\ve_\eps\cdot \nabla \we_0+\we_0\cdot\nabla \ve_A )\cdot\we_1+\big(2\nu(c_\eps) D\we_0:\nabla\we_1\big)\big)\sd x\bigg|\sd t\bigg).\label{wenergyequ}
\end{align}
Now we estimate the different terms on the right-hand side separately.

The most important step is to show
\begin{align}
&\eps\int_0^T\bigg|\int_{\Omega}\nabla u\otimes \nabla c_A:\nabla\we_1\sd x\bigg|\sd t
 \nonumber\\
  &\quad \leq C(R)\eps^{\order+\frac{1}{4}}\|\nabla\we_1\|_{L^2(\Omega\times(0,T))}+ C(R,M)\eps^{\order+\frac{1}{4}}(T^{\frac12}+\eps^{\frac14})\|\nabla\we_1\|_{L^2(\Omega\times(0,T))}.\label{w1est-1}
\end{align}
To this end we decompose $\Omega$ into  $\Omega\backslash\Gamma_t^\eps(\frac{3\delta}{2})$ and $\Gamma_t^\eps(\frac{3\delta}{2})$ and split the integrals accordingly. Then the proof of \eqref{w1est-1} will
consist of three parts.

First of all,  we have
\begin{align}
&\varepsilon\int_0^{T}\bigg|\int_{\Omega\backslash \Gamma_t^\eps(\frac{3\delta}{2})}\nabla u\otimes \nabla c_A:\nabla\we_1\sd x\bigg| \sd t                \nonumber\\&\qquad
\leq C(M)\varepsilon\|\nabla u\|_{L^2(\Omega\times (0,T)\backslash\Gamma(\frac{3\delta}2))\big)}\|\nabla\we_1\|_{L^2(\Omega\times(0,T))}\nonumber\\
  &\qquad
    \leq C(R,M)\eps^{\order+\frac{5}{4}}\|\nabla\we_1\|_{L^2(\Omega\times(0,T))}.\label{w1est-11}
\end{align}
Furthermore, since
\begin{equation}\label{eq:LowerBdd}
\big|d_\eps\big|\geq \frac{3\delta}{2}\quad \text{for all } x\in \Gamma_t(3\delta)\backslash\Gamma_t(\tfrac52\delta),\ t\in[0,T], \eps\in (0,\eps_0]
\end{equation}
due to \eqref{eq_Gamma_eps}, it holds
\begin{align}
  \nabla c_A (x,t)&=\nabla(\zg )\(\tc_A^{in}(\rho,s,t)-c_A^+\chi_+-c_A^-\chi_- \)+\varepsilon^{-1}\zg\theta_0'(\rho)\nn_\eps \nonumber\\&\quad
  +\zg\sum_{k=\frac{3}{2}}^{N+2}\eps^{k} \nabla_{\btau_\eps}\tc_k(\rho,s,t)+\zg\sum_{k=\frac{3}{2}}^{N+2}\eps^{k-1} \partial_{\rho}\tc_k(\rho,s,t)\nn_\eps\nonumber\\
  & \quad +\eps^{N-\frac74}\theta_0''(\rho)\no_\eps h_{N,\eps}\circ S_\eps+\eps^{N-\frac34}\theta_0'(\rho)\nabla( h_{N,\eps}\circ S_\eps)\nonumber\\
  &= \varepsilon^{-1}\zg\theta_0'(\rho)\nn_\eps+\eps^{N-\frac34}\theta_0'(\rho)\nabla( h_{N,\eps}\circ S_\eps)+O(\varepsilon^\frac{1}{2})
  \nonumber 
\end{align}
in $L^\infty(\Gamma(3\delta))$ because of the matching conditions.
Therefore we can estimate
\begin{align}
&\varepsilon\int_0^{T}\bigg|\int_{\Gamma^\eps_t(\frac{3\delta}{2})}\nabla u\otimes \nabla c_A:\nabla\we_1\sd x\bigg| \sd t
                \nonumber\\&\leq \underbrace{\int_0^{T}\int_{\Gamma^\eps_t(\frac{3\delta}{2})}\zg\theta_0'(\rho)\partial_{\nn_\eps} u\, \no_\eps \otimes \no_\eps :\nabla \we_1\sd x \sd t}_{=: I}\nonumber\\
  &\quad + C(M)\eps^{N-\frac54}\|\nabla u\|_{L^2\big(0,T;L^4(\Gamma^\eps_t(\frac{3\delta}{2}))\big)}\|\nabla\we_1\|_{L^2(\Omega\times (0,T))}
  \nonumber\\&\quad+ C(M)\varepsilon^{\frac{3}{2}}\|\nabla u\|_{L^2\big(0,T;L^2(\Gamma^\eps_t(\frac{3\delta}{2}))\big)}\|\nabla\we_1\|_{L^2(\Omega\times (0,T))}\nonumber\\
  &\leq I + C(R,M)\eps^{N+\frac12}\|\nabla \we_1\|_{L^2(\Omega\times (0,T))}\label{w1est-3}
\end{align}
since $\sqrt\eps \partial_s h_{N,\eps}\in L^\infty(0,T;L^4(\T^1))$ is bounded due to \eqref{eq:hEstim2} and \eqref{eq:EstimLinStokesMCF2}.
Because of $\partial_{\no_\eps} (\theta_0'(\rho))^2= \frac1\eps \partial_\rho (\theta_0'(\rho))^2$ and Lemma~\ref{lem:DivergenceFreeRemainder} and using \eqref{decompose u}, we obtain
\begin{align}
|I|
&=\left|\,\int_0^{T}\int_{\Gamma^\eps_t(\frac{3\delta}{2})}\zg\theta_{0}'(\rho)\partial_{\nn_\eps}\left(
\eps^{-\frac{1}{2}}Z(S_\eps,t)\left(\beta_\eps\theta_{0}'(\rho)+\Psi_\eps(\rho,S_\eps,t)\right)+\psi^{\mathbf{R}}\right)\no_\eps \otimes \no_\eps :\nabla \we_1\, \sd x \sd t\right|\nonumber\\
& \leq\left|\,\int_0^{T}\int_{\Gamma^\eps_t(\frac{3\delta}{2})}\frac{1}{2}\partial_{\nn_\eps}\left(\theta_{0}'(\rho)^{2}\right)\varepsilon^{-\frac{1}{2}}Z(S_\eps,t)\beta_\eps \no_\eps \otimes \no_\eps :\nabla \we_1\sd x\sd t\right|\nonumber\\
&\quad+\,\left|\int_{0}^{T}\int_{\mathbb{T}^{1}}\int_{-\frac{3\delta}{2\varepsilon}}^{\frac{3\delta}{2\varepsilon}}\theta_{0}'(\rho)
\varepsilon^{-\frac{3}{2}}Z(s,t)\partial_{\rho}\Psi_\eps(\rho,s,t)\no_\eps \otimes \no_\eps :\nabla \we_1\circ X_\eps\, J_\eps\left(\eps\rho,s,t\right)
\sd\rho \sd s\sd t\right|\nonumber\\
  & \quad+C\left\Vert \psi^{\mathbf{R}}\right\Vert _{L^{2}\left(0,T;H^{1}\left(\Gamma^\eps_t(\frac{3\delta}{2})\right)\right)}\left\Vert \we_1\right\Vert _{L^{2}\left(0,T;H^{1}\left(\Gamma^\eps_t(\frac{3\delta}{2})\right)\right)}\nonumber\\
  &\quad +C(M)e^{-\frac{\delta}{2\varepsilon}}\|\nabla u\|_{L^2\big(\Gamma^\eps(\frac{3\delta}{2})\big)}\|\we_1\|_{L^2(\Omega\times (0,T))}\nonumber\\
  & \leq C\|Z\|_{L^2(0,T;H^1(\mathbb{T}^{1}))}\|\we_1\|_{L^2(0,T;H^1(\Omega))}\nonumber\\
  &\quad +\frac{C}\eps \bigg(\sup_{t\in[0,T]}\sup_{s\in\T^1}\int_{I_{\varepsilon}}|\partial_{\rho} \Psi_\eps(\tilde\rho,s,t)|^2J(\eps \tilde\rho,s,t)d\tilde{\rho}\bigg)^{\frac{1}{2}}\|Z\|_{L^2(0,T;H^1(\mathbb{T}^{1}))}\|\we_1\|_{L^2(0,T;H^1(\Omega))}\nonumber
\\
& \quad+C\left\Vert \psi^{\mathbf{R}}\right\Vert _{L^{2}\left(0,T;H^{1}\left(\Gamma^\eps_t(\frac{3\delta}{2})\right)\right)}\left\Vert \we_1\right\Vert _{L^{2}\left(0,T;H^{1}\left(\Gamma^\eps_t(\frac{3\delta}{2})\right)\right)}\nonumber\\
& \quad+C(M)e^{-\frac{\delta}{2\varepsilon}}\|\nabla u\|_{L^2\big(0,T;L^2(\Gamma^\eps_t(\frac{3\delta}{2}))\big)}\|\we_1\|_{L^2\big(0,T;L^2(\Omega)\big)}.\label{estimation:I}
\end{align}
Now using \eqref{f2u estimate-2} and \eqref{f1u estimate} we derive
\begin{align}
|I|\leq C\bigg(\bigg\|\frac{\Lambda_\varepsilon}{\varepsilon}\bigg\|_{L^1(0,T)}^{\frac{1}{2}}+\|u\|_{L^2\big(0,T;L^2(\Gamma^\eps_t(\frac{3\delta}{2}))\big)}\bigg)\|\nabla\we_1\|_{L^2(\Omega\times (0,T))}
,\label{w1est-7}
\end{align}
where
\begin{equation*}
  \Lambda_\eps = \int_{\Gamma^\eps_t(\frac{3\delta}{2})}\left(\eps |\nabla u|^2 + \frac1\eps f''(c_A^\eps)u^2\right) \sd x.
\end{equation*}
Altogether \eqref{assumptions}, \eqref{w1est-7}, and \eqref{w1est-3} yield
\begin{align}
&\varepsilon\int_0^{T}\bigg|\int_{\Gamma^\eps_t(\frac{3\delta}{2})}\nabla u\otimes \nabla \tilde{c}_A:\nabla\we_1\sd x\bigg| \sd t
                \nonumber\\&
\leq   C(R)\eps^{\order+\frac{1}{4}}
  \|\nabla\we_1\|_{L^2(\Omega\times (0,T))}+C(R,M)\eps^{\order+\frac12}\|\nabla\we_1\|_{L^2(\Omega\times (0,T))}.
  \label{w1est-0}
\end{align}
This shows \eqref{w1est-1} because of \eqref{w1est-11} and \eqref{w1est-0}.

For the following we will use that by construction
\begin{equation*}
  \ve_A^{in}(x,t)= \tilde{\ve}_A^{in}(x,t) + \eps^{N+\frac14} \hat{\we}_\eps(x,t).
\end{equation*}
Using that $\|\we_\eps\|_{L^2(0,T;H^1(\Omega))}$ is uniformly bounded and $\hat{\ve}_0(\rho,x,t)=\ve_0^+(x,t)\eta(\rho)+\ve_0^-(x,t)(1-\eta(\rho))$ together with $\ve_0^+|_{\Gamma}=\ve_0^-|_\Gamma$, one can show similarly as in \cite[Estimate (4.28)]{AbelsFei}
\begin{equation*}
  \nabla\ve_A = \nabla \tilde{\ve}_A + \eps^{N+\frac14} \nabla\we_\eps,
\end{equation*}
where $\|\nabla\tilde{\ve}_A\|_{L^\infty(\Omega\times (0,T))}$, $\|\nabla \we_\eps\|_{L^2(\Omega\times (0,T))}$, $\sqrt{\eps}\|\nabla \we_\eps\|_{L^2(0,T;L^r(\Omega))}$ are uniformly bounded for every $1\leq r<\infty$.
Therefore we obtain
\begin{align}\nonumber
  &\int_0^T\bigg|\int_{\Omega}\bigg(\big(\nu(c_\eps)-\nu(c_A)\big) D\ve_A:\nabla\we_1\bigg)\sd x\bigg|\sd t\\\nonumber
  &\leq C(M)\left(T^{\frac{1}{2}}\| u\|_{L^\infty(0,T;L^2)}+\eps^{N+\frac14}\|u\|_{L^\infty(0,T;L^4)}\|\nabla \we_\eps\|_{L^2(0,T;L^4)}\right)\|\nabla\we_1\|_{L^2(\Omega\times (0,T))}\\
  &\leq C(R,M)(T^{\frac12}+\eps)\eps^{N+\frac12} \label{w1est-00}
\end{align}
because of $|\nu(c_\eps)-\nu(c_A)|\leq \|\nu'\|_{L^\infty(\R)}|u|$.

Next we use that
\begin{align}
\|\nabla u\|_{L^4(\Omega\times(0,T))}&\leq C\bigg(\|\nabla u\|_{L^{\infty}(0,T;L^2(\Omega))}^{\frac{1}{2}}\|\Delta u\|_{L^2(\Omega\times(0,T))}^{\frac{1}{2}}+T_0^{\frac{1}{4}}\|\nabla u\|_{L^{\infty}(0,T;L^2(\Omega))}\bigg)
\nonumber\\ &\leq C(R)\big(\eps^{\order-\frac{13}{8}}+\eps^{\order-\frac{3}{2}}\big)\leq C(R)\eps^{N-\frac{13}8}, \label{est:ul4l4}
\end{align}
which yields
\begin{align}\nonumber
  &\varepsilon\int_{0}^T\bigg|\int_{\Omega}\bigg(\nabla u\otimes\nabla u:\nabla\we_1\bigg)\sd x\bigg|\sd t\leq C\varepsilon\|\nabla u\|_{L^4(0,T;L^4(\Omega))}^2\|\nabla \we_1\|_{L^2(\Omega\times (0,T))}
\nonumber\\&\qquad\qquad\qquad\qquad\leq C(R)\varepsilon^{2N-\frac{9}{4}}\|\nabla \we_1\|_{L^2(\Omega\times (0,T))}\leq C(R)\varepsilon^{N+\frac{1}{4}}\|\nabla \we_1\|_{L^2(\Omega\times (0,T))}
\label{w2est-2}
\end{align}
due to $N\geq \frac52$.
Because of \eqref{estimation:Reps},
\begin{align}
  \int_{0}^T\bigg|\int_{\Omega}R_\eps\cdot\we_1\sd x\bigg|\sd t
                                                                    &\leq C(M)\eps^{N+\frac14}\|\we_1\|_{L^\infty(0,T;L^2(\Omega))}\nonumber\\
  &\quad + C(R,M)\eps^{N+\frac14}(T^{\frac{1}{2}}+\eps^{\frac14})\|\we_1\|_{L^\infty(0,T;L^2(\Omega))}.\label{w2est-3}
\end{align}
Using \eqref{est:w0} one obtains as in \cite[Proof of Theorem~4.1]{AbelsFei}
\begin{align}\nonumber
  &\int_0^T \left|\weight{\partial_t \we_0, \we_1}_{V',V} \right|\sd t +\int_{0}^T\bigg|\int_{\Omega}\bigg(  (\ve_\eps\cdot \nabla \we_0-\we_0\cdot\nabla \ve_A)\cdot\we_1  +\big(2\nu(c_\eps) D\we_0:\nabla\we_1\big)\bigg)\sd x\bigg|\sd t\\
  &\quad \leq C(M)\varepsilon^{N+\frac12}\|\nabla\we_1\|_{L^2(\Omega\times (0,T))}+C(M) \eps^{N+\frac12}\|\we_1\|^{\frac{1}{2}}_{L^{\infty}(0,T;L^2)}\|\we_1\|^{\frac{1}{2}}_{L^2(0,T;H^1)},
 \label{w2est-5}
\end{align}
Combining \eqref{w1est-0}, \eqref{w1est-00}, \eqref{w2est-2}-\eqref{w2est-5} and utilizing \eqref{initial assumption}, Korn's and Young's  inequality we conclude
 \begin{equation}
    \|\we_1\|_{L^\infty(0,T;L^2(\Omega))}+\|\we_1\|_{L^2(0,T;H^1(\Omega))}\leq C(R)\eps^{\order+\frac{1}{4}}+C(R,M)\eps^{\order+\frac{1}{4}}(T^{\frac12}+\eps^{\frac14}).\label{w1barestimate}
  \end{equation}
Furthermore, by testing \eqref{neweq:w} with  $\boldsymbol\varphi\in L^2(0,T;V(\Omega))$ and using the similar arguments as above we arrive at
\begin{align*}
\|\partial_t\we_1\|_{L^2(0,T;V(\Omega)')} \leq C(R)\eps^{\order+\frac{1}{4}}+C(R,M)\eps^{\order+\frac{1}{4}}(T^{\frac12}+\eps^{\frac14}),
\end{align*}
which by interpolation leads to
\begin{align}\label{est:w1bar-11}
 \|\we_1\|_{H^{\frac{1}{2}}(0,T;L^2(\Omega))}\leq C(R)\eps^{\order+\frac{1}{4}}+C(R,M)\eps^{\order+\frac{1}{4}}(T^{\frac12}+\eps^{\frac14}).
\end{align}
Finally, \eqref{est:w0} and \eqref{w1barestimate}-\eqref{est:w1bar-11} yield the desired result.
\end{proof*}

Since by definition  $\we=\frac{\we_1}{\varepsilon^{N+\frac{1}{4}}}$, we get for $T\in(0,T_\varepsilon)$
 \begin{equation}
     \|\partial_t\we|_{V(\Omega)}\|_{L^2(0,T;V(\Omega)')}+\|\we\|_{L^\infty(0,T;L^2(\Omega))}+\|\we\|_{L^2(0,T;H^1(\Omega))}\leq C(R).\label{neww1estimate}
  \end{equation}

\begin{prop}\label{prop:ErrorConvectionTerm} For $T\in(0,T_\varepsilon)$ there holds
 \begin{alignat}{2}
\int_0^{T}\bigg|\int_{\Omega}\big(\we-\we|_{X_\eps (0,S_\eps,t)}\big)\cdot \nabla c_Au\, \sd x\bigg|\sd t&\leq C(R,T)\eps^{\order+\frac{3}{4}},\label{est:diffw1couple}
\end{alignat}
where $C(R,T)\rightarrow 0$ as $T\rightarrow 0$.
\end{prop}
\begin{proof}
Let
\begin{alignat}{1}
  c^{(0)}(x,t)&=\zg(x,t) \theta_0(\rho) +(1-\zg(x,t) )\(c_A^+(x,t)\chi_+(x,t) +c_A^-(x,t)\chi_-(x,t) \)\label{leading order app}
\end{alignat}
be the leading part of $c_A$.
Using that
   \begin{align}
   \bigg(\theta_0(\rho)-\big(\pm1\big|_{\Omega^{\pm}(t)}\big)\bigg)\nabla(\zg)=O(e^{-\frac{\alpha\delta}{2\varepsilon}}), \label{outer decay}
   \end{align}
due to \eqref{eq:LowerBdd},   we  obtain
 \begin{align}
&\int_{\Gamma^\eps_t(\frac{3\delta}{2})}\big(\we-\we|_{X_\eps (0,S_\eps,t)}\big)\cdot\nabla c^{(0)}u\sd x=J+O(e^{-\frac{\alpha\delta}{2\varepsilon}})\|\we\|_{H^1(\Omega)}\|u\|_{L^2(\Gamma^\eps_t(\frac{3\delta}{2}))},
\label{est:diffw1-1}
\end{align}
where
\begin{align*}
  J&:= \varepsilon^{-1}\int_{\Gamma^\eps_t(\frac{3\delta}{2})}\zg\big(\we-\we|_{X_\eps (0,S_\eps,t)}\big)\cdot\nn_\eps\theta_0'(\rho)u\, \sd x.
\end{align*}
Now we use that
\begin{align*}
  \partial_r X_\eps(r,s,t)= \no_\eps (X_\eps(0,s,t),t)+O(\eps^2)= \partial_r X_\eps(r',s,t)+O(\eps^2)
\end{align*}
for all $r,r'\in (-\tfrac{3\delta}2,\tfrac{3\delta}2)$, $s\in \T^1, t\in [0,T_0]$ because of \eqref{eq:XepsNoEps}.
Therefore we obtain
  \begin{align}
J&=\int_{-\frac{3\delta}{2}}^{\frac{3\delta}{2}}\int_{\T^1} \frac1\eps\int_0^r(\no_\eps(X_\eps(r,.)) \cdot \left(\partial_r X_\eps (r',.)\cdot (\nabla \we)(X_\eps(r',.))\right)\sd r'\, \theta_0'(\tfrac{r}\eps) u(X_\eps(r,.))J_\eps(r,.)\sd s\sd r\nonumber\\
&= - \int_{-\frac{3\delta}{2}}^{\frac{3\delta}{2}}\int_{\T^1}  \frac1\eps\int_0^{r}(\Div \we) (X_\eps(r',.))\sd r'\, \theta_0'(\tfrac{r}\eps) u(X_\eps(r,.))J_\eps(r,.)\sd s\sd r\nonumber\\
     &\quad - \int_{-\frac{3\delta}{2}}^{\frac{3\delta}{2}}\int_{\T^1} \frac1\eps\int_0^{r}\nabla S_\eps (X_\eps(r',.))\cdot (\partial_s X_\eps(r',.)\cdot \nabla \we)(X_\eps(r',.))\sd r'\, \theta_0'(\tfrac{r}\eps) u(X_\eps(r,.))J_\eps(r,.)\sd s\sd r\nonumber\\
    &\quad + \int_{-\frac{3\delta}{2}}^{\frac{3\delta}{2}}\int_{\T^1} \frac1\eps\int_0^r \mathbb{A}_\eps (r,r',s,t):(\nabla \we)(X_\eps(r',.))\sd r'\, \theta_0'(\tfrac{r}\eps) u(X_\eps(r,.))J_\eps(r,.)\sd s\sd r,
  \end{align}
  where $\mathbb{A}_\eps (r,r',s,t)= O(\eps^2)$ and
  \begin{align*}
    &\left|\int_{-\frac{3\delta}{2}}^{\frac{3\delta}{2}}\int_{\T^1} \frac1\eps\int_0^{r}\nabla S_\eps (X_\eps(r,.))\cdot (\partial_s X_\eps(r',.)\cdot \nabla \we)(X_\eps(r',.))\sd r'\, \theta_0'(\tfrac{r}\eps) u(X_\eps(r,.))J_\eps(r,.)\sd s\sd r\right|\\
    &= \left|\int_{-\frac{3\delta}{2}}^{\frac{3\delta}{2}}\int_{\T^1} \frac1\eps\int_0^{r}  \we(X_\eps(r',.)) \theta_0'(\tfrac{r}\eps) \cdot \partial_s\left( \nabla S_\eps (X_\eps(r',.) u( X_\eps(r,.))J_\eps(r,.)\right)\sd r'\sd s\sd r\right|\\
   &\leq  C\varepsilon^{\frac{1}{2}} \|\we(\cdot,t)\|_{L^2(\Omega)}^{\frac{1}{2}}\|\we(\cdot,t)\|_{H^1(\Omega)}^{\frac{1}{2}}\big(\|\nabla_{\btau_\eps}u\|_{L^2(\Gamma^\eps_t(\frac{3\delta}{2}))}+\|u\|_{L^2(\Gamma^\eps_t(\frac{3\delta}{2}))}\big)
  \end{align*}
  because of
    \begin{align*}
      \left|\tfrac1\eps\int_0^{r}\we(X_\eps(r,s,t)\sd r'\right| &\leq \frac{|r|}\eps \|\we (X_\eps(\cdot,s,t))\|_{L^\infty (-\frac{3\delta}2,\frac{3\delta}2)}\\
      &\leq \frac{|r|}\eps \|\we (X_\eps(\cdot,s,t))\|_{L^2 (-\frac{3\delta}2,\frac{3\delta}2)}^{\frac12}\|\we (X_\eps(\cdot,s,t))\|_{H^1 (-\frac{3\delta}2,\frac{3\delta}2)}^{\frac12}.
  \end{align*}
  Similarly we have
  \begin{align*}
    &\left|\int_{-\frac{3\delta}{2}}^{\frac{3\delta}{2}}\int_{\T^1}  \frac1\eps\int_0^{r}(\Div \we) (X_\eps(r',s,t))\sd r'\, \theta_0'(\tfrac{r}\eps) u(X_\eps(r,s,t))J_\eps(r,s,t)\sd s\sd r\right|\\
    &\leq C \|\Div \we\|_{L^2(\Omega)} \|u\|_{L^2(\Gamma^\eps_t(\frac{3\delta}{2}))}
  \end{align*}
  due to
    \begin{align*}
      \left|\tfrac1\eps\int_0^{r}(\Div\we)(X_\eps(r,s,t)\sd r'\right|  &\leq \frac{|r|^{\frac12}}\eps \|(\Div \we) (X_\eps(\cdot,s,t))\|_{L^2 (-\frac{3\delta}2,\frac{3\delta}2)}.
    \end{align*}
    Now using $\|\Div \we\|_{L^2(\Omega\times (0,T_\eps))} = O(\eps^{\frac14}) $, we obtain
 \begin{align}
   \int_{0}^T |J|\sd t
                      &\leq C(R)T^{\frac{1}{4}}\big(\varepsilon^{\frac{1}{2}}\|\nabla_{\btau_\eps}u\|_{L^2(\Omega\times (0,T))}+T^{\frac{1}{2}}\varepsilon^{\frac{1}{4}}\|u\|_{L^\infty(0,T;L^2(\Omega))}\big) \nonumber+ CT^{\frac12}\eps^{N+1}
   \nonumber\\&\leq C(R)T^{\frac{1}{4}}(1+T^{\frac{1}{2}})\eps^{\order+\frac{3}{4}}.
 \label{est:j1}
  \end{align}
Combining \eqref{est:j1} and \eqref{est:diffw1-1}
 we obtain
 \begin{align}
\int_{0}^T \bigg|\int_{\Gamma^\eps_t(\frac{3\delta}{2})}\big(\we-\we|_{X_\eps (0,S_\eps,t)}\big)\cdot\nabla c^{(0)} u\, \sd x\bigg|\sd t\leq C(R,T)\eps^{\order+\frac{3}{4}},
\label{est:diffw1-2}
\end{align}
where $C(R,T)\rightarrow 0$ as $T\rightarrow 0$.

Finally, the corresponding estimate $\nabla (c_A- c^{(0)})$ can be done in a straight forward manner since all terms are of higher order in $\eps$ (by at least a factor $\eps^{\frac32}$) compared to $\nabla c^{(0)}$. This finishes the proof.
\end{proof}

\begin{rem}\label{rem_ErrorConvectTerm}
	Analogous to Proposition \ref{prop:ErrorConvectionTerm} one can show that
	 \begin{align}
		\int_0^{T}\bigg|\int_{\Omega}\big(\we_\eps-\we_\eps|_{X_\eps (0,S_\eps,t)}\big)\cdot \nabla c_Au\, \sd x\bigg|\sd t&\leq C(R,T)\eps^{\order+\frac{3}{4}},
	\end{align}
	where $C(R,T)\rightarrow 0$ as $T\rightarrow 0$. To this end one uses the same computations and estimates as in the proof of Proposition \ref{prop:ErrorConvectionTerm} for $\we_\eps$ instead of $\we$ and the estimate \eqref{eq_estimate_wA}.
\end{rem}

\subsection{Proof of the Main Result Theorem \ref{thm:main}}

In order to estimate the error due to linearization of $\eps^{-\frac32}f'(c)$ we need:
\begin{prop}\label{prop:u3Estim}
  Under the assumptions of Section~\ref{subsec:LeadingErrorVelocity} we have for every $T\in (0,T_\eps)$
\begin{align}
\int_0^{T}\int_{\Omega} |u|^3\sd x\sd t\leq C(R)T^{\frac{1}{2}}\varepsilon^{3N+\frac{7}{8}}.\label{est:nonlinear}
 \end{align}
\end{prop}
\begin{proof}
  The proof is almost identical to \cite[Proposition~4.3]{AbelsFei} with $\Gamma_t(\delta)$ and $\nabla_\btau u$ replaced by $\Gamma_t^\eps (\tfrac{3\delta}2)$, $\nabla_{\btau_\eps} u$, respectively. In the present case the power of $\eps$ in the estimate for $\|\nabla_{\btau_\eps} u\|_{L^2}^{\frac12}$ is decreased by $\frac18$, which cause the loss of $\frac18$ in the power of $\eps$ in the present case compared to \cite[Proposition~4.3]{AbelsFei}.
\end{proof}

In the following the proof is similar to \cite[Section~4.2]{AbelsFei}. But because of the different powers of $\eps$ in the estimates, some terms are critical compared to \cite{AbelsFei} and we have to choose additionally $T>0$ sufficiently small to finally control all terms.

First of all, by definition $  \ve_\eps= \ve_A+ \varepsilon^{N+\frac{1}{4}}\we-\we_0$. Therefore \eqref{eq:NSAC3} and \eqref{eq:ApproxS3} (with $\ue=\we$) imply
 \begin{alignat}{2}\nonumber
   \partial_t u&+\ve_\eps\cdot \nabla u+ \eps^{N+\frac14}\big(\we-\we|_{X_\eps (0,S_\eps,t)}-(\we_\eps-\we_\eps|_{X_\eps (0,S_\eps,t)})\big)\cdot \nabla c_A-\we_0\cdot \nabla c_A\\\label{eq:u}
   &= \eps^{\frac{1}{2}}\Delta u-\eps^{-\frac{3}{2}}f''(c_A)u -\eps^{-\frac{3}{2}}\mathcal{N}(c_A,u)- s^1_\eps-s^2_\eps,
  \end{alignat}
where
$
  \mathcal{N}(c_A,u)=f'(c_\eps)-f'(c_A)-f''(c_A)u.
$
 Taking the $L^2(\Omega)$-inner product of \eqref{eq:u} and $u$, using integration by parts we obtain
  \begin{alignat}{2}\label{ineq:u energy-00}
    &\frac{1}{2}\frac{d}{dt}\int_{\Omega} u^2\sd x+ \eps^{\frac{1}{2}}\int_{\Omega}\big( |\nabla u|^2+\frac{f''(c_A)}{\eps^2}u^2\big) \sd x\nonumber\\&\leq\eps^{N+\frac14}\bigg|\int_{\Omega}\big(\we-\we|_{X_\eps (0,S_\eps,t)}-(\we_\eps-\we_\eps|_{X_\eps (0,S_\eps,t)})\big)\cdot \nabla c_Au\, \sd x\bigg|+\bigg|\int_{\Omega}\we_0\cdot \nabla c_Au\, \sd x\bigg|\nonumber\\&\qquad+\eps^{-\frac{3}{2}}C\int_{\Omega}|u|^3\sd x+\bigg|\int_{\Omega} s_\eps u\, \sd x\bigg|
  \end{alignat}
  because of
  \begin{align*}
  \int_{\Omega} \mathcal{N}(c_A,u)u\sd x\geq-C\int_{\Omega} |u|^3\sd x.
  \end{align*}
  Application of Lemma \ref{thm:Spectral} now yields
  \begin{alignat}{2}\label{ineq:u energy}
    &\frac{1}{2}\frac{d}{\sd t}\int_{\Omega} u^2\sd x-C_L\eps^{\frac{1}{2}}\int_\Om u^2 \sd x +\eps^{\frac{1}{2}} \int_{\Om\setminus \Gamma^\eps_t(\frac{3\delta}{2})} |\nabla u|^2\sd x + \eps^{\frac{1}{2}} \int_{\Gamma^\eps_t(\frac{3\delta}{2})} |\nabla_{\btau_\eps} u|^2\sd x\nonumber\\&\leq\eps^{N+\frac14}\bigg|\int_{\Omega}\big(\we-\we|_{X_\eps (0,S_\eps,t)}-(\we_\eps-\we_\eps|_{X_\eps (0,S_\eps,t)})\big)\cdot \nabla c_Au\sd x\bigg|+\bigg|\int_{\Omega}\we_0\cdot \nabla c_Au\sd x\bigg|\nonumber\\&\qquad+C\eps^{-\frac{3}{2}}\int_{\Omega}|u|^3\sd x+\bigg|\int_{\Omega} s_\eps u\, \sd x\bigg|.
  \end{alignat}
Therefore we obtain
  \begin{alignat}{2}\label{ineq:u energy-1}
    &\sup_{0\leq t\leq T}\frac{1}{2}\int_{\Omega} u(x,t)^2\sd x + \eps^{\frac{1}{2}}\int_{0}^T\int_{\Om\setminus \Gamma^\eps_t(\frac{3\delta}{2})} |\nabla u|^2\sd x \sd t +  \eps^{\frac{1}{2}}\int_{0}^T\int_{\Gamma^\eps_t(\frac{3\delta}{2})} |\nabla_{\btau_\eps} u|^2d x \sd t\nonumber\\&\leq e^{C_LT_0}\bigg(\frac{1}{2}\int_{\Omega} u(x,0)^2\sd x +\eps^{N+\frac14}\int_{0}^T\bigg|\int_{\Omega}\big(\we-\we|_{X_\eps (0,S_\eps,t)}-(\we_\eps-\we_\eps|_{X_\eps (0,S_\eps,t)})\big)\cdot \nabla c_Au\,\sd x\bigg|\sd t\nonumber\\&\qquad\qquad+\bigg|\int_{\Omega}\we_0\cdot \nabla c_Au\sd x\bigg|+C\eps^{-\frac{3}{2}}\int_{0}^T\int_{\Omega}|u|^3\sd x\sd t+\int_{0}^T\bigg|\int_{\Omega} s_\eps u\, \sd x\bigg|\sd t\bigg)
    \nonumber\\&\leq \frac{R^2}{8}\eps^{2N+1}+C'(R,T)\big(\eps^{2N+1}+\eps^{3N-\frac58}\big)\leq \frac{R^2}{8}\eps^{2N+1}+C(R,T)\eps^{2N+1}
  \end{alignat}
  for all $0\leq T\leq T_\varepsilon$
  due to \eqref{estimation:Seps}, \eqref{initial assumption-0}, Proposition~\ref{prop:ErrorConvectionTerm}, Remark \ref{rem_ErrorConvectTerm}, Proposition~\ref{prop:u3Estim}, Gronwall's inequality, and $N\geq 3$, where $C(R,T),C'(R,T)\to_{T\to 0} 0$.
  Hence, if $\varepsilon\in\min(0,\varepsilon_0)$ and $\varepsilon_0>0$ and $T>0$ are sufficiently small, we have
  \begin{align*}
\frac{R^2}{8}\eps^{2N+1}+C(R,T)\eps^{2N+1}\leq \frac{R^2}{4}\eps^{2N+1}
  \end{align*}
and therefore
 \begin{alignat}{2}\nonumber
   &\sup_{0\leq t\leq T}\frac{1}{2}\int_{\Omega} u(x,t)^2\sd x +\eps^{\frac{1}{2}} \int_{\Om\times (0,T)\setminus \Gamma^\eps(\frac{3\delta}{2})} |\nabla u|^2\sd (x,t) \\\label{ineq:u energy-2}
   &\quad+ \eps^{\frac{1}{2}} \int_{\Omega\times (0,T)\cap \Gamma^\eps(\frac{3\delta}{2})} |\nabla_{\btau_\eps} u|^2\sd (x,t)\leq \frac{R^2}{4}\eps^{2N+1}.
  \end{alignat}
Combining this estimate with  \eqref{ineq:u energy-00} we obtain
  \begin{alignat}{2}
    &\eps^{\frac{1}{2}}\int_0^T \int_{\Omega}\left( |\nabla u|^2+\frac{f''(c_A)}{\eps^2}u^2\right) \sd x \sd t\leq\frac{1}{2}\int_{\Omega} u(x,0)^2\sd x
     \nonumber\\&\quad +\eps^{N+\frac14}\int_{0}^T\bigg|\int_{\Omega}\big(\we-\we|_{X_\eps (0,S_\eps,t)}-(\we_\eps-\we_\eps|_{X_\eps (0,S_\eps,t)})\big)\cdot \nabla c_Au\sd x\bigg|\sd t+\bigg|\int_{\Omega}\we_0\cdot \nabla c_Au\sd x\bigg|\nonumber\\&\quad+C\eps^{-\frac{3}{2}}\int_{0}^T\int_{\Omega}|u|^3\sd x\sd t+\int_{0}^T\bigg|\int_{\Omega}s_\eps u\sd x\bigg|\sd t
    \nonumber\\\label{ineq:u energy-4}
    &\quad\leq\frac{R^2}{8}\eps^{2N+1}+C(R,T)\eps^{2N+1}\leq \frac{R^2}{4}\eps^{2N+1}.
  \end{alignat}
   for $T\in (0, T_\varepsilon]$ sufficiently small and, if $\eps_0>0$ is sufficiently small,
  \begin{alignat}{2}\label{ineq:u energy-5}
   \eps^{2}\int_0^T \int_{\Omega}|\partial_{\nn_\eps} u|^2\sd x \sd t&\leq\eps^{2}\int_0^T \int_{\Omega}|\nabla u|^2\sd x \sd t
    \nonumber\\&\leq\int_0^T \int_{\Omega}\big( \eps^{2}|\nabla u|^2+f''(c_A)u^2\big) \sd x \sd t+C\int_0^T \int_{\Omega}u^2\sd x \sd t\nonumber\\&\leq\frac{R^2}{4}\eps^{2N+\frac{5}{2}}+CT\frac{R^2}{2}\eps^{2N+1}
    \leq\frac{3R^2}{4}\eps^{2N+1}.
  \end{alignat}

Next we derive the  estimates for $\Delta u$ in $L^2(\Omega\times (0,T))$ and $\nabla u$ in $L^\infty(0,T;L^2(\Omega))$. To this end we
 take the $L^2(\Omega)$-inner product of \eqref{eq:u} and $-\varepsilon^4\Delta u$, integrate by parts, and obtain
  \begin{alignat}{2}\label{ineq:u energy-6}
    &\sup_{0\leq t\leq T}\frac{\varepsilon^4}{2}\int_{\Omega} |\nabla u|^2(x,t)\sd x+ \varepsilon^{\frac{9}{2}}\int_0^T\int_{\Omega}|\Delta u|^2\sd x\sd t
    \nonumber\\&\leq\frac{\varepsilon^4}{2}\int_{\Omega} |\nabla u|^2(x,0)\sd x+\varepsilon^{\frac{5}{2}}\int_0^T\int_{\Omega}\big|f''(c_A)u\Delta u\big|\sd x \sd t
    +\varepsilon^4\int_0^T\int_{\Omega}\big|\ve_\eps\cdot \nabla u\Delta u  \big|\sd x \sd t
    \nonumber\\&\quad+\eps^{N+\frac{17}{4}}\int_0^T\int_{\Omega}\big|\big(\we-\we|_{X_\eps (0,S_\eps,t)}-(\we_\eps-\we_\eps|_{X_\eps (0,S_\eps,t)})\big)\cdot \nabla c_A\Delta u\big|\sd x\sd t
    \nonumber\\&\quad+\varepsilon^4\bigg|\int_{\Omega}\we_0\cdot \nabla c_A\Delta u\sd x\bigg|+\varepsilon^{\frac{5}{2}}\int_0^T\int_{\Omega}\big|\mathcal{N}(c_A,u)\Delta u\big|\sd x \sd t+\varepsilon^4\int_0^T\bigg|\int_{\Omega}s_\eps\Delta u\,\sd x\bigg|\sd t,
  \end{alignat}
where
   \begin{alignat}{2}\label{ineq:u energy-7}
\varepsilon^{\frac{9}{2}}\int_0^T\int_{\Omega}\big|f''(c_A)u\Delta u\big|\sd x \sd t&\leq C\varepsilon^{\frac{9}{2}} T^{\frac{1}{2}}\|u\|_{L^\infty(0,T;L^2(\Omega))}\|\Delta u\|_{L^2(0,T;L^2(\Omega))}
\nonumber\\&\leq CR\varepsilon^{N+3} T^{\frac{1}{2}}\|\Delta u\|_{L^2(\Omega\times (0,T))}.
\end{alignat}
  As in \cite[Section~4.2]{AbelsFei} one estimates
    \begin{alignat}{2}\label{ineq:u energy-13}
    &\eps^{2}\int_0^T\int_{\Omega}\big|\mathcal{N}(c_A,u)\Delta u\big|\sd x \sd t 
    \leq CR^2\eps^{2N+2}\|\Delta u\|_{L^2(0,T;L^2(\Omega))}
  \end{alignat}
and, using  $   \ve_\eps= \ve_A+ \varepsilon^{N+\frac{1}{4}}\we-\we_0$,
   \begin{alignat}{2}\label{ineq:u energy-11}
    \varepsilon^4\int_0^T\int_{\Omega}\big|\ve_\eps\cdot \nabla u\Delta u  \big|\sd x \sd t&\leq C R\varepsilon^{N+\frac{7}{2}}\|\Delta u\|_{L^2(0,T;L^2(\Omega))}\nonumber\\&\quad+C(R)\eps^{N+\frac{17}{4}}\|\nabla u\|_{L^\infty(0,T;L^2(\Omega))}^{\frac{1}{2}}\|\Delta u\|_{L^2(0,T;L^2(\Omega))}^{\frac{3}{2}}.
  \end{alignat}
Furthermore
    \begin{alignat}{2}\nonumber
    &\eps^{N+\frac92}\int_0^T\int_{\Omega}\big|\big(\we-\we|_{X_\eps (0,S_\eps,t)}-(\we_\eps-\we_\eps|_{X_\eps (0,S_\eps,t)})\big)\cdot \nabla c_A\Delta u\big|\sd x\sd t\\
    &\qquad \leq C(R)\eps^{N+\frac{13}{4}}\|\Delta u\|_{L^2(0,T;L^2(\Omega))},\label{ineq:u energy-12}
    \\&\varepsilon^4\bigg|\int_{\Omega}\we_0\cdot \nabla c_A\Delta u\sd x\bigg|\leq  C(R)\eps^{N+\frac{9}{2}}\|\Delta u\|_{L^2(0,T;L^2(\Omega))}\label{ineq:u energy-13}
  \end{alignat}
 and
 \begin{alignat}{2}\nonumber
   \varepsilon^4\int_0^T\bigg|\int_{\Omega}s_\eps\Delta u\,\sd x\bigg|\sd t&\leq \varepsilon^4\|s^1_\eps+s^2_\eps\|_{L^2(0,T;L^2(\Omega))}\|\Delta u\|_{L^2(0,T;L^2(\Omega))}\\\label{ineq:u energy-14}
   &\leq C\varepsilon^{N+\frac{15}{4}}\|\Delta u\|_{L^2(0,T;L^2(\Omega))}.
 \end{alignat}
 due to \eqref{estimation:Seps'}.

 Now using \eqref{initial assumption-0}, \eqref{ineq:u energy-7} and \eqref{ineq:u energy-11}-\eqref{ineq:u energy-14} in \eqref{ineq:u energy-6} leads to
  \begin{alignat}{2}
    &\sup_{0\leq t\leq T}\frac{\varepsilon^4}{2}\int_{\Omega} |\nabla u|^2(x,t)\sd x+ \varepsilon^{\frac{9}{2}}\int_0^T\int_{\Omega}|\Delta u|^2\sd x\sd t
    \nonumber\\&\leq \frac{R^2}{8}\eps^{2\order+1}+C(R)\eps^{N+\frac{17}4}\|\nabla u\|_{L^\infty(0,T;L^2(\Omega))}^{\frac{1}{2}}\|\Delta u\|_{L^2(0,T;L^2(\Omega))}^{\frac{3}{2}}
    \nonumber\\\nonumber&\quad+C(R)\varepsilon^{N+3} \|\Delta u\|_{L^2(0,T;L^2(\Omega))}.
  \end{alignat}
 An application of Young's inequality yields
 \begin{alignat}{2}\label{ineq:u energy-16}
    &\sup_{0\leq t\leq T}\frac{3\varepsilon^4}{8}\int_{\Omega} |\nabla u|^2(x,t)\sd x+ \frac{3}{8}\varepsilon^{\frac{9}{2}}\int_0^T\int_{\Omega}|\Delta u|^2\sd x\sd t
    \nonumber\\&\qquad \leq \frac{R^2}{8}\eps^{2\order+1}+C(R)\eps^{2N+\frac32}\leq \frac{R^2}{4}\eps^{2\order+1},
  \end{alignat}
and
 \begin{alignat}{2}\label{ineq:u energy-17}
    \sup_{0\leq t\leq T}\int_{\Omega} |\nabla u|^2(x,t)\sd x + \eps^{\frac{1}{2}}\int_{0}^T\int_{\Omega} |\Delta u|^2\sd x \sd t \leq \frac{2R^2}{3}\eps^{2N-3}.
  \end{alignat}
  if $\varepsilon_0>0$ is small enough.

  Altogether we see from \eqref{ineq:u energy-2}, \eqref{ineq:u energy-4}, \eqref{ineq:u energy-5} and \eqref{ineq:u energy-17} and the definition of $T_\eps$ that there are $\eps_0>0$ and $T_1>0$ such that
  $T_\varepsilon>T_1$ for all $\eps\in (0,\eps_0)$ and therefore \eqref{assumptions} hold true for $\tau =T_1$.

  Finally, \eqref{eq:convVelocityb} follows from $\tilde{\we}=\ve_\eps-\ve_A$ and Theorem~\ref{thm:ErrorVelocity}, in particular \eqref{w1estimate}, and the remaining two conclusions in Theorem \ref{thm:main} are a consequence of the constructions of $c_A$ and  $\ve_A$. This finishes the proof of Theorem \ref{thm:main}.
  
\appendix

\section{Wellposedness of Linearized Two-Phase Flow System}
\begin{thm}\label{solvingkorder} Let $\Omega^\pm, \Gamma$ be smooth and as in Section~\ref{sec:Introduction} with $V_{\Gamma_t}= \no_{\Gamma_t}\cdot \ve_0^\pm$ for all $t\in[0,T_0]$ and some smooth $\ve_0^\pm \colon \ol{\Omega^\pm}\to \R^2$ with $\Div \ve_0^\pm =0$ in $\Omega^\pm$. Moreover, let $T\in (0,T_0]$, $q\in (4,\infty)$, and $\mathbf{b}_j\colon \Gamma\to \R^2 $, $j=0,1,2$, $a_k\colon \T^1\times [0,T_0]\to \R$, $k=0,1$ be smooth. Then for all
  \begin{align*}
    \mathbf{f}&\in L^q(\Omega\times (0,T))^2, \quad
    g\in L^q(0,T;W^1_q(\Omega\setminus \Gamma_t)), \quad w \in L^q(0,T;W_q^{2-\frac1q}(\T^1)), \\
    \mathbf{a}_1&\in L^q(0,T;W_q^{2-\frac1q}(\Gamma_t))^2\cap W_q^{\frac12-\frac1{2q}}(0,T;L^q(\Gamma_t))^2,\quad     \mathbf{v}_0 \in W^{2-\frac2q}_q(\Omega\setminus \Gamma_0)^2, \\
    \mathbf{a}_2&\in L^q(0,T;W_q^{1-\frac1q}(\Gamma_t))^2\cap W_q^{\frac12-\frac1{2q}}(0,T;L^q(\Gamma_t))^2, \quad h_0 \in W^{3-\frac2q}_q(\T^1), \\
    \mathbf{a}&\in L^q(0,T;W^{2-\frac1q}_q(\partial\Omega))^2\cap W_q^{1-\frac1{2q}}(0,T;L^q(\partial\Omega))^2
  \end{align*}
  satisfying
\begin{align}\label{eq:CompCondStokes}
\int_{\Omega\setminus \Gamma_t}g\,\sd x=-\int_{\Gamma_{t}}\mathbf{a}_{1}\cdot\mathbf{n}_{\Gamma_{t}}\sd\mathcal{\mathcal{H}}^{1}+\int_{\partial\Omega}\mathbf{a}\cdot\mathbf{n}_{\partial\Omega}\sd\mathcal{H}^{1}
\end{align}
for almost all $t\in (0,T)$ such that $g=\Div R$ for some $R\in W^1_q(0,T;L^q(\Omega))^2$, $\Div \mathbf{v}_0|_{t=0} = g|_{t=0}$ and $\llbracket\mathbf{v}_0\rrbracket =\mathbf{a}_{1}|_{t=0}+ \mathbf{b}_0 h_0\circ S_0|_{t=0}$, $\ve^-_0|_{\partial\Omega}= \mathbf{a}|_{t=0}$, there are unique
\begin{align*}
  \mathbf{v}&\in L^q(0,T; W^2_q(\Omega\setminus \Gamma_t))^2\cap W_q^1(0,T; L^q(\Omega))^2, \quad
  p\in L^q(0,T; W^1_{q,(0)}(\Omega\setminus \Gamma_t)),\\
  h&\in L^q(0,T;W_q^{3-\frac1q}(\T^1))\cap W^1_q(0,T;W^{2-\frac1q}_q(\T^1))
\end{align*}
solving
\begin{align}\label{eq:CoupledStokes1}
\partial_t\mathbf{v}^{\pm}-\nu^\pm\Delta\ve^{\pm}+\nabla p^{\pm}  &=\mathbf{f} &  &\text{in }\Omega^{\pm}_t, t\in (0,T),\\\label{eq:CoupledStokes2}
\operatorname{div}\mathbf{v}^{\pm} & =g  &  &\text{in }\Omega^{\pm}_t, t\in (0,T),\\\label{eq:CoupledStokes3}
\llbracket\mathbf{v}\rrbracket - \mathbf{b}_0 h\circ S_0 & =\mathbf{a}_{1}  &  & \text{on }\Gamma_t, t\in (0,T),\\\label{eq:CoupledStokes5}
  \llbracket 2\nu D\ve-p \tn{I}\rrbracket\no +\ol{\nu}\llbracket \partial_{\no} \ve\rrbracket-\sigma\Delta^{\Gamma_t} h\circ S_0\no &=   \mathbf{b}_1\partial_s h\circ S_0\nonumber\\
  &\quad +\mathbf{b}_2h\circ S_0+ \mathbf{a}_2  &  & \text{on }\Gamma_t, t\in (0,T),\\\label{eq:CoupledStokes6}
\partial_t h+a_1\partial_s h +\tfrac12(\ve^{+}_\no+\ve^{-}_\no)\circ X_0 + a_0h&=w &  & \text{on }\T^1\times(0,T),\\ \label{eq:CoupledStokesLast}
  \mathbf{v}^{-} & =\mathbf{a}  &  &\text{on }\partial\Omega\times (0,T),\\\label{eq:CoupledIV1}
  \mathbf{v}^\pm|_{t=0} &= \mathbf{v}_0^\pm &&\text{in }\Omega,\\\label{eq:CoupledIV2}
    h|_{t=0} &= h_0 &&\text{in }\T^1,
\end{align}
where $\ve^\pm = \ve|_{\Omega^\pm}$, $p^\pm= p|_{\Omega^\pm}$, $\no=\no_{\Gamma_t}$.
\end{thm}
\begin{proof}
First of all by subtracting suitable extensions of $\mathbf{a}$, $\ve_0$, and $h_0$ we can easily reduce to the case $\mathbf{a}=0$, $\ve_0=0$, and $h_0=0$.

\medskip

\noindent
  \textbf{Step 1: Time-independent interface, zero lower order terms:}
  Let us first consider the case that $\Gamma_t=\Gamma$, $\Omega^\pm_t =\Omega^\pm$ are independent of $t\in [0,T_0]$, $\mathbf{b}_j=0$ for $j=0,1,2$, $b_0=0$.
  First of all by subtracting suitable extensions, we can easily reduce to the case that $g=0$ and $\mathbf{a}_1=0$. Then
  \begin{equation*}
    \ol{\nu}\llbracket \nabla \ve_{\no}\rrbracket=0\qquad \text{if } \llbracket \ve\rrbracket=0\text{ on }\Gamma,\ \Div \ve^\pm =0\text{ in }\Omega^\pm. 
  \end{equation*}
  Hence in this case \eqref{eq:CoupledStokes5} is equivalent to
  \begin{equation}\label{eq:CoupledStokes5'}
    \llbracket 2(\nu+\ol{\nu}) D\ve-p \tn{I}\rrbracket\no =\sigma\Delta^\Gamma h\no +\mathbf{a}_2.
  \end{equation}
  Thus the result follows e.g.\ from \cite[Corollary~8.1.3]{PruessSimonettMovingInterfaces}.

\medskip

  \noindent \textbf{Step 2: Existence for $T=T_1>0$ sufficiently small:} Let $\Phi\colon \ol{\Omega}\times [0,T_0]\to \ol{\Omega}$ be defined by
  \begin{alignat*}{2}
    \frac{d}{dt} \Phi_t(\xi) &= \ve_0^\pm(\Phi_t(\xi),t) &\quad & \text{for all }\xi\in\overline{\Omega^\pm_0}, t\in [0,T_0],\\
    \Phi_0(\xi) &= \xi &\quad & \text{for all }\xi\in\overline{\Omega^\pm_0}.
  \end{alignat*}
  Then $\Phi$ is smooth in $\ol{\Omega^\pm}$ and $\Phi(\Gamma_0,t)=\Gamma_t$, $\Phi(\Omega^\pm_0,t)=\Omega^\pm_t$ for all $t\in [0,T_0]$. Moreover, $(\mathbf{\ve}^\pm,p^\pm,h)$ solves \eqref{eq:CoupledStokes1}-\eqref{eq:CoupledIV2} if and only if $(\tilde{\ve}^\pm,\tilde{p}^\pm,h)$, where $\tilde\ve^\pm(x,t)= \ve^\pm(\Phi_t(x),t)$, $\tilde{p}^\pm(x,t)= p^\pm(\Phi_t(x),t)$ for all $x\in \ol{\Omega_0^\pm}$, $t\in [0,T]$,  solves the perturbed system
\begin{align}\label{eq:CoupledStokes1p}
\partial_t\tilde{\mathbf{v}}^{\pm}-\nu^\pm\Delta\tilde\ve^{\pm}+\nabla \tilde{p}^{\pm}  &=\tilde{\mathbf{f}} + \mathbf{R}_1 &  &\text{in }\Omega^{\pm}_0\times(0,T),\\\label{eq:CoupledStokes2p}
\operatorname{div}\tilde{\mathbf{v}}^{\pm} & =\tilde{g} +R_2  &  &\text{in }\Omega^{\pm}_0\times (0,T),\\\label{eq:CoupledStokes3p}
  \llbracket\tilde{\mathbf{v}}\rrbracket & =\tilde{\mathbf{a}}_{1}+ \tilde{b}_0\btau h   &  & \text{on }\Gamma_0\times (0,T),\\\label{eq:CoupledStokes5p}
    \llbracket 2\nu D\tilde{\ve}-\tilde{p} \tn{I}\rrbracket\no +\ol{\nu}\llbracket \partial_{\no} \tilde\ve\rrbracket -\sigma\Delta^{\Gamma_0} h\circ S_0^0 &= \no + \tilde{\mathbf{b}}_1\partial_s h\circ S_0^0\nonumber\\
  &\quad +\tilde{\mathbf{b}}_2h\circ S_0^0+ \tilde{\mathbf{a}}_2+\mathbf{R}_3  &  & \text{on }\Gamma_0, t\in (0,T),\\\label{eq:CoupledStokes6p}
\partial_t h+a_1\partial_s h +\tfrac12(\ve^{+}_\no+\ve^{-}_\no)\circ X_0^0 &=w- a_0h + R_4 &  & \text{on }\T^1\times(0,T),
\end{align}
together with $\tilde\ve|_{\partial\Omega}=0$, $\tilde\ve|_{t=0}=0, h|_{t=0}=0$,
where $S_0^0 (x,t)= S_0(x,0)$ for all $(x,t)\in \Omega\times (0,T)$, $X_0^0(s)= X_0(s,0)$ for all $s\in\T^1$, and
\begin{alignat*}{3}
  \tilde{\mathbf{f}}(x,t)&=\mathbf{f}(\Phi_t(x),t), &\quad  \tilde{g}(x,t)&=g(\Phi_t(x),t) &\qquad&\text{for all }(x,t)\in \Omega_0^\pm\times (0,T),\\
  \tilde{\mathbf{a}}_j(x,t)&=\mathbf{a}_1(\Phi_t(x),t), &\quad\tilde{\mathbf{b}}_j(x,t)&=\mathbf{b}_j(\Phi_t(x),t)  &\qquad&\text{for all }(x,t)\in \Gamma_0\times (0,T),j=1,2,\\
     \tilde{b}_0(x,t)&=b_0(\Phi_t(x),t),&&   &\qquad&\text{for all }(x,t)\in \Gamma_0\times (0,T).
\end{alignat*}
Here $(\mathbf{R}_1, R_2, \mathbf{R}_3,0, \mathbf{R}_3, R_4)\in \mathbb{F}(T):= \mathbb{F}_1(T)\times \ldots\times \mathbb{F}_5(T)$ depends linearly on $(\tilde\ve, \tilde{p},h)\in \mathbb{E}(T):= \mathbb{E}_1(T)\times \mathbb{E}_2(T) \times \mathbb{E}_3(T)$, where
\begin{align*}
  \mathbb{E}_1(T)&:=  L^q(0,T; W^2_q(\Omega\setminus \Gamma_0)\cap W^1_{q,0}(\Omega))^2\cap {}_0W^1_q(0,T; L^q(\Omega))^2, \\
  \mathbb{E}_2(T)&:= L^q(0,T; W^1_{q,(0)}(\Omega\setminus \Gamma_0))\cap {}_0W^1_q(0,T; \dot{W}^{-1}_{q,(0)}(\Omega)) \\
  \mathbb{E}_3(T)&:=L^q(0,T;W_q^{3-\frac1q}(\T^1))\cap {}_0 W^1_q(0,T;W_q^{2-\frac1q}(\T^1)),\\
  \mathbb{F}_1(T)&:= L^q(\Omega\times (0,T))^2, \quad \mathbb{F}_2(T):= L^q(0,T;W^1_q(\Omega\setminus \Gamma_0)), \\
  \mathbb{F}_3(T)&:=L^q(0,T;W_q^{2-\frac1q}(\Gamma_0))^2\cap {}_0W_q^{\frac12-\frac1{2q}}(0,T;L^q(\Gamma_0))^2, \quad\\
    \mathbb{F}_4(T)&:=L^q(0,T;W_q^{1-\frac1q}(\Gamma_0))^2\cap W_q^{\frac12-\frac1{2q}}(0,T;L^q(\Gamma_0))^2, \quad
    \mathbb{F}_5(T):=  L^q(0,T;W_q^{2-\frac1q}(\T^1)), 
\end{align*}
normed in the standard way,
and, for a Banach space $X$ and $s>1-\frac1q$,
\begin{equation*}
  {}_0 W^s_q(0,T;X) :=\{u \in W^s_q(0,T;X): u(0)=0\}. 
\end{equation*}
Moreover, since $\Phi_t\to_{t\to 0} \operatorname{id}_{\overline{\Omega}}$ in $C^k(\overline{\Omega})$ for every $k\in\N$, we have that
\begin{equation*}
  \|(\mathbf{R}_1, R_2, 0,  \mathbf{R}_3, \mathbf{R}_4)\|_{\mathbb{F}(T)}\leq C(T) \|(\tilde\ve, \tilde{p},h)\|_{\mathbb{E}(T)},
\end{equation*}
for some $C(T)\to_{T\to 0} 0$. Furthermore using e.g.\ $\mathbb{E}_3(T)\hookrightarrow C([0,T];W_q^{3-\frac3q}(\T^1))$ one can show that
\begin{align*}
  \|(0, 0, \tilde{b}_0\btau h, \tilde{\mathbf{b}}_1\partial_s h\circ S_0^0+\tilde{\mathbf{b}}_2h\circ S_0^0, a_0h)\|_{\mathbb{F}(T)}\leq CT^\alpha \|(\tilde\ve, \tilde{p},h)\|_{\mathbb{E}(T)}
\end{align*}
for some $\alpha>0$ and $C$ independent of $T\in (0,T_0]$. Hence, since the set of all invertible linear operators is open, there is some $T_1>0$ such that \eqref{eq:CoupledStokes1p}-\eqref{eq:CoupledStokes6p} possesses a unique solution in $(\tilde\ve, \tilde{p},h)\in \mathbb{E}(T)$ provided $T\in (0,T_1]$.
Transforming $(\tilde{\ve},\tilde{p})$ to $\Omega^\pm$ with the aid of $\Phi_t^{-1}$ yields the statement in this case.
\medskip

    \noindent
    \textbf{Step 3: Existence for general $T>0$:} Since the system is linear, the existence time $T_1>0$ in the second step is independent of the norms of the data. Moreover, as in the second step we obtain that for any $t_0\in [0,T)$ there is some $T_1(t_0)>t_0$ such that the system has a unique solution for $t\in (t_0,T_1(t_0))$ for a given initial value $\ve|_{t=t_0}=\tilde{\ve}_0$ at $t=t_0$. Because of the compactness of $[0,T]$ and uniqueness of the solutions, we can concatenate these solutions and obtain a solution on $[0,T]$. 
  \end{proof}
  \begin{rem}\label{rem:Solvability}
    With the aid of Theorem~\ref{solvingkorder} one can obtain that for all smooth $\mathbf{f}, g, a_{1,2}, \mathbf{b}, \mathbf{w}, \mathbf{a}$ (without precribed initial values $\ve_0$, $h_0$) a smooth solution of \eqref{eq:CoupledStokes1}-\eqref{eq:CoupledStokesLast}. To this end one extends the smooth data $\mathbf{f}$, $g$, $w$, $\mathbf{a}_k$, $k=1,2$, $\mathbf{a}$, and $\Omega^\pm_t, \Gamma_t$ on a time interval $[-1,T_0]$ in a smooth manner such that these functions vanish in $[-1,\frac12]$. Then one can apply Theorem~\ref{thm:Approx1} to obtain a solution $(\ve^\pm, p^\pm,h)$ of \eqref{eq:CoupledStokes1}-\eqref{eq:CoupledStokesLast} on a time intervall $(-1,T_0)$ instead of $(0,T_0)$ and with initial values $\ve_0 =0$, $h_0=0$. Then one can apply the parameter-trick in space and time (cf.\ e.g.\ \cite[Section~9.4]{PruessSimonettMovingInterfaces} to obtain that $\ve^\pm,p$ are smooth in $\bigcup_{t\in (-1,T_0]}\Omega^\pm_t\times \{t\}$ and $h$ is smooth in $\bigcup_{t\in (-1,T_0]}\T^1\times \{t\}$. Restriction to $[0,T_0]$ in time yields the existence of a smooth solution to \eqref{eq:CoupledStokes1}-\eqref{eq:CoupledStokesLast}, which satisfy \eqref{eq:CoupledIV1}-\eqref{eq:CoupledIV2} for some $\ve_0^\pm:= \ve^\pm|_{t=0}$, $h_0 := h|_{t=0}$.
  \end{rem}

  \section*{Acknowledgments}
M.~Fei was partially supported by NSF of China under Grant No.12271004 and Anhui Provincial Talent Funding Project under Grant No.gxbjZD2022009.
Moreover, M.~Moser has received funding from the European Research Council (ERC) under the European Union's Horizon 2020 research and innovation programme (grant agreement No 948819).


\def\ocirc#1{\ifmmode\setbox0=\hbox{$#1$}\dimen0=\ht0 \advance\dimen0
  by1pt\rlap{\hbox to\wd0{\hss\raise\dimen0
  \hbox{\hskip.2em$\scriptscriptstyle\circ$}\hss}}#1\else {\accent"17 #1}\fi}

\bigskip

\noindent
{\it
  (H. Abels) Fakult\"at f\"ur Mathematik,
  Universit\"at Regensburg,
  93040 Regensburg,
  Germany}\\
{\it E-mail address: {\sf helmut.abels@mathematik.uni-regensburg.de} }\\[1ex]
{\it
(M. Fei) School of  Mathematics and Statistics, Anhui Normal University, Wuhu 241002, China}\\
{\it E-mail address: {\sf mwfei@ahnu.edu.cn} }\\[1ex]
{\it
(M. Moser) Institute of Science and Technology Austria, Am Campus 1, AT-3400 Klosterneuburg, Austria}\\
{\it E-mail address: {\sf maximilian.moser@ist.ac.at} }\\[1ex]

\end{document}